\documentclass[12pt,letterpaper]{article}
\usepackage{renstyles}
\renewcommand{\RED}[1]{{#1}}

\title{A Model-Consistent Data-Driven Computational Strategy for PDE Joint Inversion Problems}
\author{
Kui Ren\thanks{
		Department of Applied Physics and Applied Mathematics, Columbia University, New York, NY 10027;
		\href{mailto:kr2002@columbia.edu}{kr2002@columbia.edu}
	}
\and
Lu Zhang\thanks{
		Department of Computational Applied Mathematics and Operations Research, and Ken Kennedy Institute, Rice University, Houston, TX 77005;
		\href{mailto:lz82@rice.edu}{lz82@rice.edu}
	}
}

\date{}

\begin{document}
%%%%%%%%%%%%%%%%%%%%%%%%%%%%%%%%%%%%%%%%%%%%%%%%%%%%%%%%%%%%%%%%%%
%%%%%%%BEGIN DOCUMENT %%%%BEGIN DOCUMENT %%%%BEGIN DOCUMENT%%%%%%%
%%%%%%%%%%%%%%%%%%%%%%%%%%%%%%%%%%%%%%%%%%%%%%%%%%%%%%%%%%%%%%%%%%

\maketitle

%\tableofcontents

%%%%%%%%%%%%%%%%%%%%%%%%%%%
%%%%%%%%ABSTRACT%%%%%%%%%%%
%%%%%%%%%%%%%%%%%%%%%%%%%%%
\begin{abstract}

The task of simultaneously reconstructing multiple physical coefficients in partial differential equations (PDEs) from observed data is ubiquitous in applications. In this work, we propose an integrated data-driven and model-based iterative reconstruction framework for such joint inversion problems where additional data on the unknown coefficients are supplemented for better reconstructions. Our method couples the supplementary data with the PDE model to make the data-driven modeling process consistent with the model-based reconstruction procedure. \RED{This coupling strategy allows us to characterize the impact of learning uncertainty on the joint inversion results for two typical inverse problems.} Numerical evidence is provided to demonstrate the feasibility of using data-driven models to improve the joint inversion of multiple coefficients in PDEs.

\end{abstract}
  
%%%%%%%%%%%%%%%%%%%%%%%%%%%
%%%%%%%%%KEYWORDS%%%%%%%%%%
%%%%%%%%%%%%%%%%%%%%%%%%%%%

%%%%%%%%%%%%%%%%%%%%%%%%%%
%%%   AMS or PACS   %%%%%%
%%%%%%%%%%%%%%%%%%%%%%%%%%

%%%%%%%%%%%%%%%%%%%%%%%%%%%%%%%%%%%%%%%%%%%%%%%%%%%%%%%%%%%%%%%%%%
%%%%%% BEGINNING TEXT %%% BEGINNING TEXT %%% BEGINNING TEXT %%%%%%
%%%%%%%%%%%%%%%%%%%%%%%%%%%%%%%%%%%%%%%%%%%%%%%%%%%%%%%%%%%%%%%%%%

%%%%%%%%%%%%%%%%%%%%%%%%%%%%%%%%%%%%%%%%%%%%%%%%%%%%%%%%%%%%%%%%%%
%%%%%%%%%%%%%%%%%%%%%%%%%%%%%%%%%%%%%%%%%%%%%%%%%%%%%%%%%%%%%%%%%%
\section{Introduction}
\label{SEC:Intro}
%%%%%%%%%%%%%%%%%%%%%%%%%%%%%%%%%%%%%%%%%%%%%%%%%%%%%%%%%%%%%%%%%%
%%%%%%%%%%%%%%%%%%%%%%%%%%%%%%%%%%%%%%%%%%%%%%%%%%%%%%%%%%%%%%%%%%

Let $X$ and $Y$ be two spaces of functions and $\cA_h: X\times X\mapsto Y$ an operator between them, depending on the parameter $h\in H$ ($H$ being some given space of parameters). We consider the following abstract inverse problem of reconstructing $f\in X$ and $g\in X$ from measured data encoded in $u_h\in Y$:
\begin{equation}\label{EQ:Model Gen}
	\cA_h(f, g)=u_h\,.
\end{equation}
We will make the forward operator $\cA_h$ more explicit when discussing concrete inverse problems. For the moment, we assume that the operators $\cA_h(f, \cdot)$ and $\cA_h(\cdot, g)$ are both uniquely invertible at given $f$ and $g$ in $X$.

\RED{This type of inverse problem appears in many applications where $f$ and $g$ are physical quantities we are interested in imaging from measured data $u_h$. For instance, in diffuse optical tomography~\cite{Arridge-IP99}, $f$ and $g$ could be, respectively, the optical absorption and scattering coefficients of the underlying medium to be probed. In photoacoustic imaging~\cite{MaAn-IP17,DiReVa-IP15,StUh-IPI13,KiSc-SIAM13,BaRe-IP11,ReGaZh-SIAM13}, $f$ and $g$ could be respectively the ultrasound speed and the optical absorption coefficient of the medium. In geophysical applications such as seismic imaging~\cite{CrStGh-IP18,GaMe-JGR04,HaOl-IP97}, $f$ and $g$ could be, respectively, the bulk modulus and the density field of the Earth. We refer interested readers to ~\cite{ArLeSeKoVa-IEEE21,BeJeYa-AA08,CaDu-SIAM80,CaPeGaSc-GJI22,LoRu-IP93,RoCr-IP12,HaChCh-RS86,ReDeBaMiCoWaNu-IEEE12,Zhdanov-Book23} and references therein for more examples of such multiple-coefficient inverse problems. In Section~\ref{SEC:Examples}, we will present two concrete examples of such inverse problems.}

\RED{In many applications, we face the issue that when $f$ and $g$ are treated as independent variables, the data available in $\{u_h\}_{h\in H}$ are not sufficient to reconstruct $f$ and $g$ simultaneously. This is the situation in diffuse optical tomography where boundary current data are not enough to uniquely determine both the absorption and scattering coefficients of the underlying medium~\cite{Arridge-IP99}. The inversion process could be unstable even when the available data are sufficient to uniquely determine the two coefficients. Therefore, one must introduce additional information to improve joint inversions.}

There are roughly two lines of ideas in the literature for joint inversion. The first type of method introduces additional measurement data to enrich the information content of the data $\{u_h\}_{h\in H}$. The other measurement could come from either the same physics in the original problem~\cite{ReGaZh-SIAM13} or a different but related physical process that is coupled to the original problem~\cite{AbGaHaLi-IP12, DoBrMe-Geophysics22,FeReLiZh-Geophysics17, KnHoKoOtBrSo-IEEE17} (often called model fusion). The second type of method uses \emph{a priori} information one has on the unknown $f$ and $g$ to improve the reconstruction. This includes, for instance, structural methods where one assumes that $f$ and $g$ share the same structural features that one can impose in the reconstruction process through regularization; see~\cite{CrStGh-IP18, GaMe-GRL03} and references therein.

\RED{In some biological applications, $f$ and $g$ are physical coefficients related to each other either because they describe physical properties that are connected or because their values change similarly in response to an environmental change. For instance, in diffuse optical tomography, the optical absorption and scattering coefficients of the tissue to be probed change in correlation when tumorous cells are developed. In such a case, one can potentially utilize the relation between the coefficients to improve their computational reconstruction. In the rest of this work, we consider the situation where we have a large amount of past data on $f$ and $g$ that allows us to learn some simplified relation, denoted by $\cN_\theta$, parameterized by $\theta\in\Theta$, between features of $f$ and $g$. We then use this learned relation to help the simultaneous reconstruction of $f$ and $g$. Let $\cF(f)$ represent the part of the information on $f$ that would like to link with the same information on $g$. Then the process can be written mathematically as
\begin{equation}\label{EQ:Relation Gen}
	\cF(g) = \cN_\theta\big(\cF(f)\big)\,.
\end{equation}
The operator $\cF$ is decided by the given \emph{a priori} knowledge of the physics of the problem. Examples of $\cF$ for practical applications include, for instance, the gradient operator for situations where only the gradient of the $f$ and $g$ are expected to be related~\cite{CrStGh-IP18}, and the truncated Fourier transform for applications where only some (lower) Fourier modes between $f$ and $g$ are expected to be related. In the simplest case where $\cF$ is invertible, $\cF^{-1}\circ \cN_{\theta}\circ \cF$ gives the relation between the original coefficients $f$ and $g$.}

\RED{Using machine learning methods to solve joint inversion problems has been extensively studied in recent years; see, for instance, ~\cite{BaWaWaYuLuStYa-RS24,CoLiRoSaTu-SEG20,DoBrMe-Geophysics22,FaWaZhLiMeZh-IEEE25,HeBaTaTa-AWR20,GuZhWiLiLiElYaXuAb-Geophysics25,HuWeWuSuChHuCh-Geophysics23,HuWuHaZh-IEEE19,HuYuWaWa-BSSA24,JiDoZhZeLi-IEEE23,LiPaZhHeZhReTa-IEEE24,MaDeLiGuZhLi-SG25,MoKiBe-GJI24,ReLiLiLiJi-IEEE24,SoMaSpMa-GJI25,WaJiaLiLu-TIMAGE23,YaLiZhDoZh-IEEE23} and references therein. In most references cited, machine learning methods are used as an inversion method, meaning that the data are already fixed and one is only interested in using machine learning methods to improve or replace a classical model-based inversion scheme. In the current work, we use machine learning as a way to couple two different types of data, that is, the offline historical coefficient data pairs and the online measurement used to infer new coefficient pairs. Through an offline-online coupling scheme, we improve the overall joint reconstruction procedure. Moreover, we will show that our coupling scheme allows us to characterize how the error in the offline learning process is propagated into the online reconstruction of the coefficients with new data.}

The rest of the paper is organized as follows. We first present the general computational strategy in~\Cref{SEC:Num}. We then discuss the strategy and present numerical simulations in two examples of joint inversion problems, one in quantitative photoacoustic imaging and the other in full waveform inversion, in~\Cref{SEC:Examples} and~\Cref{SEC:Simulation}. In~\Cref{SEC:Uncertainty}, we briefly discuss the impact of learning accuracy on the joint inversion results. Concluding remarks are offered in~\Cref{SEC:Concl}.

%%%%%%%%%%%%%%%%%%%%%%%%%%%%%%%%%%%%%%%%%%%%%%%%%%%%%%%%%%%%%%%%%%
%%%%%%%%%%%%%%%%%%%%%%%%%%%%%%%%%%%%%%%%%%%%%%%%%%%%%%%%%%%%%%%%%%
\section{Computational framework}
\label{SEC:Num}
%%%%%%%%%%%%%%%%%%%%%%%%%%%%%%%%%%%%%%%%%%%%%%%%%%%%%%%%%%%%%%%%%%
%%%%%%%%%%%%%%%%%%%%%%%%%%%%%%%%%%%%%%%%%%%%%%%%%%%%%%%%%%%%%%%%%%

The joint inversion problem we introduced can be summarized as follows. Given a set of $N$ historical coefficient pairs
\begin{equation}\label{EQ:Data Hist}
	\{f_k, g_k\}_{k=1}^N\,,
\end{equation}
and a set of newly measured data 
\begin{equation}\label{EQ:Data Meas}
	\{u_h\}_{h\in H}
\end{equation}
corresponding to the mathematical model~\eqref{EQ:Model Gen}, we are interested in reconstructing the functions $f$ and $g$ corresponding to the measurements~\eqref{EQ:Data Meas}.

\RED{The inversion process with data~\eqref{EQ:Data Hist} and~\eqref{EQ:Data Meas} is a two-step procedure. In the first step, we learn from data~\eqref{EQ:Data Hist} a relation $\cN_\theta$ between $\cF(f)$ and $\cF(g)$ as given in~\eqref{EQ:Relation Gen}. The choice of the features to be learned, that is, the operator $\cF$, is of critical importance in this step. More discussions will be provided later in this section. In the second step, we use this relation and data~\eqref{EQ:Data Meas} to reconstruct the coefficients $f$ and $g$. Our idea is philosophically similar to the idea of learning a regularization functional from historical data, such as that in~\cite {GiOnWi-IEEE19}, but is fundamentally different in terms of the details of the implementation, including how the learning process is done and how the learned model is utilized in the joint inversion stage.}

%%%%%%%%%%%%%%%%%%%%%%%%%%%%%%%%%%%%%%%%%%%%%%%%%%%%%%%%%%%%%%%%%%
\subsection{Model-consistent learning from historical data}
%%%%%%%%%%%%%%%%%%%%%%%%%%%%%%%%%%%%%%%%%%%%%%%%%%%%%%%%%%%%%%%%%%

The first stage of the inversion is a modeling process where we use existing historical data~\eqref{EQ:Data Hist} to learn the relation $\cN_\theta$. This process is done independently of the reconstruction process. This is similar to the works where one learns a regularization functional based on the available data~\cite{GiOnWi-IEEE19}. However, given that the relation $\cN_\theta$ we learned will be used in model-based inversion with data~\eqref{EQ:Data Meas} in the next step, it is advantageous to make the modeling process consistent with the inversion procedure. Therefore, we perform the modeling step taking the physical model~\eqref{EQ:Model Gen} as a constraint in addition to existing data~\eqref{EQ:Data Hist}. More precisely, we first pass the data~\eqref{EQ:Data Hist} through the mathematical model~\eqref{EQ:Model Gen} to obtain the new training dataset
\begin{equation}\label{EQ:Data Hist New}
	\{f_k, g_k, \{u_{h,k}\}_{h\in H}\}_{k=1}^N\,,
\end{equation}
where the data point $u_{h,k}$ is obtained by evaluating the forward model~\eqref{EQ:Model Gen} with coefficient pair $(f_k, g_k)$ and parameter $h$, that is, $u_{h,k}=\cA_h(f_k, g_k)$. We then use this new dataset as the additional training dataset to impose consistency between the learned model $\cN_\theta$ and the forward model~\eqref{EQ:Model Gen}. The learning process is realized by solving the following minimization problem:
\begin{equation}\label{EQ:Model-Consistent Training}
    \begin{array}{c}
	\wh \theta =\displaystyle \argmin_{\theta\in\Theta}\cL(\theta),\\[1ex] \cL(\theta):=\dfrac{1}{2N|H|}\dsum_{h\in H}\dsum_{k=1}^N\|\cA_h(f_k, \wt \cN_\theta(f_k))-u_{h,k}\|_Y^2\,,
    \end{array}
\end{equation}
where $\wt \cN_\theta:=\cF^{-1}\circ\cN_\theta\circ \cF$ is the relation between the coefficients $f$ and $g$ in the physical space $X$, and $|H|$ is the cardinality of the set $H$.

To reduce the computational cost (due to the requirement of evaluating $\cA_h(f_k, \wt\cN_\theta(f_k))$), we initialize the model-consistent training from the result of the purely data-driven training. That is, we take the initial value $\wh \theta_0$ for~\eqref{EQ:Model-Consistent Training} as
\begin{equation}\label{EQ:Data-Driven Training}
    %\begin{array}{c}
	\wh \theta_0 =\displaystyle \argmin_{\theta\in\Theta}\wt \cL(\theta),\qquad %\\[1ex]
 \wt \cL(\theta):=\dfrac{1}{2N}\sum_{k=1}^N \|\cF(g_k)-\cN_\theta(\cF(f_k))\|_{\wt X}^2\,,
    %\end{array}
\end{equation}
where $\wt X:=\cF[X]$ \RED{denotes} space of coefficients $g$ in the $\cF$ domain. This model-consistent learning process is summarized in Algorithm~\ref{ALG:Learning}. % \RED{(LZ: Do we need to explain $\wt X$ here?)}

%%%%%%%%%%%%%%%%%%%%%%%%%%%%%%%%%%%%%%%%%%%%%%
\begin{algorithm}[htb!]
\caption{Model-consistent Learning}
\label{ALG:Learning}
\begin{algorithmic}[1]
\State Learn $\wh \theta_0$ according to~\eqref{EQ:Data-Driven Training} from data $\{f_k, g_k\}_{k=1}^N$
\State Generate model-consistent data $\{f_k, \{u_{h,k}\}_{h\in H}\}_{k=1}^N$ from $\{f_k, g_k\}_{k=1}^N$
\State Learn $\wh \theta$ according to~\eqref{EQ:Model-Consistent Training} using $\wh \theta_0$ as the initial guess 
\end{algorithmic}
\end{algorithm}
%%%%%%%%%%%%%%%%%%%%%%%%%%%%%%%%%%%%%%%%%%%%%%

\paragraph{Parameterization of unknown coefficients.} Neither the data-driven step~\eqref{EQ:Data-Driven Training} nor the model-consistent step~\eqref{EQ:Model-Consistent Training} of the learning problem is computationally easy as the problems are generally non-convex (since the relation $\cN_\theta$ is most likely nonlinear), and the dimension of the problem can be very high. It is critical to select efficient representations for the functions $f$ and $g$ and the relation between them, that is, the operator $\cN_\theta$. Our objective in this proof-of-concept study is not to develop such representations for a given problem but rather to show that with reasonable parameterization of the objects involved, we can have a feasible learning procedure that can be useful for joint reconstructions in the next stage of the computational procedure.

We consider in the rest of this work parameterization of the functions $f$ and $g$ by their Fourier coefficients. More precisely, we take $\cF(f)$ to be the $K$-term truncated generalized Fourier series of $f$ in the sense that
\begin{equation}\label{EQ:Fourier Model}
    \cF(f)=\sum_{k=0}^K \wh f_k \varphi_k(\bx)
\end{equation}
where $\varphi_k(\bx)$ is the eigenfunctions of the Laplace operator with Neumann boundary condition, that is,
\begin{equation}\label{EQ:Eigenfunction}
	-\Delta \varphi_k=\lambda_k \varphi_k(\bx),\ \ \mbox{in}\ \ \Omega, \qquad \bn\cdot\nabla \varphi_k(\bx) = 0, \ \ \mbox{on}\ \partial\Omega.
\end{equation}
\RED{When $\Omega$ is smooth, $\{\varphi_{k}\}_{k\ge 0}$ forms an orthonormal basis in $L^2(\Omega)$; see for instance~\cite{Jost-Book12}. We will thus identify any function $f$ with its corresponding truncated generalized Fourier coefficient vector $\wh \bff=(\wh f_0, \cdots, \wh f_K)$ and attempt to learn a nonlinear relation between the generalized Fourier coefficients of $f$ and $g$:
\[
    \cN_\theta:\begin{array}{lcl}
    \wh \bff &\longmapsto& \wh \bg\\[1ex]
    \bbR^K &\longmapsto& \bbR^K\,,
    \end{array}
\]
where $K$ is relatively small (as we are not attempting to learn precise relations but only reasonably good approximations to help the joint reconstructions in the next step of this data-driven joint inversion procedure), to be specified later in the numerical experiments. Obviously, one does not want to take a $K$ that is too large, as that would make the dimension of the learning problem too big to be practical.} 

\paragraph{Parameterization of the unknown map.} We implement two different representations of the map $\cN_\theta$. The first representation of $\cN_\theta$ is through multi-dimensional polynomials. Let $\balpha=(\alpha_0, \cdots, \alpha_K)$ be a multi-index, $|\balpha|:=\sum_{k=0}^K \alpha_k$, and $P_{\balpha}(\wh \bff)$ be the monomial of order $\balpha$ formed from $\wh \bff$, that is,
\[
    P_{\balpha}(\wh \bff):=\prod_{k=0}^K \wh f_k^{\alpha_k}\,.
\]
Then the specific polynomial model we take can be written in the form:
\begin{equation}\label{EQ:Poly}
    \wh \bg= \cN_\theta(\wh \bff):=\Big(\sum_{|\balpha|\le n}\theta_{1,{\balpha}} P_{\balpha}(\wh\bff), \cdots, \sum_{|\balpha|\le n}\theta_{j,{\balpha}} P_{\balpha}(\wh\bff), \cdots, \sum_{|\balpha|\le n}\theta_{K, {\balpha}} P_{\balpha}(\wh\bff)\Big)
\end{equation}
where $n$ is the order of the polynomial, and $\{\theta_{j,\balpha}\}_{j=0}^K$ is the set of coefficients of the polynomials. In the numerical simulations, we are interested in limiting the size of the parameter space of the problem. To limit the size of parameter space, we will impose sparsity constraints on the coefficient tensor, as the precise relation between the coefficients is not what we are interested in representing. The details are provided in~\Cref{SUBSEC:Polynomial}.

The second representation of $\cN_\theta$ that we implemented, as a feasibility study and also as a benchmark to the first approach, is a fully connected neural network. We adopt a standard autoencoder network architecture. The learning network contains three major substructures: an encoder network $E_\theta$, a decoder network $D_\theta$, and an additional predictor network $P_\theta$. The encoder-decoder substructure is trained to regenerate the input data, while the predictor reads the latent variable to predict the velocity field. In terms of the input-output data, the network is trained such that $\wh\bff= D_\theta(E_\theta(\wh \bff))$ and $\wh \bg= P_\theta(E_\theta(\wh \bff))$. Once the training is performed, the learned model is $\cN_{\wh \theta}: = P_{\wh \theta} \circ E_{\wh \theta}$. We provided more details on the implementation of the network structure as well as the training process in~\Cref{SUBSEC:Network}.

%%%%%%%%%%%%%%%%%%%%%%%%%%%%%%%%%%%%%%%%%%%%%%%%%%%%%%%%%%%%%%%%%%
\subsection{Joint reconstruction from new data}\label{sec:joint reconstruction}
%%%%%%%%%%%%%%%%%%%%%%%%%%%%%%%%%%%%%%%%%%%%%%%%%%%%%%%%%%%%%%%%%%

We can now perform joint inversion of $f$ and $g$ from given noisy data $\{u_h^\delta\}$, \RED{where} $\delta$ indicates the level of random noise in the data, with the learned relation $\wt\cN_{\wh\theta} = \cF^{-1}\circ\cN_{\wh\theta}\circ \cF$. \RED{We perform the reconstruction with the classical model-based iterative inversion framework, where we pursue the reconstruction by minimizing the mismatch between measured data and the model predictions. Instead of using the learned relation between $f$ and $g$ as a hard constraint in the inversion process, we use it as a guide for the reconstruction. We start with the initial solution where we force the relation $\wt\cN_{\wh \theta}$ exactly in the reconstruction. That is, we reconstruct $(\wh f_0, \wh g_0:=\wt\cN_{\wh \theta}(\wh f_0))$ with $\wh f_0$ given by
\begin{equation}\label{EQ:Min 0}
	\wh f_0 = \argmin_{f\in X}\Phi_0(f):=\dfrac{1}{2|H|}\sum_{h=1}^{|H|} \|\cA_h(f, \wt\cN_{\wh\theta}(f))-u_h^\delta\|_{Y}^2+\dfrac{\beta}{2}\|f\|_{X}^2\,.
\end{equation}
We then reconstruct a sequence of $(\wh f_j, \wh g_j)$ by solving 
\begin{equation}\label{EQ:Min j}
	(\wh f_j, \wh g_j)=\argmin_{(f, g) \in X\times X}\Phi_j(f, g)\,,
\end{equation}
where
\begin{equation}
	\Phi_j(f, g):=\dfrac{1}{2|H|}\sum_{h=1}^{|H|} \|\cA_h(f, g)-u_h^\delta\|_{Y}^2+\dfrac{\eta_j}{2}\|g-\wt\cN_{\wh \theta}(f)\|_{X}^2+\dfrac{\beta}{2}\|(f, g)\|_{X}^2\,.
\end{equation}
We summarize the process in Algorithm~\ref{ALG:JointInv}.}
%%%%%%%%%%%%%%%%%%%%%%%%%%%%%%%%%%%%%%%%%%%%%%
\begin{algorithm}
\caption{Joint Inversion with Learned Data Model}
\label{ALG:JointInv}
\begin{algorithmic}[1]
\State Evaluate $\wh f_0$ according to~\eqref{EQ:Min 0} and $\wh g_0:=\wt\cN_{\wh \theta}(\wh f_0)$
\State Set $\eta_0$;
\For{$j=1$ to $j=J$}
	
	\State Set $\eta_{j}=\eta_{j-1}/2$;
	
	\State Evaluate $(\wh f_j, \wh g_j)$ according to~\eqref{EQ:Min j} using $(\wh f_{j-1}, \wh g_{j-1})$ as the initial guess
	
\EndFor
\end{algorithmic}
\end{algorithm}
%%%%%%%%%%%%%%%%%%%%%%%%%%%%%%%%%%%%%%%%%%%%%%

There are two major reasons for us not to strictly enforce the relation $\wt\cN_{\wh\theta}$, that is, not directly taking the solution to~\eqref{EQ:Min 0} as the reconstruction result in the inversion process. First, since the data-driven modeling process is not necessarily accurate, we want to allow the reconstruction process to search around $(\wh f_0, \wt\cN_{\wh \theta}(\wh f_0))$ for coefficients that can potentially match the data better. Second, if Algorithm~\ref{ALG:JointInv} yields a solution that is dramatically different than $(\wh f_0, \wt\cN_{\wh \theta}(\wh f_0))$, it is an indication that the admissible coefficient pairs for the new data do not live in the set $(f, \wt\cN_{\wh\theta}(f))$. We, therefore, have to let the data decide what the solution is, not the prior knowledge we have from past training data.

To solve the minimization problems~\eqref{EQ:Min 0} and~\eqref{EQ:Min j}, we employed a classical BFGS quasi-Newton method. This method has been implemented for computational inversion in similar settings as ours here~\cite{ReBaHi-SIAM06}.

%%%%%%%%%%%%%%%%%%%%%%%%%%%%%%%%%%%%%%%%%%%%%%%%%%%%%%%%%%%%%%%%%%
%%%%%%%%%%%%%%%%%%%%%%%%%%%%%%%%%%%%%%%%%%%%%%%%%%%%%%%%%%%%%%%%%%
\section{Two model joint inversion problems}
\label{SEC:Examples}
%%%%%%%%%%%%%%%%%%%%%%%%%%%%%%%%%%%%%%%%%%%%%%%%%%%%%%%%%%%%%%%%%%
%%%%%%%%%%%%%%%%%%%%%%%%%%%%%%%%%%%%%%%%%%%%%%%%%%%%%%%%%%%%%%%%%%

In this section, we introduce two examples of the abstract inverse problem~\eqref{EQ:Model Gen}. We will use the two examples for numerical simulations to demonstrate the performance of the data-driven joint inversion approach we proposed in the previous section.

%%%%%%%%%%%%%%%%%%%%%%%%%%%%%%%%%%%%%%%%%%%%%%%%%%%%%%%%%%%%%%%%%%
\subsection{An inverse diffusion problem}
\label{SUBSEC:Diff}
%%%%%%%%%%%%%%%%%%%%%%%%%%%%%%%%%%%%%%%%%%%%%%%%%%%%%%%%%%%%%%%%%%

In the first case, we consider a simplified inverse problem in quantitative photoacoustic imaging~\cite{BaRe-IP11}. Let $\Omega\subseteq\bbR^d$ ($d\ge 1$) be a bounded domain representing an optical medium whose absorption and diffusion coefficients are respectively $\sigma(\bx)$ and $\gamma(\bx)$. The transport of near-infrared photons inside the medium can be described by the following diffusion equation
\begin{equation}\label{EQ:Diff}
	\begin{array}{rcll}
	-\nabla\cdot \gamma(\bx) \nabla u(\bx) + \sigma(\bx) u(\bx) &=& 0, & \mbox{in}\ \ \Omega\\[1ex]
	\bn\cdot \gamma\nabla u + \ell u(\bx) & = & S(\bx), & \mbox{on}\ \partial\Omega
	\end{array} 	 
\end{equation}
where $u(\bx)$ is the density of the photons at $\bx$ and $S(\bx)$ is the illuminating source function sending photons into the medium. The vector $\bn(\bx)$ is the unit outward normal vector of $\partial\Omega$ at $\bx$, and the parameter $\ell$ is known to be the inverse (re-scaled) extrapolation length. 

The inverse problem is to reconstruct $\sigma$ and $\gamma$ from measured internal data of the form
\begin{equation}\label{EQ:Diff Data}
	H(\bx)=\sigma(\bx) u(\bx), \quad \bx\in\bar\Omega\,.
\end{equation}
This inverse problem is a simplified model of quantitative photoacoustic tomography, a hybrid acoustic-optical imaging modality where ultrasound measurements are used to get internal data~\eqref{EQ:Diff Data}; see, for instance, ~\cite{BaRe-IP11} and references therein for more recent overviews in this direction. The same diffusion model, but with additional boundary current data, is also very useful in practical applications such as diffuse optical tomography~\cite{Arridge-IP99}. 
It is well-known that there is an equivalence between $\sigma$ and $\gamma$ in the diffusion equation that prevents the unique reconstruction of the pair $(\sigma, \gamma)$ from a given datum $H$~\cite{BaRe-IP11}. It has also been shown~\cite{BaRe-IP11,BaUh-IP10} that $\sigma$ and $\gamma$ can be uniquely reconstructed with two well-selected datasets when $\gamma$ is known on the boundary of the domain. If that is not the case, then uniqueness with two datasets can not hold~\cite{ReGaZh-SIAM13}. In the rest of this work, we consider the case where additional data on the coefficients $\sigma$ and $\gamma$ are available to allow us to learn a relation between $\sigma$ and $\gamma$ that would help the online reconstruction of the coefficients.

%%%%%%%%%%%%%%%%%%%%%%%%%%%%%%%%%%%%%%%%%%%%%%%%%%%%%%%%%%%%%%%%%%
\subsection{An inverse wave propagation problem}
\label{SUBSEC:Acous}
%%%%%%%%%%%%%%%%%%%%%%%%%%%%%%%%%%%%%%%%%%%%%%%%%%%%%%%%%%%%%%%%%%

The second example of joint inversion problems we consider here is an inverse problem for the acoustic wave equation. Let $\kappa$ and $\rho$ be, respectively, the bulk modulus and the density of the underlying medium $\Omega\subseteq\bbR^d$, and $p$ be the acoustic pressure due to a boundary source $S(t,\bx)$. Then $p$ solves the following equation:
\begin{equation}\label{EQ:Wave}
	\begin{array}{rcll}
		\dfrac{1}{\kappa}\dfrac{\partial^2 p}{\partial t^2} -\nabla\cdot\big(\dfrac{1}{\rho}\nabla p\big) &=& 0, & \mbox{in}\ \bbR_+ \times \Omega\\[2ex]
    p(0, \bx)=\dfrac{\partial p}{\partial t}(0, \bx) & = & 0, & \mbox{in}\ \Omega
	\end{array}
\end{equation}
\RED{This wave propagation model, when supplemented with appropriate boundary conditions in bounded domains, serves as the forward model for inverse problems in seismic imaging~\cite{AkBiGh-SC02,BoDrMaZa-IP18,EnFrYa-CMS16,EpAkGhBi-IP08,LiBeLePeTr-arXiv20,Pratt-Geophysics99,Symes-GP08,TrTaLi-GJI05}, medical ultrasound imaging~\cite{BaTr-JAS20,BeMoKoLa-PMB17,GuCaTaNaWa-DM20,JaLuCo-IP20,LiViPaAn-arXiv21,LuPeTrCo-arXiv21,MaPoLiWaAn-SIAM18,WiMaBoPiKl-SR20}, among many other applications.}

The inverse problem of interests~\cite{CrStGh-IP18,GaMe-JGR04,HaOl-IP97} is to reconstruct the coefficients $\kappa$ and $\rho$ of the underlying medium from measured field $p$ on the boundary of the medium:
\begin{equation}\label{EQ:Wave Data}
    H(t, \bx):=p(t, \bx)_{|(0, T]\times\partial\Omega}
\end{equation}
for a period of time $T$ that is sufficiently long.

There is extensive literature on the mathematical, computational, and practical sides of this inverse problem; see the references cited above. In most cases, the problem is simplified by assuming that $\rho$ is constant so that the velocity field $c$ of the wave, defined through $\kappa=\rho c^2$, is the effective coefficient to be reconstructed. Here, we consider the case that historical data on the pair $(\kappa, \rho)$ are available for us to learn a relation between $\kappa$ and $\rho$ to improve the online reconstruction of the pair of coefficients. This is similar to the past studies on the same problem with prior on the geometrical similarity of the coefficients~\cite{CrStGh-IP18}.

%%%%%%%%%%%%%%%%%%%%%%%%%%%%%%%%%%%%%%%%%%%%%%%%%%%%%%%%%%%%%%%%%%
%%%%%%%%%%%%%%%%%%%%%%%%%%%%%%%%%%%%%%%%%%%%%%%%%%%%%%%%%%%%%%%%%%
\section{Numerical experiments}
\label{SEC:Simulation}
%%%%%%%%%%%%%%%%%%%%%%%%%%%%%%%%%%%%%%%%%%%%%%%%%%%%%%%%%%%%%%%%%%
%%%%%%%%%%%%%%%%%%%%%%%%%%%%%%%%%%%%%%%%%%%%%%%%%%%%%%%%%%%%%%%%%%

We now present some numerical simulations to demonstrate the feasibility of the proposed approach. Due to the proof-of-concept nature of this work, we will only work on synthetic data, meaning that the training data, as well as the joint inversion data are constructed through numerical simulations instead of real physical experiments. Our main point is to demonstrate that if there is indeed a physical law $\cN$ between the coefficients, the proposed data-driven joint reconstruction strategy can be a competitive alternative in solving joint inversion problems.
\begin{figure}[!htb]
\centering
\includegraphics[width=0.37\textwidth]{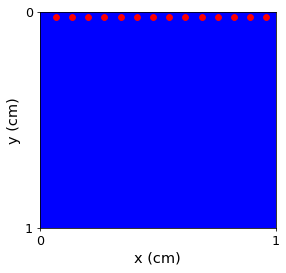}
\hskip 0.1\textwidth
\includegraphics[width=0.5\textwidth]{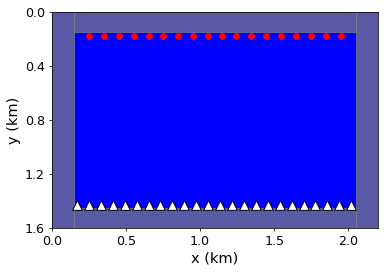}
\caption{The computational domains for the numerical simulations of the inverse diffusion (left) and inverse wave (right) problem. Sources and detectors are illustrated respectively with red dots and white triangles. The gray area surrounding the domain on the right represents a perfectly matched layer (PML).}
\label{FIG:Domain}
\end{figure}
%%%%%%%%%%%%%%%%%%%%%%%%%%%%%%%%%%%%%%%%%%%%%%%%%%%%%%%%%%%%%%%%%%
\subsection{The inverse diffusion problem}
%%%%%%%%%%%%%%%%%%%%%%%%%%%%%%%%%%%%%%%%%%%%%%%%%%%%%%%%%%%%%%%%%%

We start with the inverse diffusion problem~\eqref{EQ:Diff}--\eqref{EQ:Diff Data} introduced in~\Cref{SUBSEC:Diff}. For simplicity of presentation, we set the extrapolation length $\ell = 1$ as the particular value of the parameter has no impact on the results. We take the computational domain to be $\Omega = (0, 1) \times (0, 1)$ (see the setup in the left plot of Figure~\ref{FIG:Domain}), and discretize the problem with a $P_1$ finite element method. The mesh is constructed such that the vertices are on the Cartesian grids defined by $(x_i,y_j) = (i\Delta x,j\Delta y),i,j = 0,1,\cdots,M$ with $\Delta x = \Delta y = 1/M$.

%%%%%%%%%%%%%%%%%%%%%%%%%
\subsubsection{Experiment I: Learning with polynomial model}
%\label{SUBSEC:Num Diff}
%%%%%%%%%%%%%%%%%%%%%%%%%

In the first numerical experiment, we learn the relationship $\cN$ between $\sigma$ and $\gamma$ using the polynomial model~\eqref{EQ:Poly}. We consider the diffusion coefficient $\gamma$ and the absorption coefficient $\sigma$ as Gaussian functions. To be precise, the coefficients are in the form:
\begin{equation}\label{EQ:diffusion_absorption_gaussian}
\gamma({\bf x}) = b_1 + b_2 e^{-\frac{(x - b_4)^2 + (y - b_5)^2}{2b_3^2}},\quad\mbox{and},\quad \sigma({\bf x}) = c_1 + c_2 e^{-\frac{(x - c_4)^2 + (y - c_5)^2}{2c_3^2}},
\end{equation}
where the relation between $\{b_k\}_{k=1}^5$ and $\{c_j\}_{j=1}^5$ is given as
\begin{equation}\label{relation_gaussian_1}
\begin{array}{l}
c_1 = 0.2\dsum_{j = 1}^5 a_{1j}\cos(10 \pi b_j) + 0.1, \qquad %\\[1ex]
c_2 = \dsum_{j = 1}^5 a_{2j}\cos(20 \pi b_j) + 1,\\[1ex]
c_3  = 0.1 \dsum_{j = 1}^5 a_{3j}\cos(30 \pi b_j) + 1,\qquad %\\[1ex]
c_4 =  \frac{20}{11} \dsum_{j = 1}^5 a_{4j}\cos(2 \pi b_j) + \frac{4}{11} \\[1ex]
c_5 = \frac{25}{3} \dsum_{j = 1}^5 a_{5j}\cos(2 \pi b_j) + \frac{1}{12}\,,
\end{array}
\end{equation}
with $a_{i,j}$ ($1\le i, j\le 5$) \RED{being} the coupling coefficients.

\noindent{\bf Learning data generation.} The synthetic dataset we used for learning consists of $N = 1\times 10^4$ pairs of $(\gamma, \sigma)$ coefficient pair generated randomly according to the models in~\eqref{EQ:diffusion_absorption_gaussian}. We first generate a random collection of $\{\gamma_j\}_{j=1}^N$ by drawing $\{b_i\}_{i=1}^5$ from uniform distributions. We then draw the coupling coefficients $a_{ij}$ from the uniform distribution $\cU[0, 0.1]$. These coupling coefficients are fixed once they are generated. We then form $\sigma_j$ for each $\gamma_j$ using the relation~\eqref{relation_gaussian_1}. 
\begin{figure}[htb!]
	\centering	\includegraphics[width=0.22\textwidth,trim=1.5cm 2cm 0.5cm 1cm,clip]{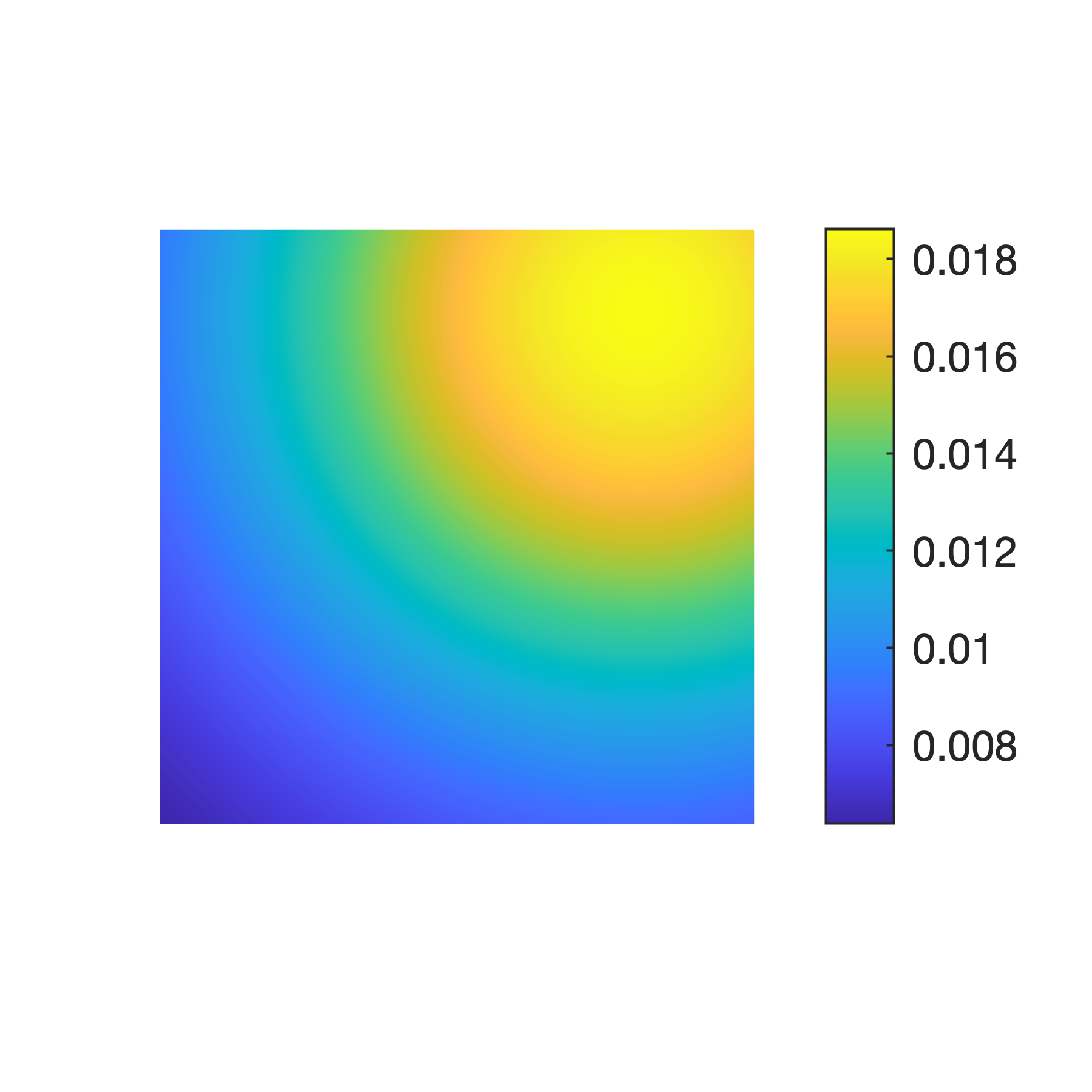}	  \includegraphics[width=0.22\textwidth,trim=1.5cm 2cm 0.5cm 1cm,clip]{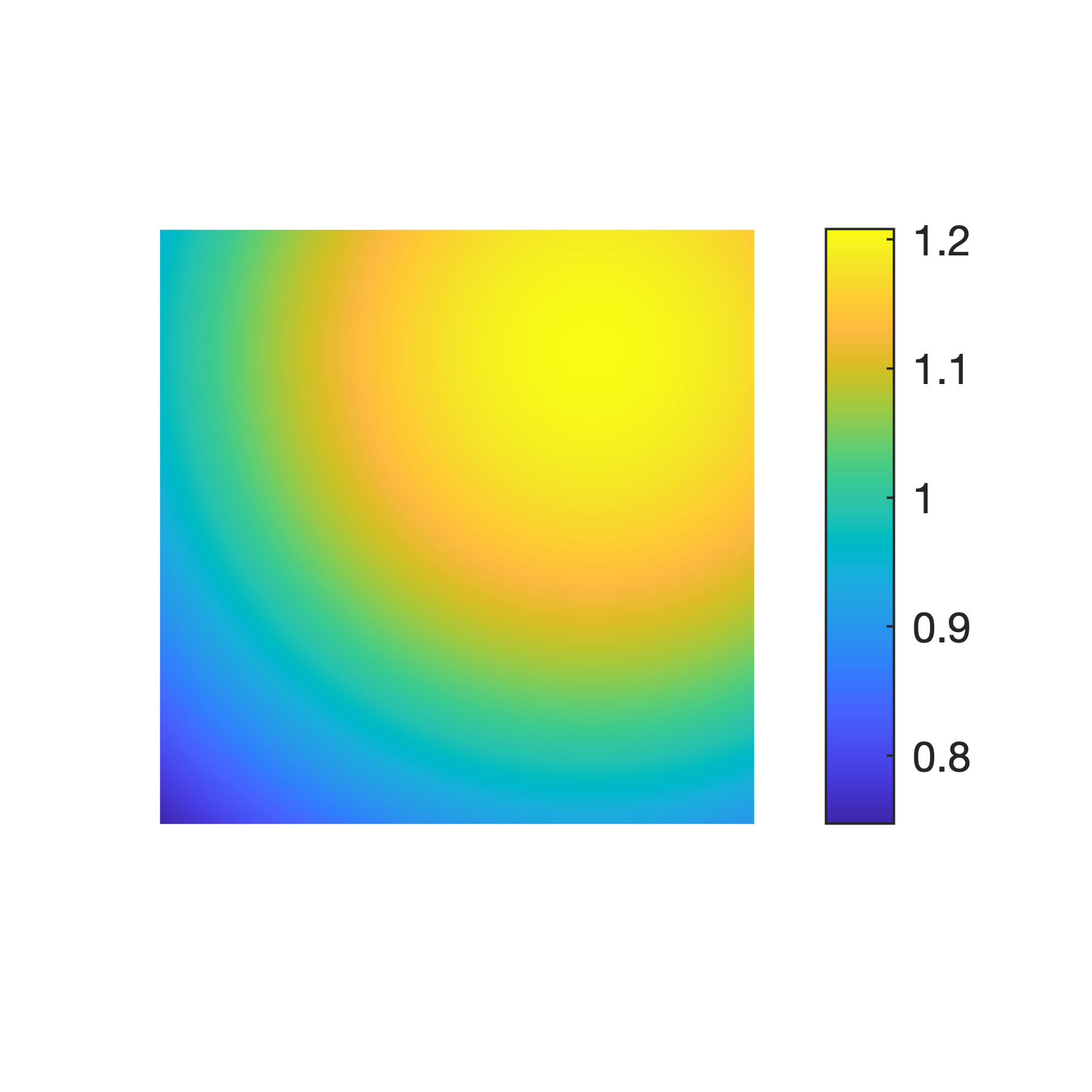}	\includegraphics[width=0.22\textwidth,trim=1.5cm 2cm 0.5cm 1cm,clip]{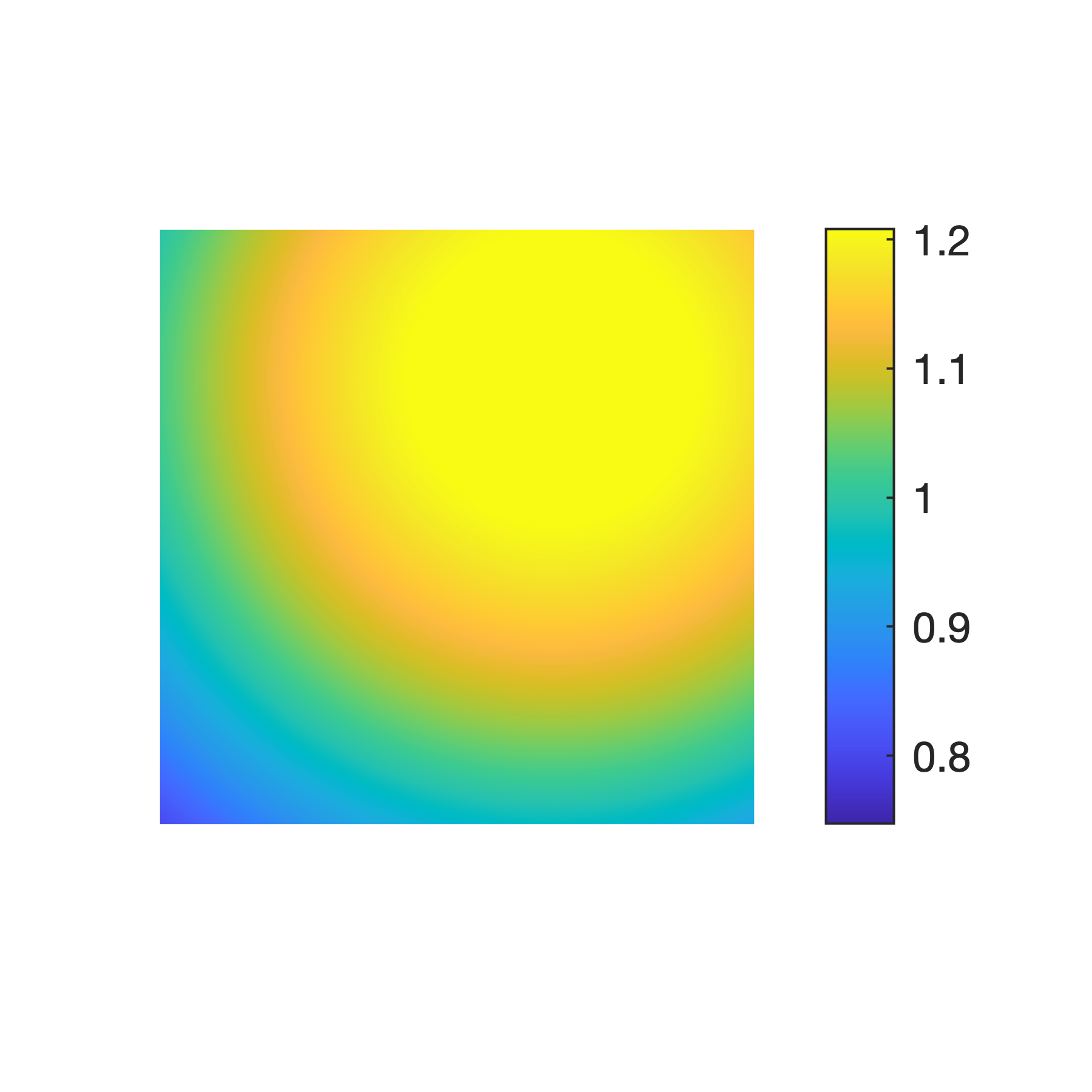}	\includegraphics[width=0.22\textwidth,trim=1.5cm 2cm 0.4cm 1cm,clip]{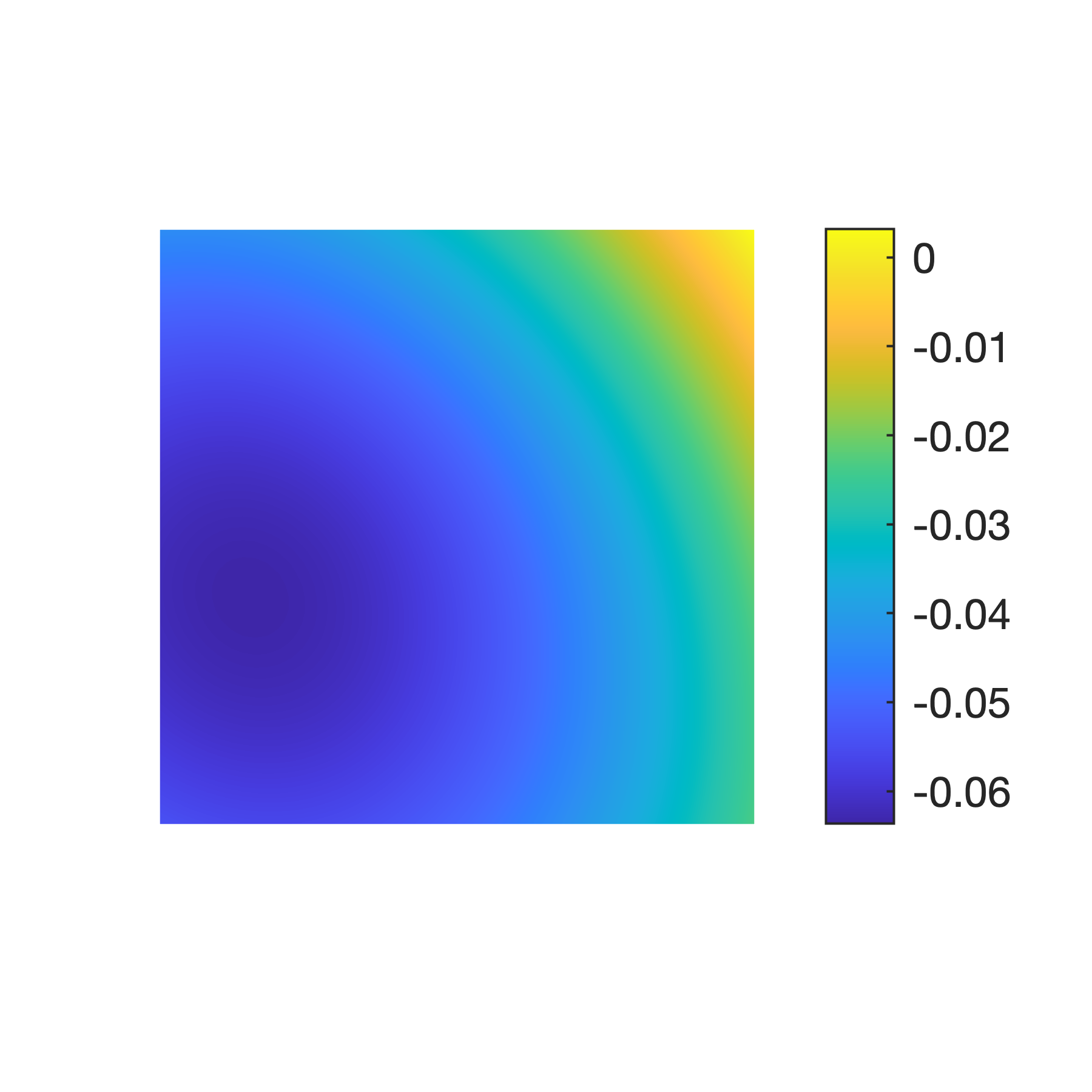}\\	\includegraphics[width=0.22\textwidth,trim=1.5cm 1.5cm 0.5cm 1cm,clip]{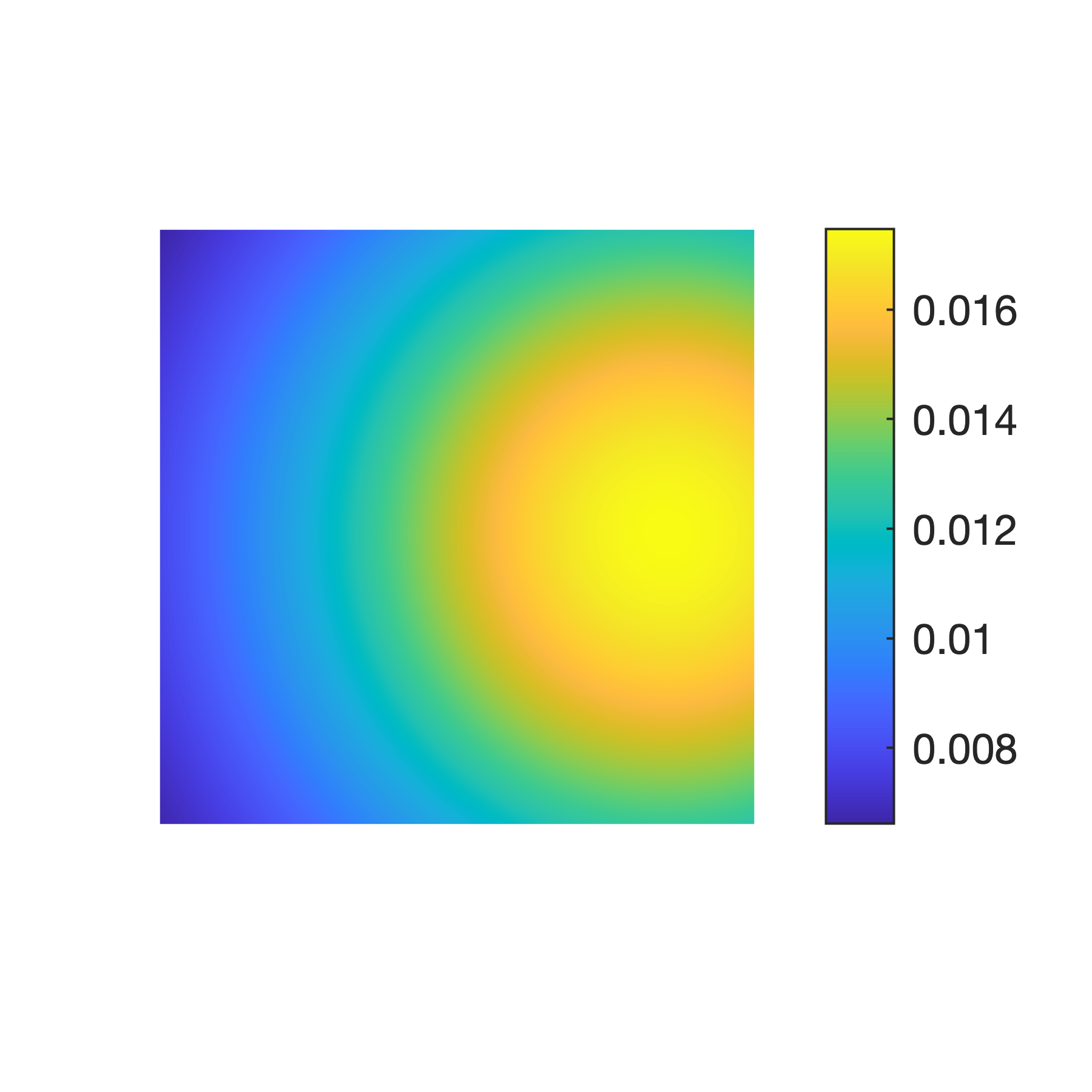}	\includegraphics[width=0.22\textwidth,trim=1.5cm 1.5cm 0.5cm 1cm,clip]{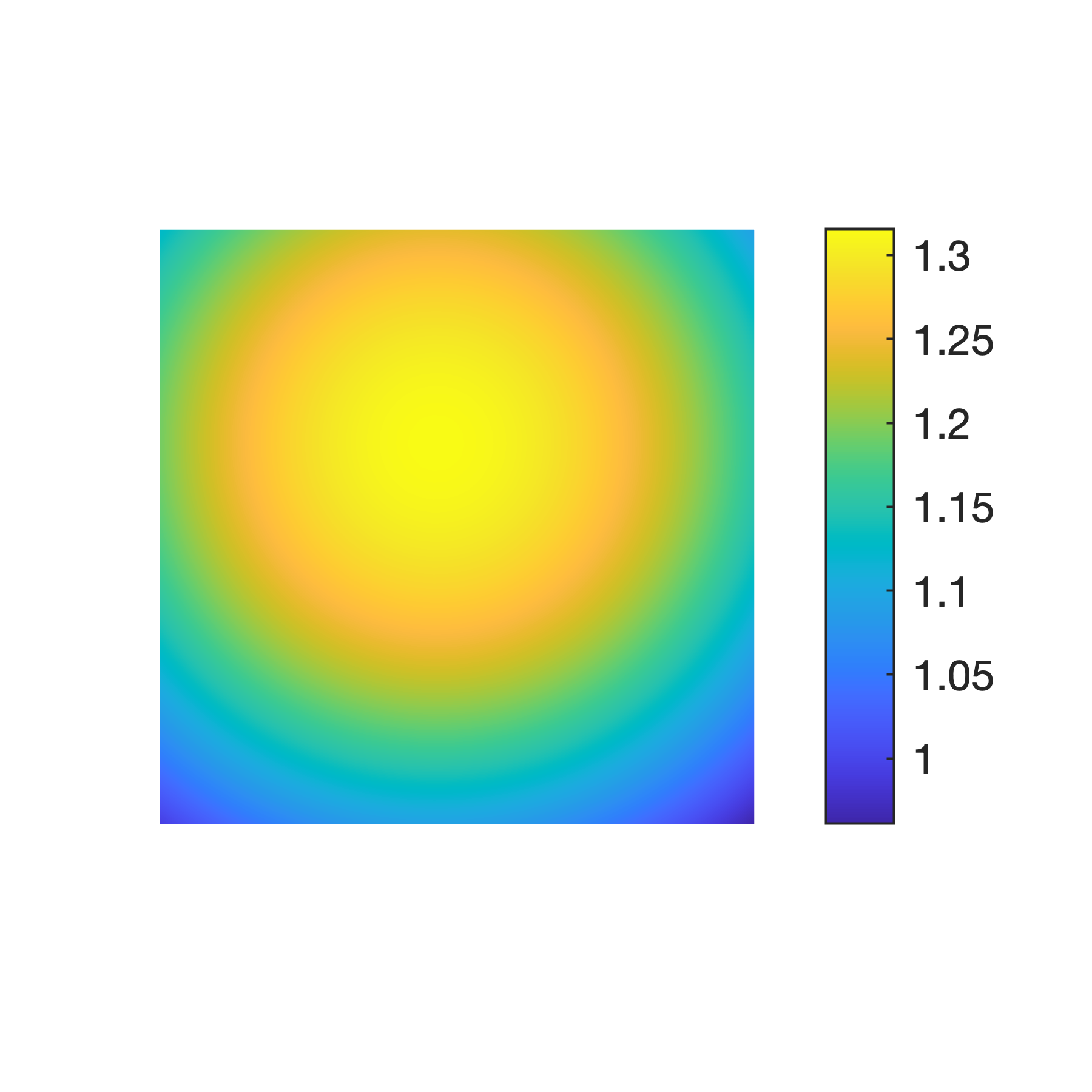}	\includegraphics[width=0.22\textwidth,trim=1.5cm 1.5cm 0.5cm 1cm,clip]{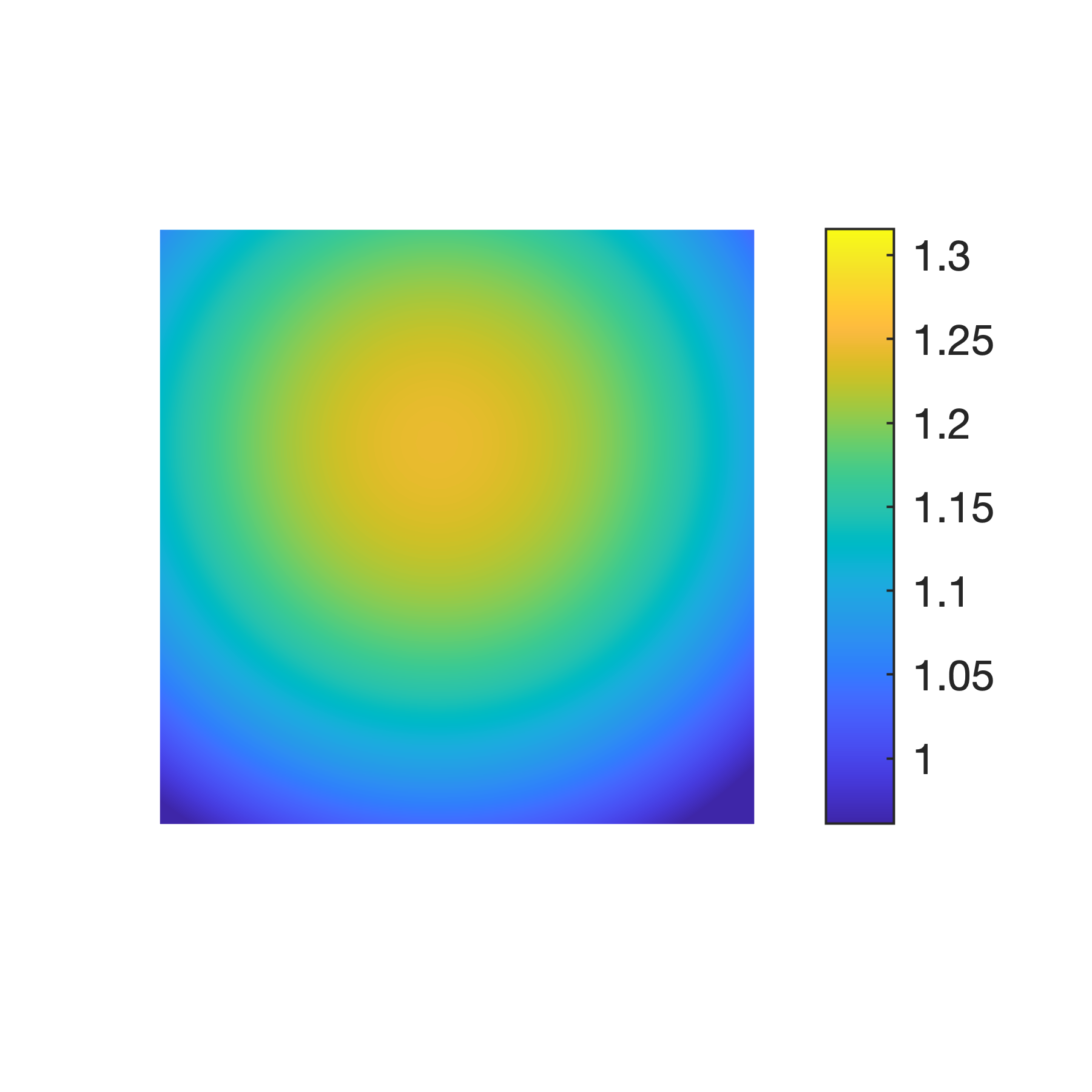}	\includegraphics[width=0.22\textwidth,trim=1.5cm 1.5cm 0.4cm 1cm,clip]{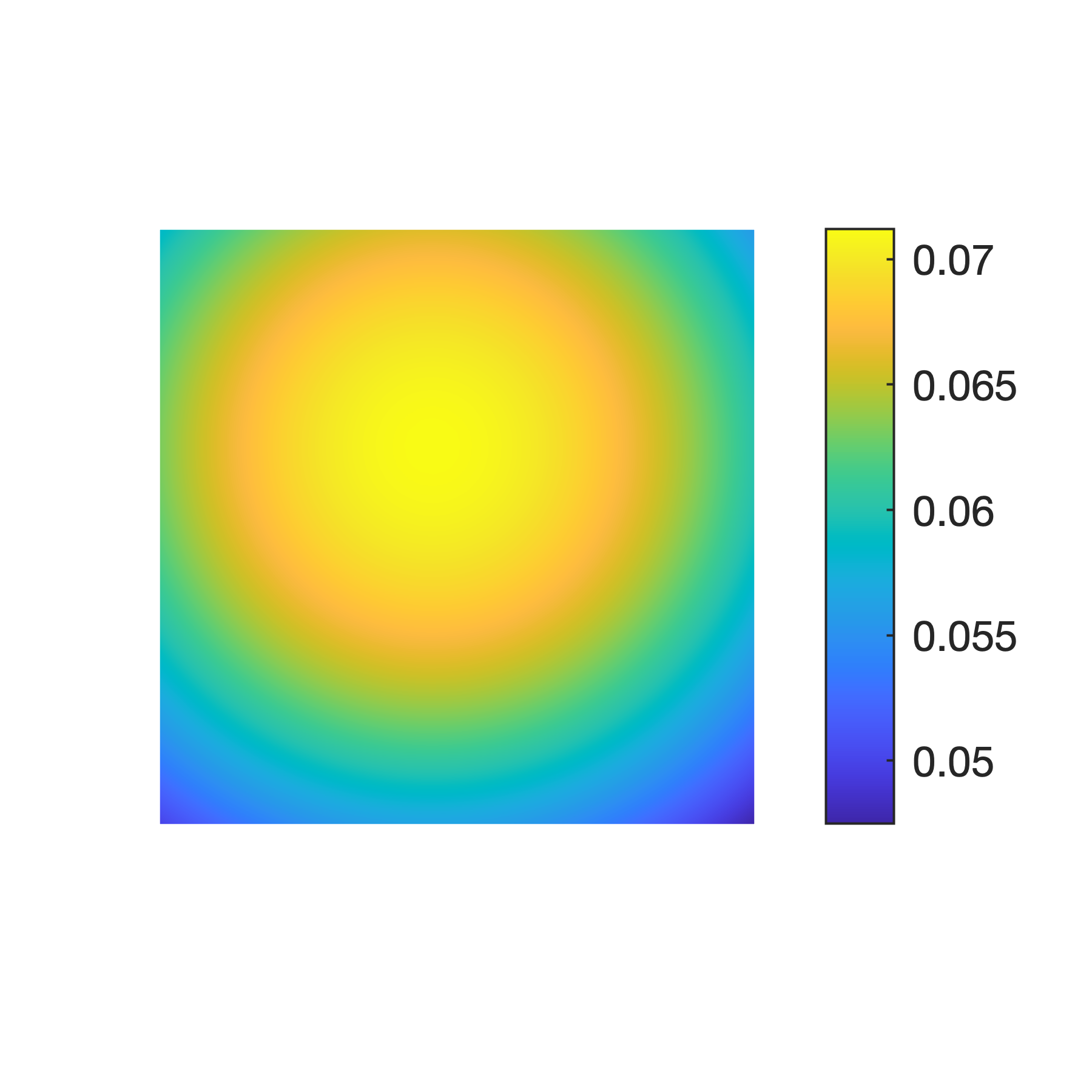}\\	\includegraphics[width=0.22\textwidth,trim=1.5cm 1.5cm 0.5cm 1.5cm,clip]{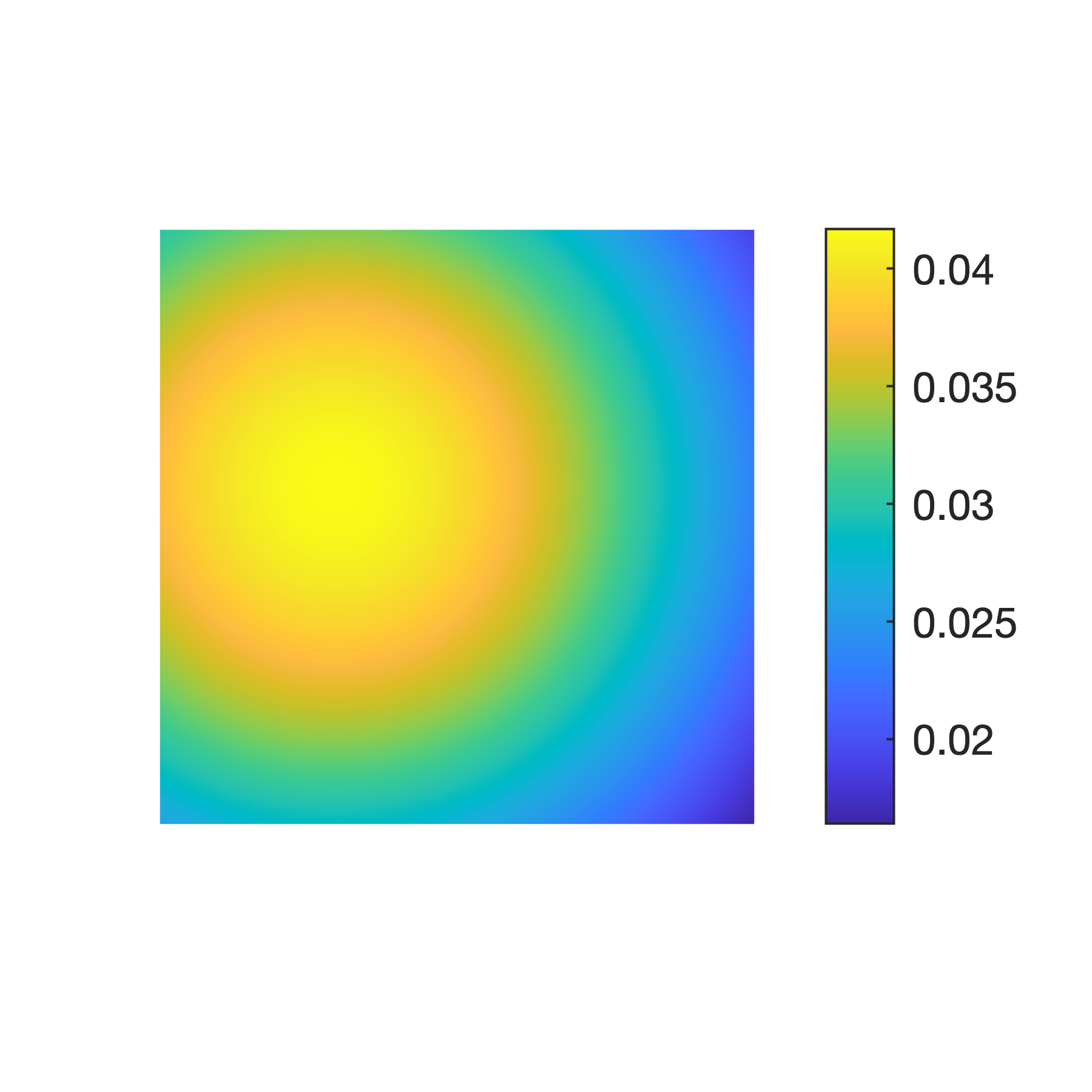}	\includegraphics[width=0.22\textwidth,trim=1.5cm 1.5cm 0.5cm 1.5cm,clip]{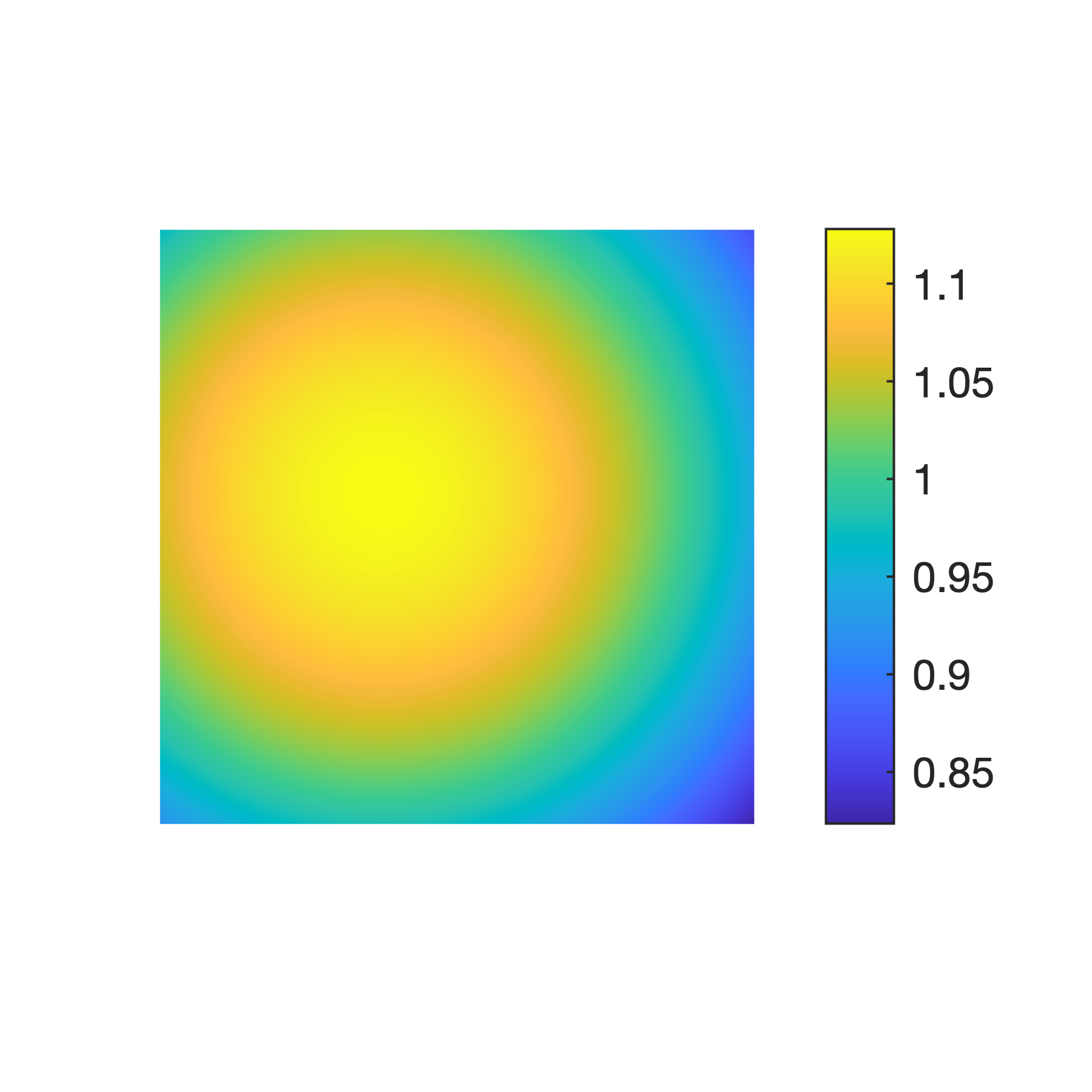}	\includegraphics[width=0.22\textwidth,trim=1.5cm 1.5cm 0.5cm 1.5cm,clip]{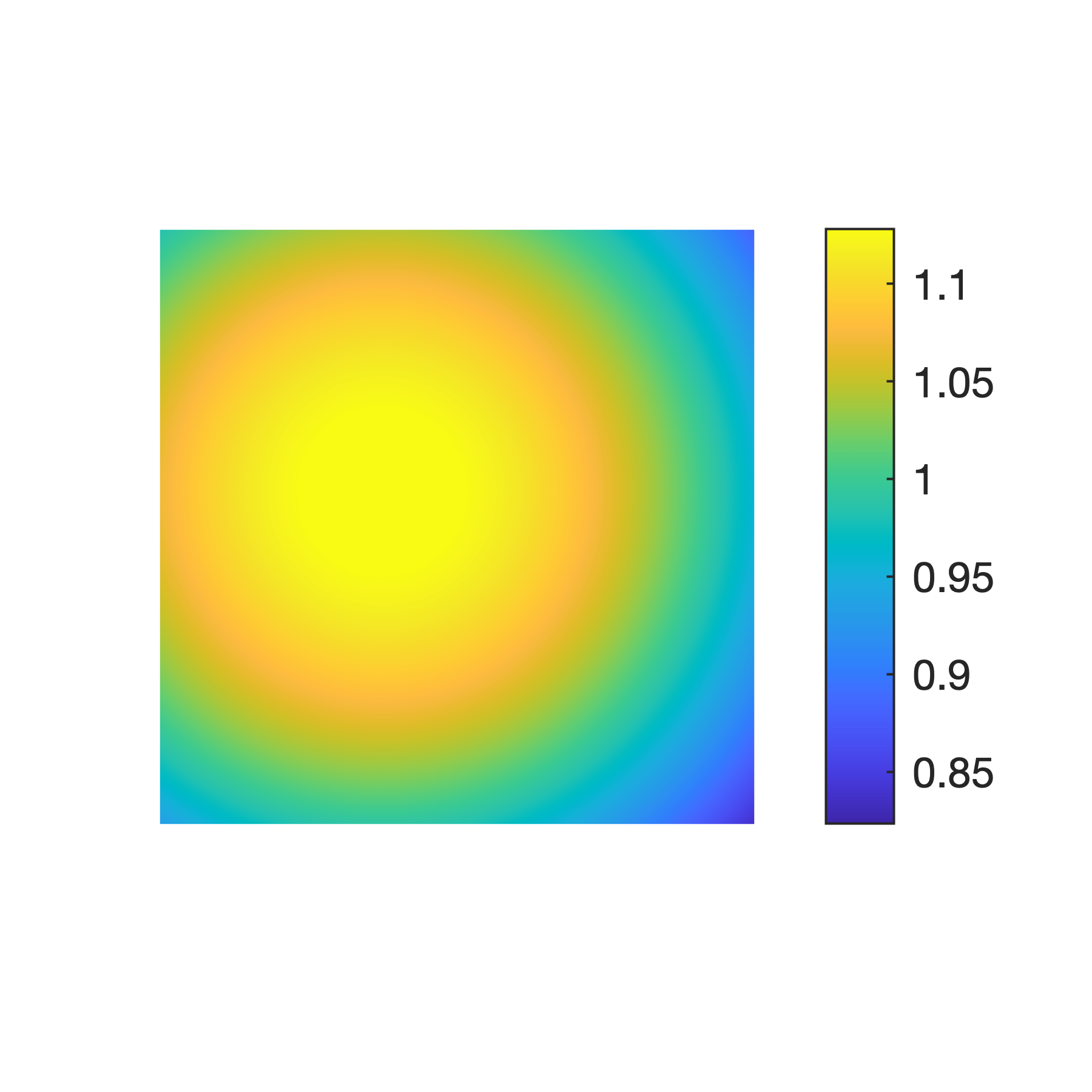}	\includegraphics[width=0.22\textwidth,trim=1.5cm 1.5cm 0.4cm 1.5cm,clip]{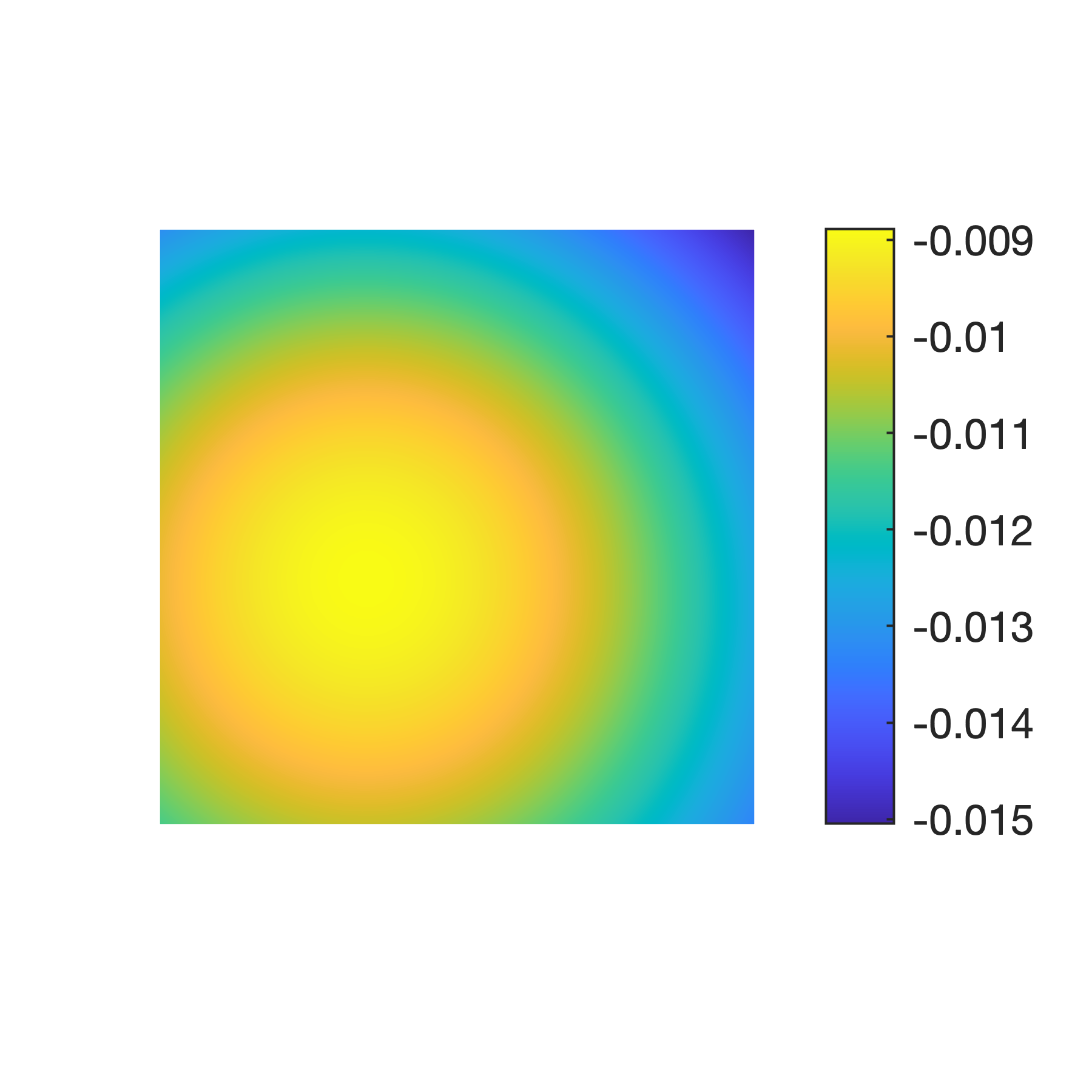}
\caption{Three randomly selected diffusion fields $\gamma$ from the testing dataset and the corresponding absorption fields $\sigma$ for Numerical Experiment I. From the left to the right are the exact $\gamma$, the corresponding exact $\sigma$, the $\sigma$ predicted by the learned polynomial model, and the pointwise error in the learning.}
\label{FIG:absorption_nn_samples_example_two}
\end{figure}

\noindent\textbf{Learning and testing performance.} We perform learning using the polynomial model~\eqref{EQ:Poly} with $n = 2$ following the computational procedure documented in~\Cref{SUBSEC:Polynomial}. To demonstrate the success in learning, we present in~\Cref{FIG:absorption_nn_samples_example_two} the prediction of three randomly selected absorption fields $\sigma$ from the testing dataset, which are generated by the diffusion fields $\gamma$ in~\eqref{EQ:diffusion_absorption_gaussian}. From left to right, we show the true diffusion fields, the corresponding exact absorption fields, the absorption fields predicted by the learned polynomial model, and the errors in the prediction. We see that the learned polynomial model gives a reasonable accuracy of the prediction.

\begin{figure}[htb!]
	\centering
	\includegraphics[width=0.3\textwidth,trim=1cm 1.5cm 0.5cm 1.5cm,clip]{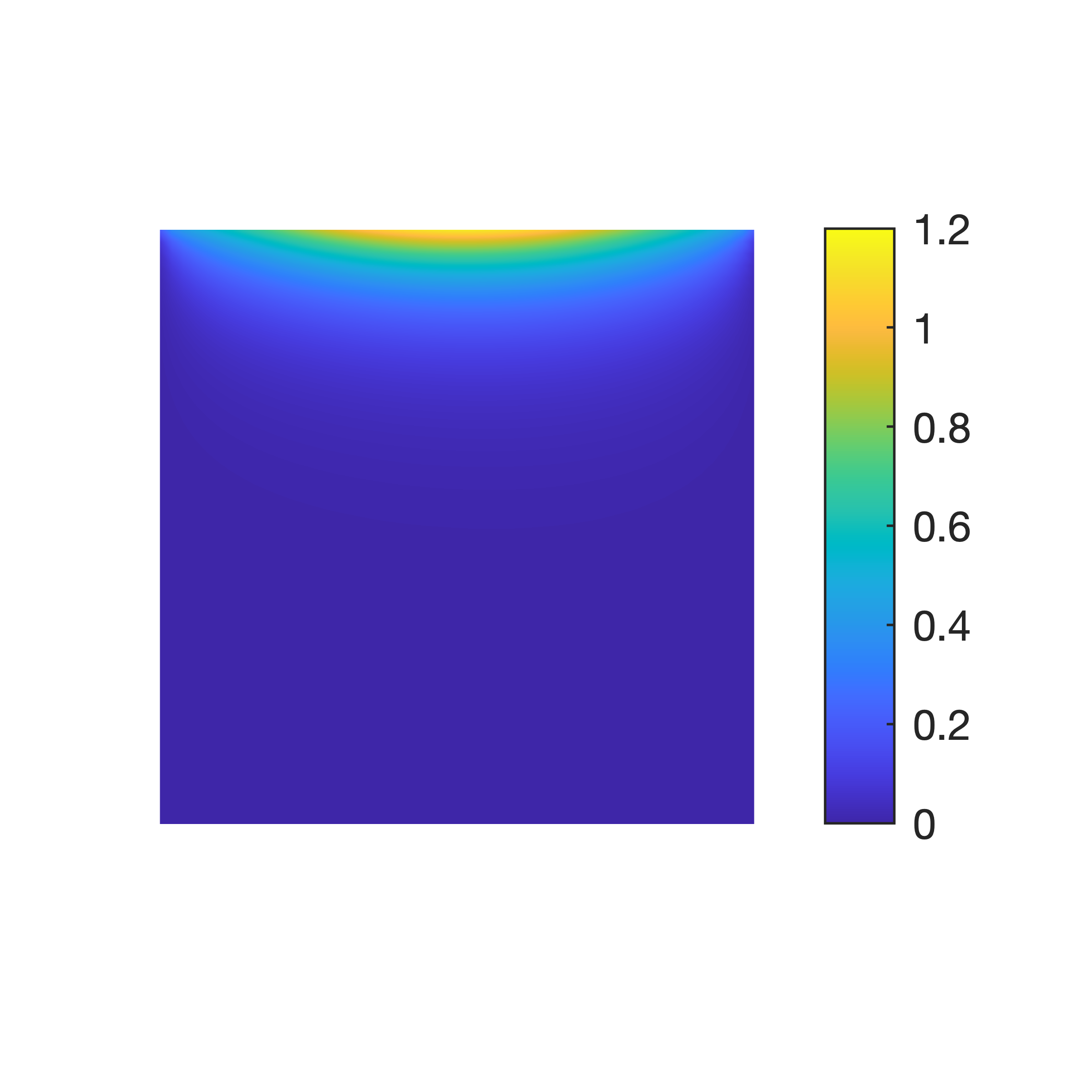}
\includegraphics[width=0.3\textwidth,trim=1cm 1.5cm 0.5cm 1.5cm,clip]{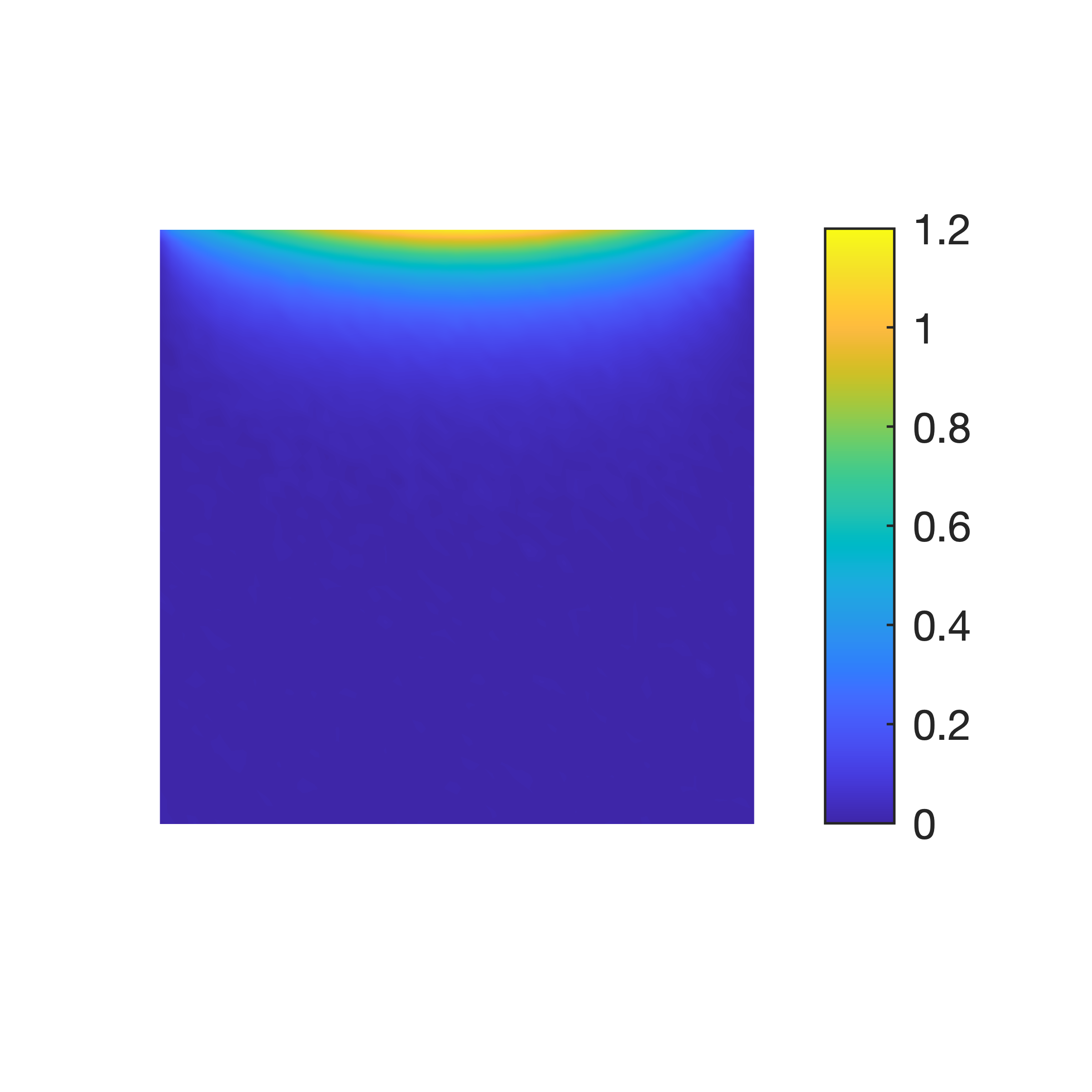}
\includegraphics[width=0.3\textwidth,trim=1cm 1.5cm 0.5cm 1.5cm,clip]{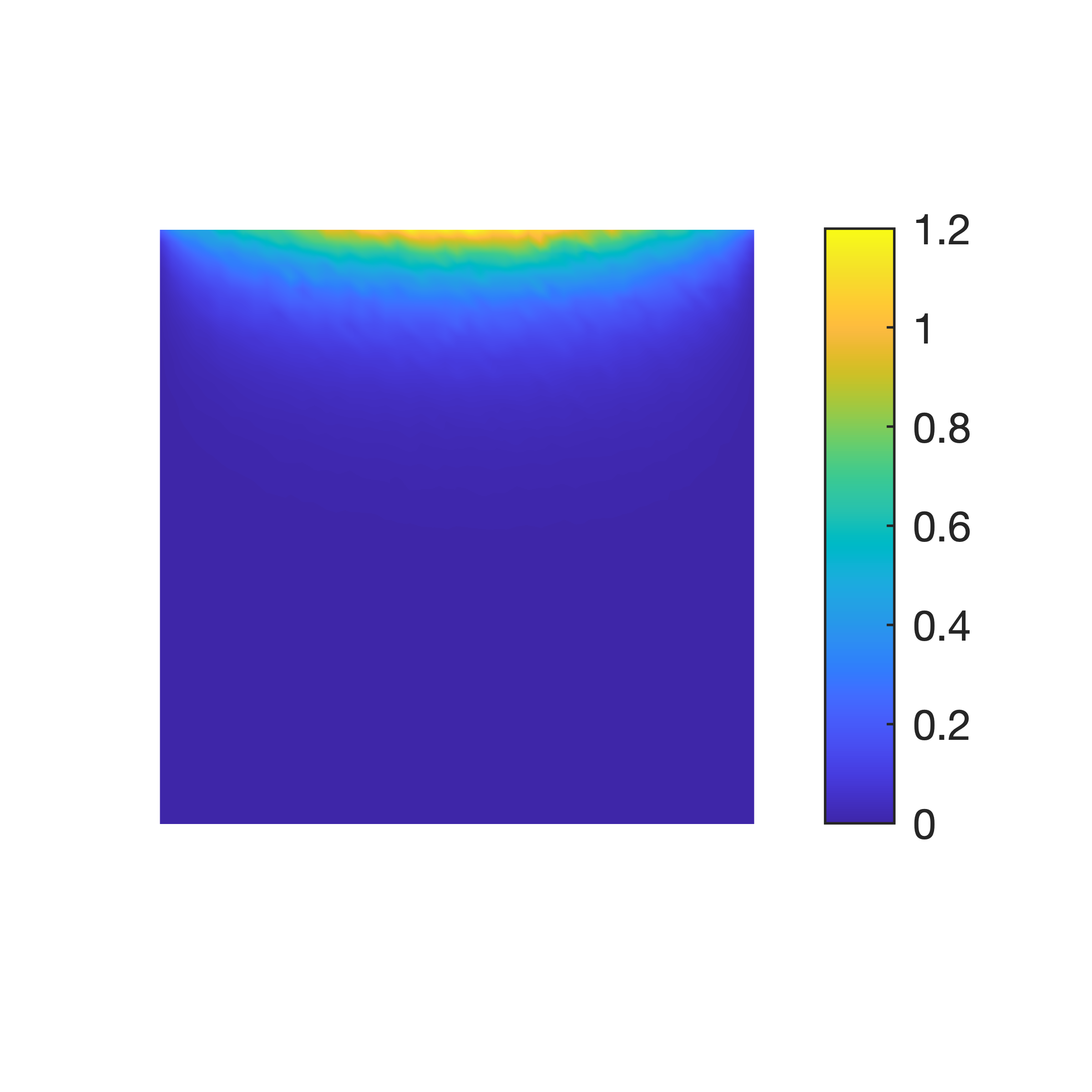} 
	\caption{Internal data $H$ generated with illuminating source~\eqref{source_top_diffusion_example} for Numerical Experiment I. From left to right are respectively noise-free data \RED{($H$)}, data with $5\%$ additive Gaussian noise \RED{($H^\delta = H + 0.05\mbox{mean}(H)\eta$)}, and data with $5\%$ multiplication Gaussian noise \RED{($H^\delta = H(1+ 0.05\eta$)). Here, $\eta$ is a standard normal random variable.}}
	\label{FIG:usigma_data_example2}
\end{figure}
\noindent\textbf{Data-driven joint inversion.} With the learned relation between $\gamma$ and $\sigma$, we perform joint reconstructions of $(\gamma, \sigma)$ from internal data~\eqref{EQ:Diff Data}. The internal data are again synthetic; see Figure~\ref{FIG:usigma_data_example2} for the datum $H$ generated with the boundary source 
\begin{equation}\label{source_top_diffusion_example}
    S(\bx)  = e^{-\frac{(x - 0.5)^2}{0.25}}, \ \ \ \bx=(x, y),\ \ \ x\in(0, 1),\ y=1\,,
\end{equation}
on the top segment of the boundary. 
We implemented Algorithm~\ref{ALG:JointInv} along with a BFGS quasi-Newton reconstruction algorithm for the optimization problems~\eqref{EQ:Min 0} and ~\eqref{EQ:Min j}, which we use the {\rm MATLAB} inline function {\rm fmincon} with the `bfgs' option. In particular, the termination tolerances are set to be $10^{-7}$ for the first-order optimality condition, and $10^{-7}$ for the norm of the updating step size. 

\RED{To test the stability of the proposed coupled scheme, we add Gaussian noise with zero mean and $5\%$ standard deviation to the data used. Figure~\ref{FIG:recover_step_two_example2} displays the reconstructed diffusion field, as well as the reconstructed absorption fields. We observe that the reconstructions are fairly accurate in general. Though the resolution of the reconstructions from the noisy data is not as good as the one from the noise-free data, the bump locations for both the diffusion and the absorption fields are clear and accurate.}
\begin{figure}[htb!]
	\centering
 \includegraphics[width=0.22\textwidth,trim=1.5cm 2cm 0.5cm 1cm,clip]{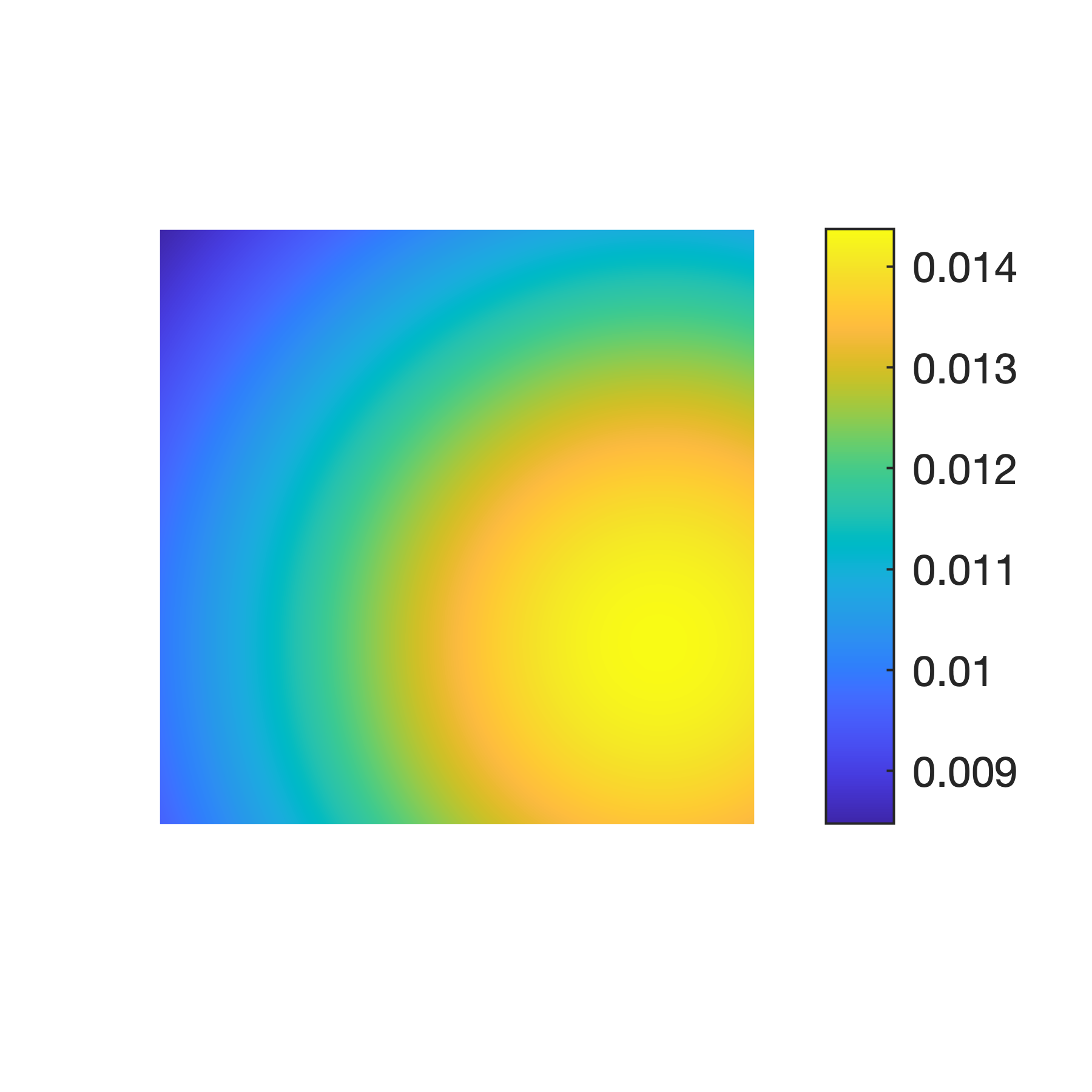}
	\includegraphics[width=0.22\textwidth,trim=1.5cm 2cm 0.5cm 1cm,clip]{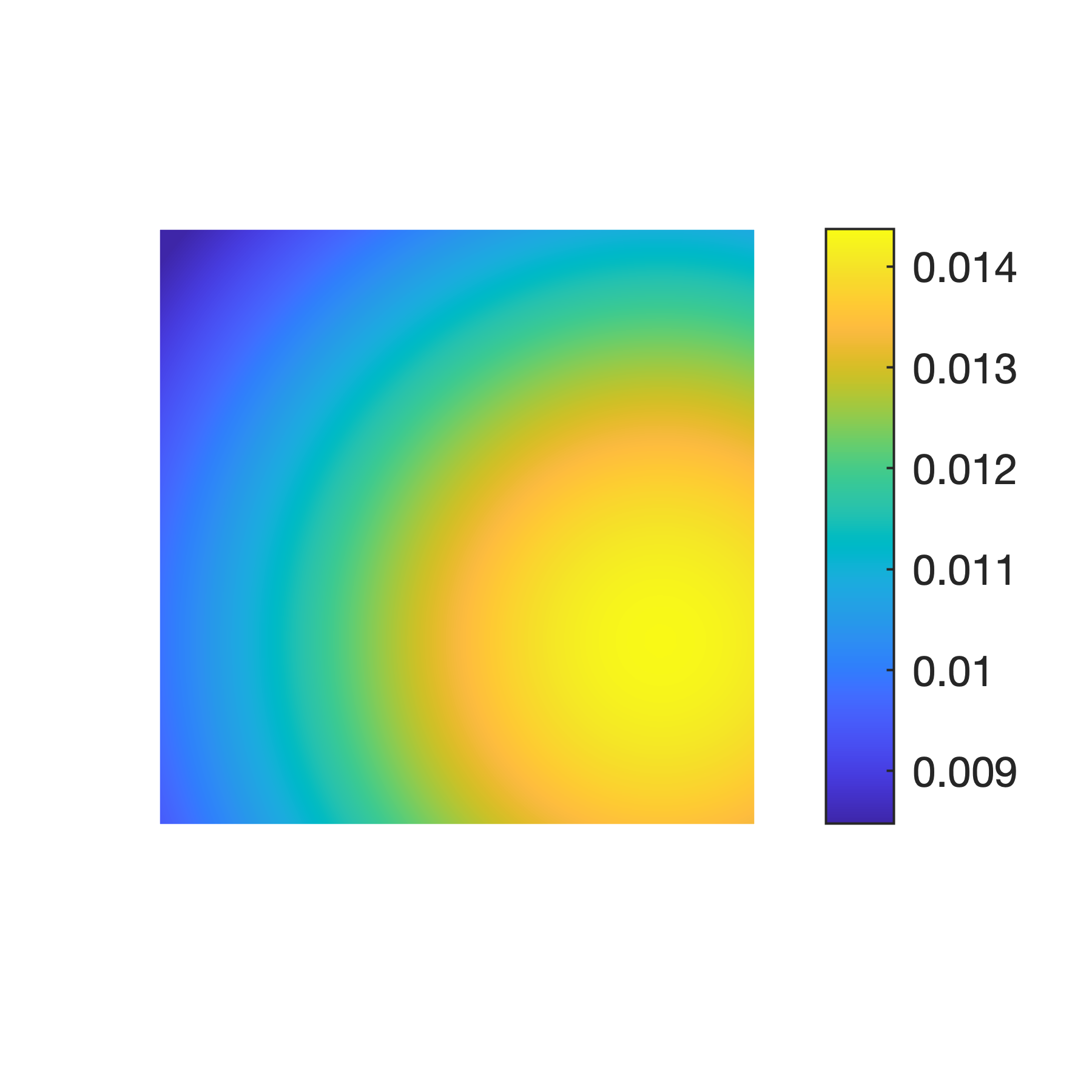}
	\includegraphics[width=0.22\textwidth,trim=1.5cm 2cm 0.5cm 1cm,clip]{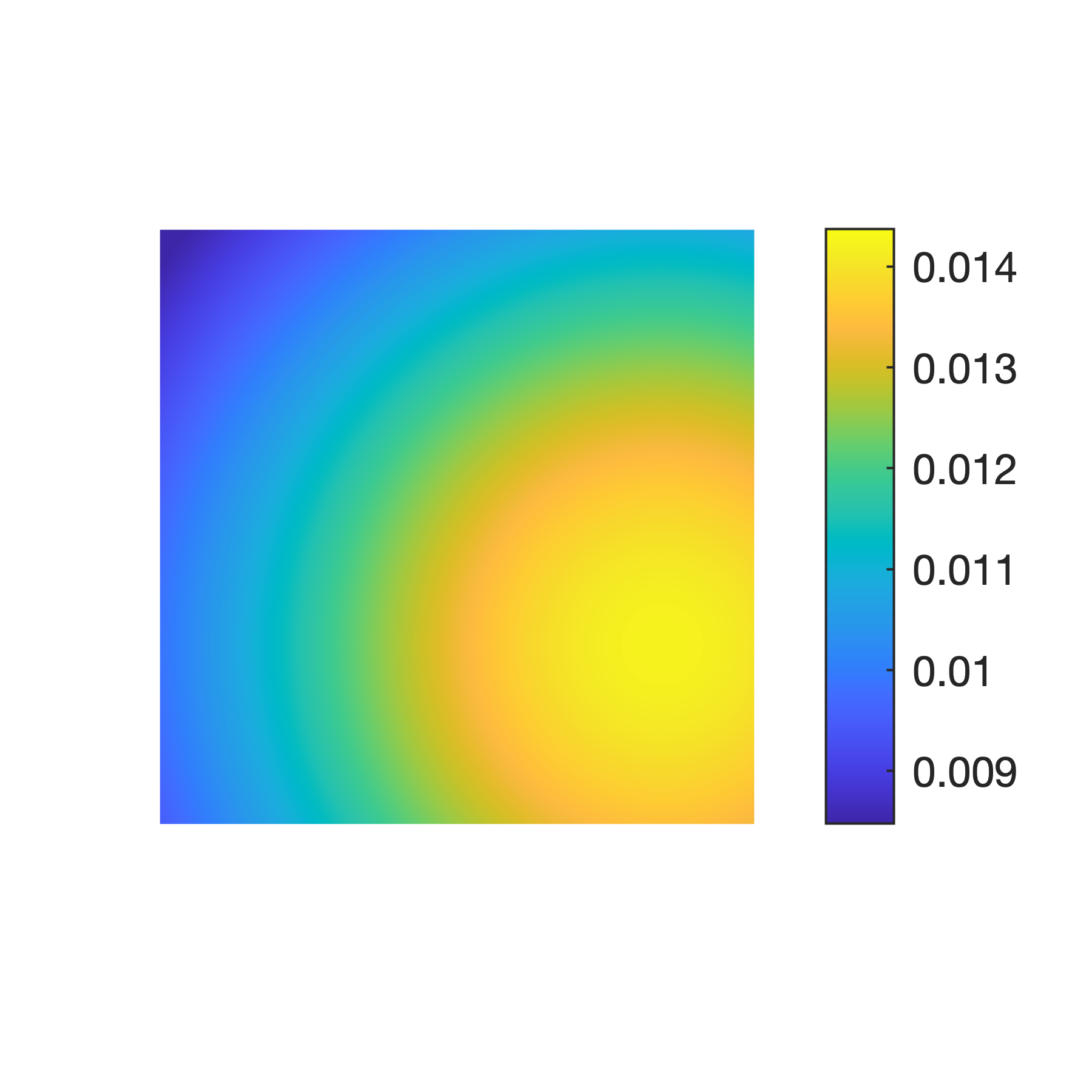}\\
	\includegraphics[width=0.22\textwidth,trim=1.5cm 2cm 0.5cm 1cm,clip]{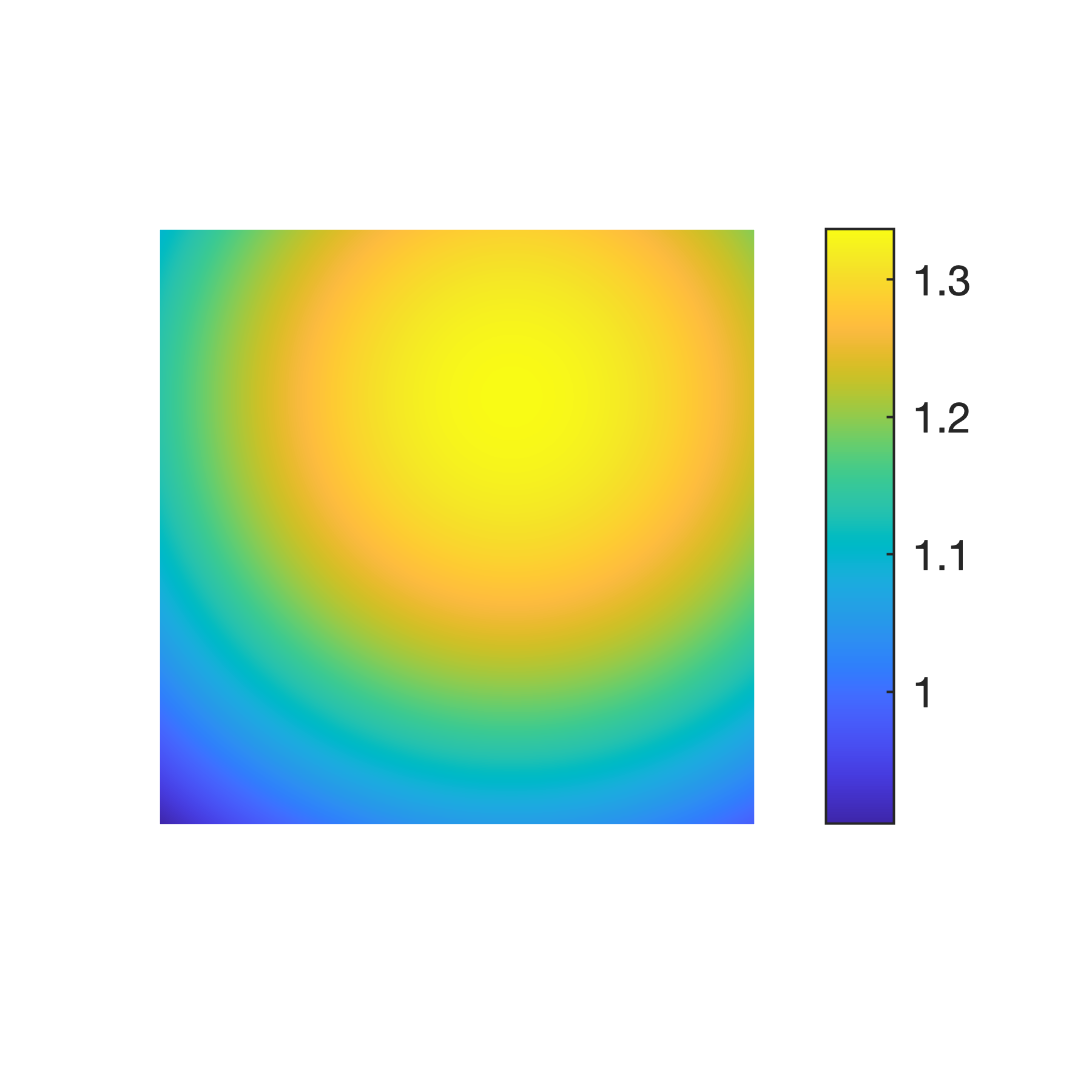}
	\includegraphics[width=0.22\textwidth,trim=1.5cm 2cm 0.5cm 1cm,clip]{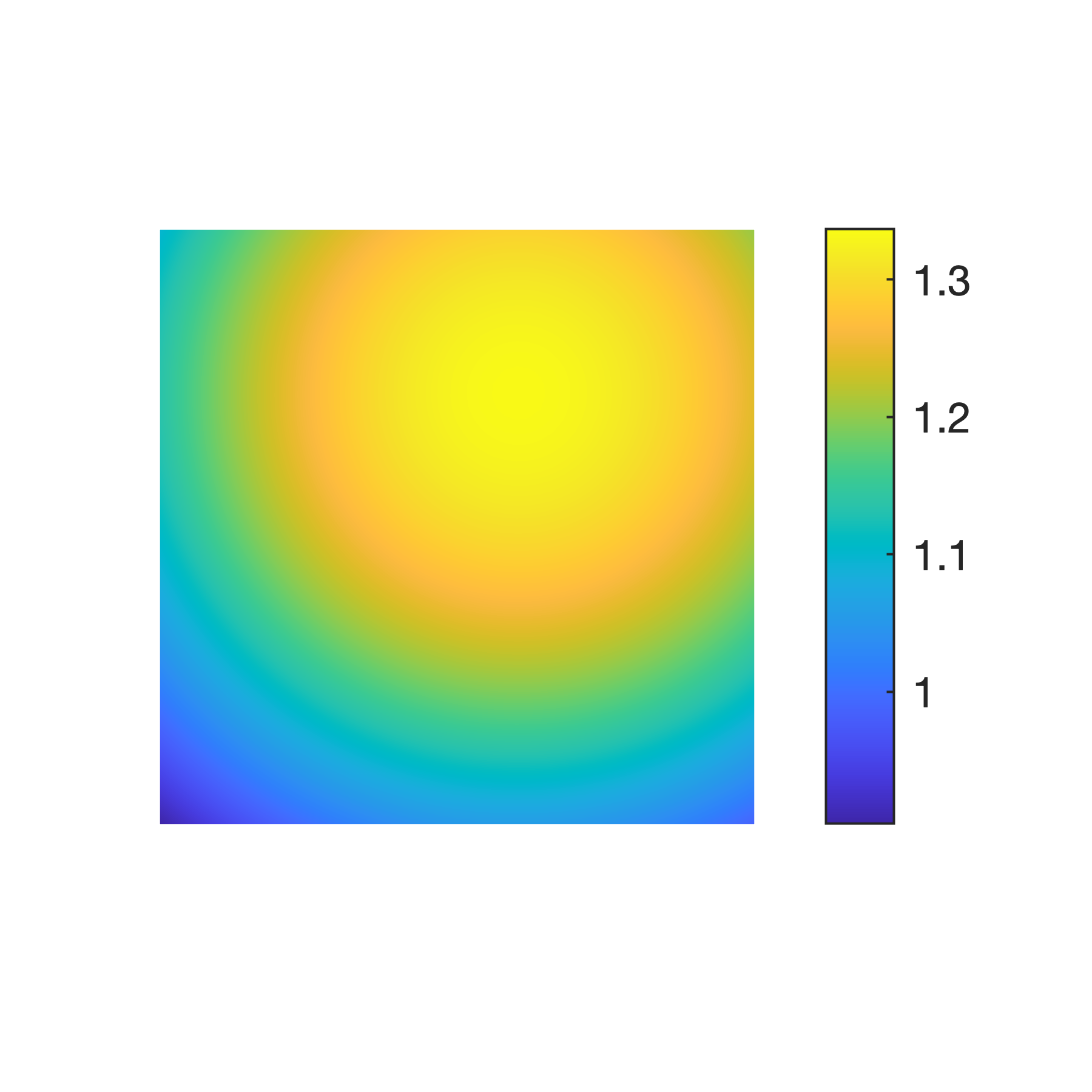}
	\includegraphics[width=0.22\textwidth,trim=1.5cm 2cm 0.5cm 1cm,clip]{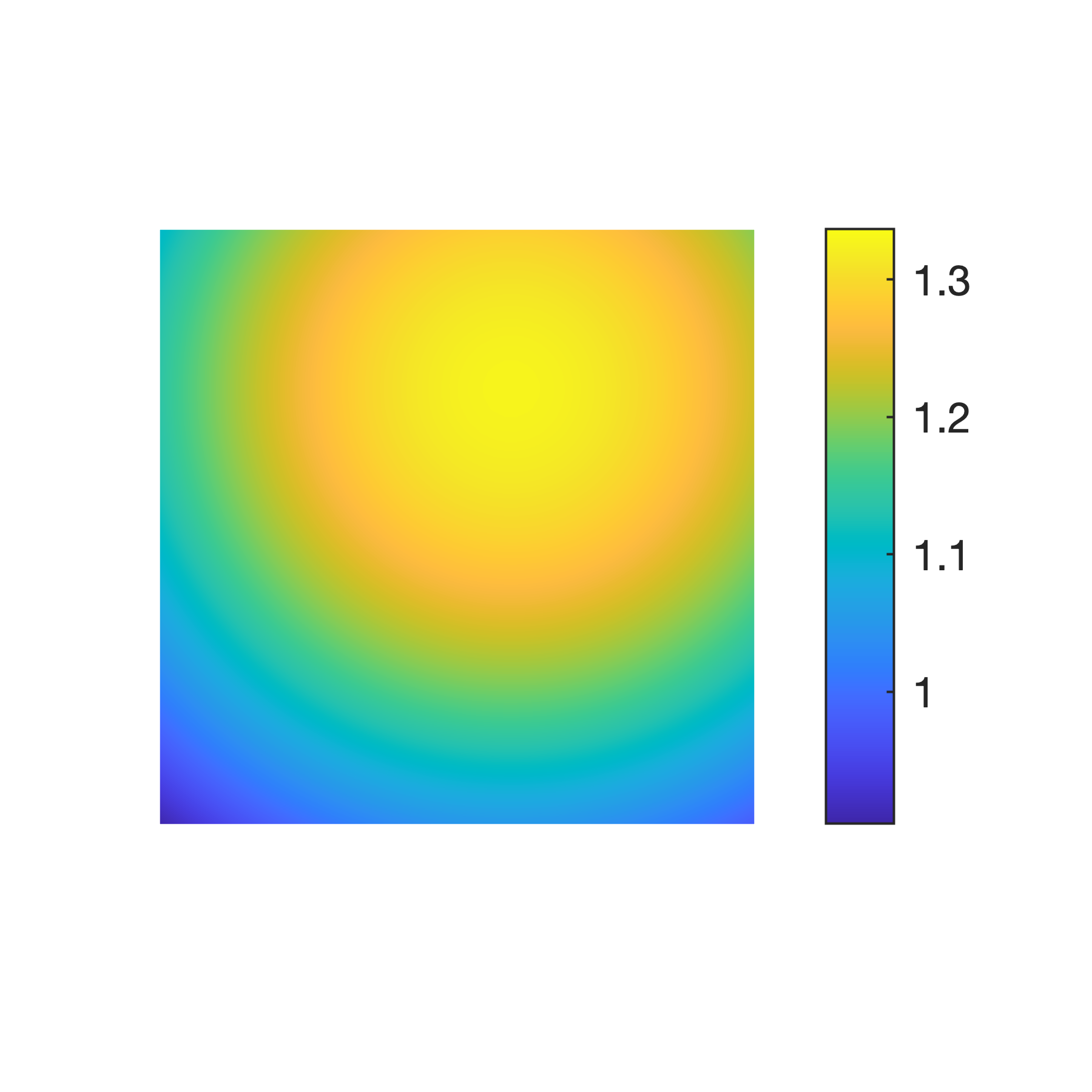}
\caption{Joint inversion of $(\gamma, \sigma)$ (top, bottom) in Numerical Experiment I. Shown from left to right are true coefficients, reconstruction from noise-free data, %reconstruction from data with $5\%$ additive Gaussian noise, 
and reconstruction from data with $5\%$ multiplication Gaussian noise.}
	\label{FIG:recover_step_two_example2}
\end{figure}

%%%%%%%%%%%%%%%%%%%%%%%%%
\subsubsection{Experiment II: Learning with neural network model}
%\label{SUBSEC:Num Diff}
%%%%%%%%%%%%%%%%%%%%%%%%%

In the second numerical experiment, we consider the coefficients $(\gamma, \sigma)$ of the forms, in the Fourier domain,
\begin{equation}\label{diffusion_fourier}
\gamma({\bx}) = \sum_{k_x,k_y = 0}^K \wh\gamma_{\bk}\phi_{\bk}(x, y), \qquad \sigma({\bx}) = \sum_{k_x,k_y = 0}^K \wh \sigma_{\bf k}\phi_{\bf k}(x, y)
\end{equation}
where ${\bf k} = (k_x, k_y)$ and $\phi_{\bf k}(x, y) = \cos(k_x \pi x)\cos(k_y \pi y)$ are the eigenfunctions in~\eqref{EQ:Eigenfunction} for the domain $\Omega=(0, 1)^2$. We assume that the generalized Fourier coefficients are related by
\begin{equation}\label{EQ:Relation 1}
    \wh \sigma_{\bk} = \sum_{k_x',k_y' = 0}^K a_{{\bf k},{\bf k}'} \sin\big(\pi (2 + {\wh \gamma}_{{\bf k}'})^{k_x+k_y}\big)\,.
\end{equation}
This relation is sufficiently smooth to be approximated with polynomials of relatively low orders.

%%%%%%%%%%%%%%%%%%%%%%
%\subsubsection{Learning dataset generation}
%%%%%%%%%%%%%%%%%%%%%%
\begin{figure}[htb!]
	\centering
 \includegraphics[width=0.24\textwidth,trim=1cm 1.5cm 0.5cm 1cm,clip]{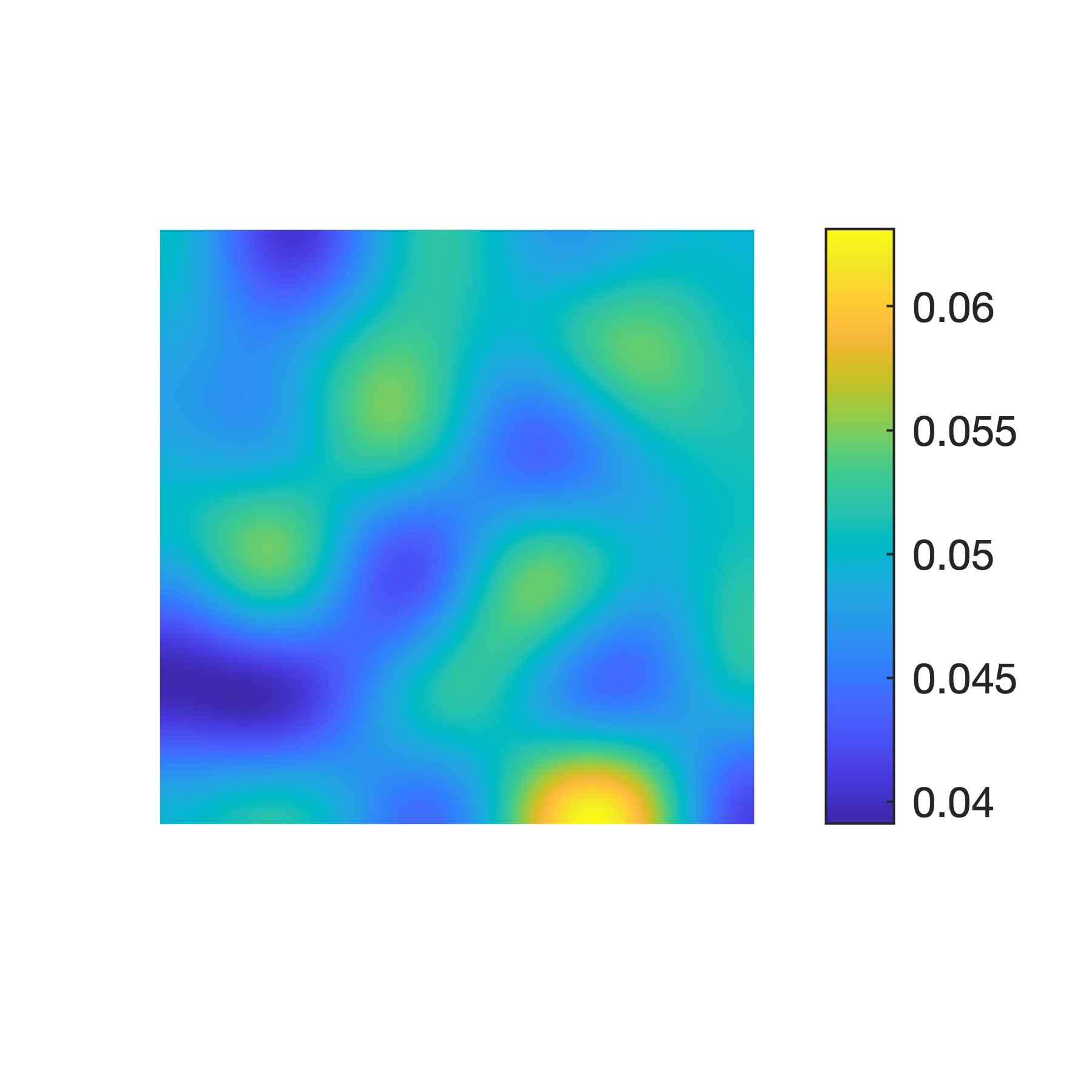}
 \includegraphics[width=0.24\textwidth,trim=1cm 1.5cm 0.5cm 1cm,clip]{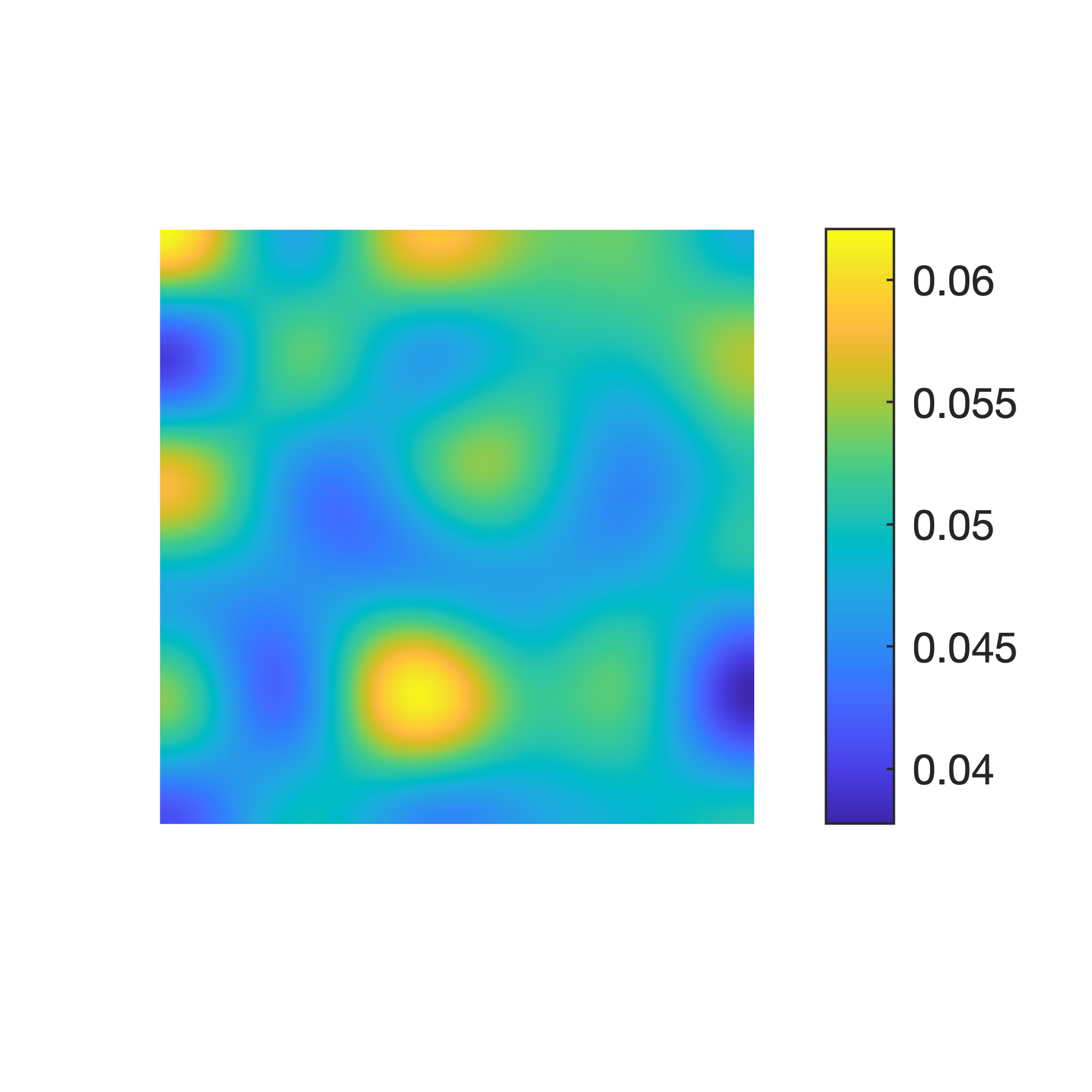}
 \includegraphics[width=0.24\textwidth,trim=1cm 1.5cm 0.5cm 1cm,clip]{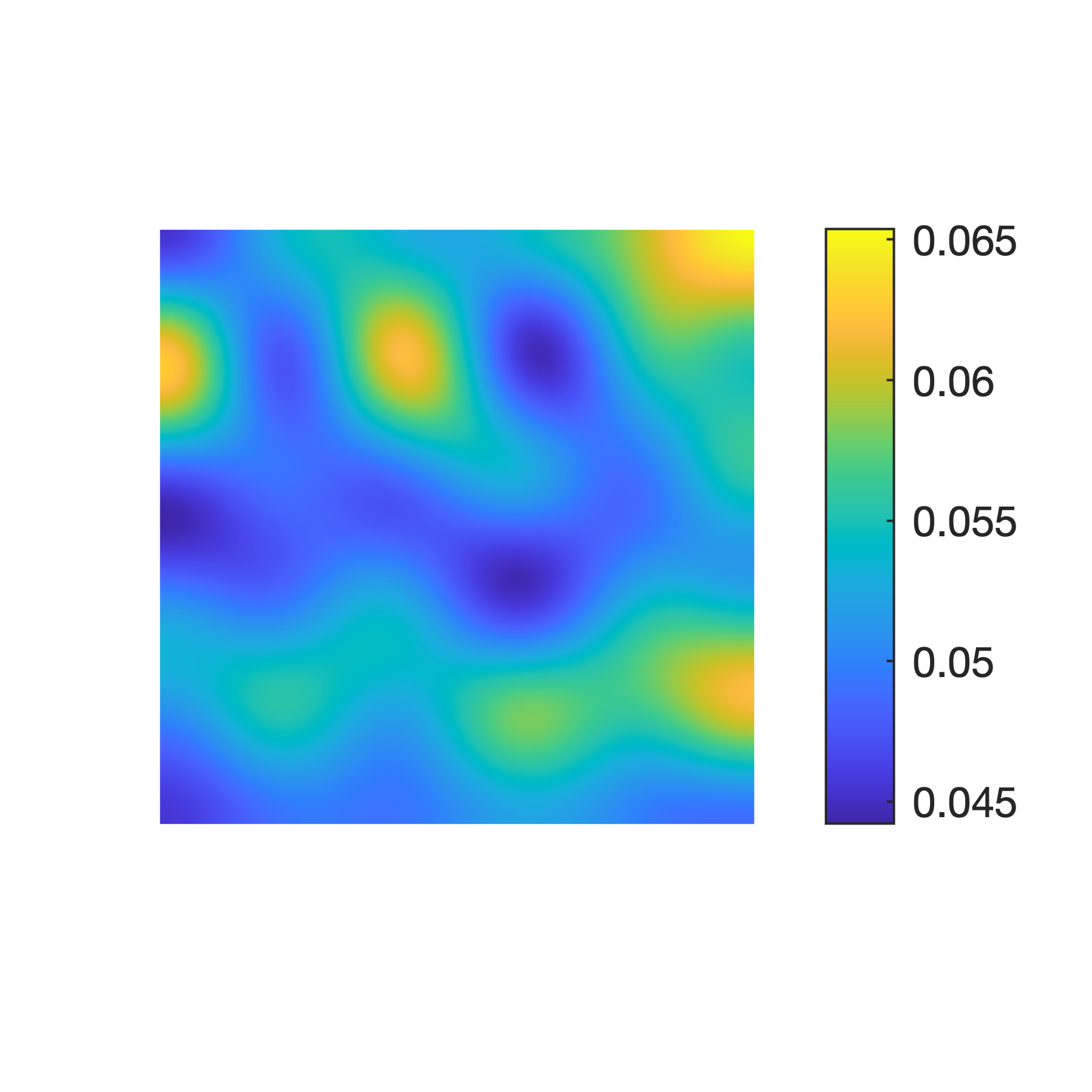}
 \includegraphics[width=0.24\textwidth,trim=1cm 1.5cm 0.5cm 1cm,clip]{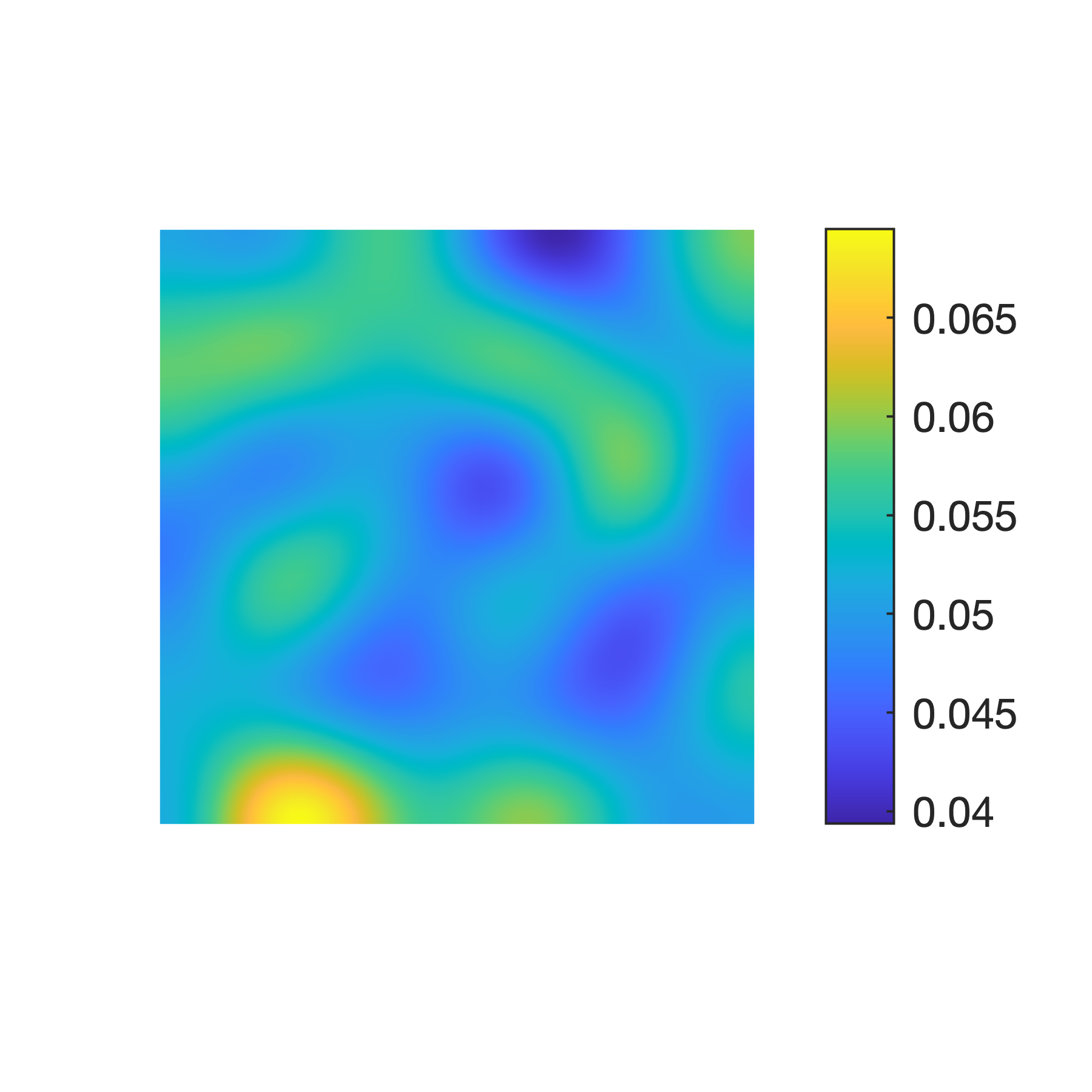}\\
 \includegraphics[width=0.24\textwidth,trim=1cm 2cm 0.5cm 1cm,clip]{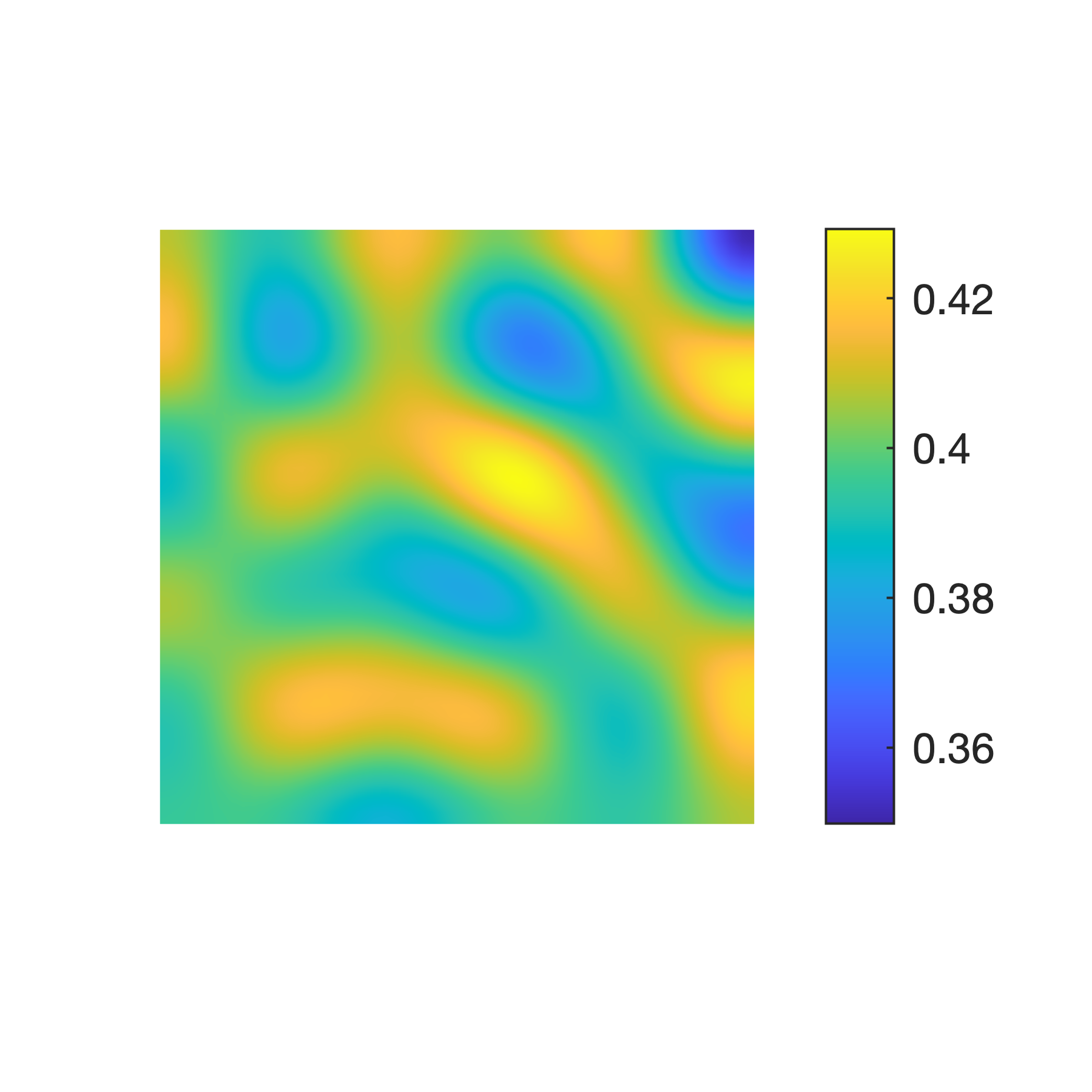}
 \includegraphics[width=0.24\textwidth,trim=1cm 2cm 0.5cm 1cm,clip]{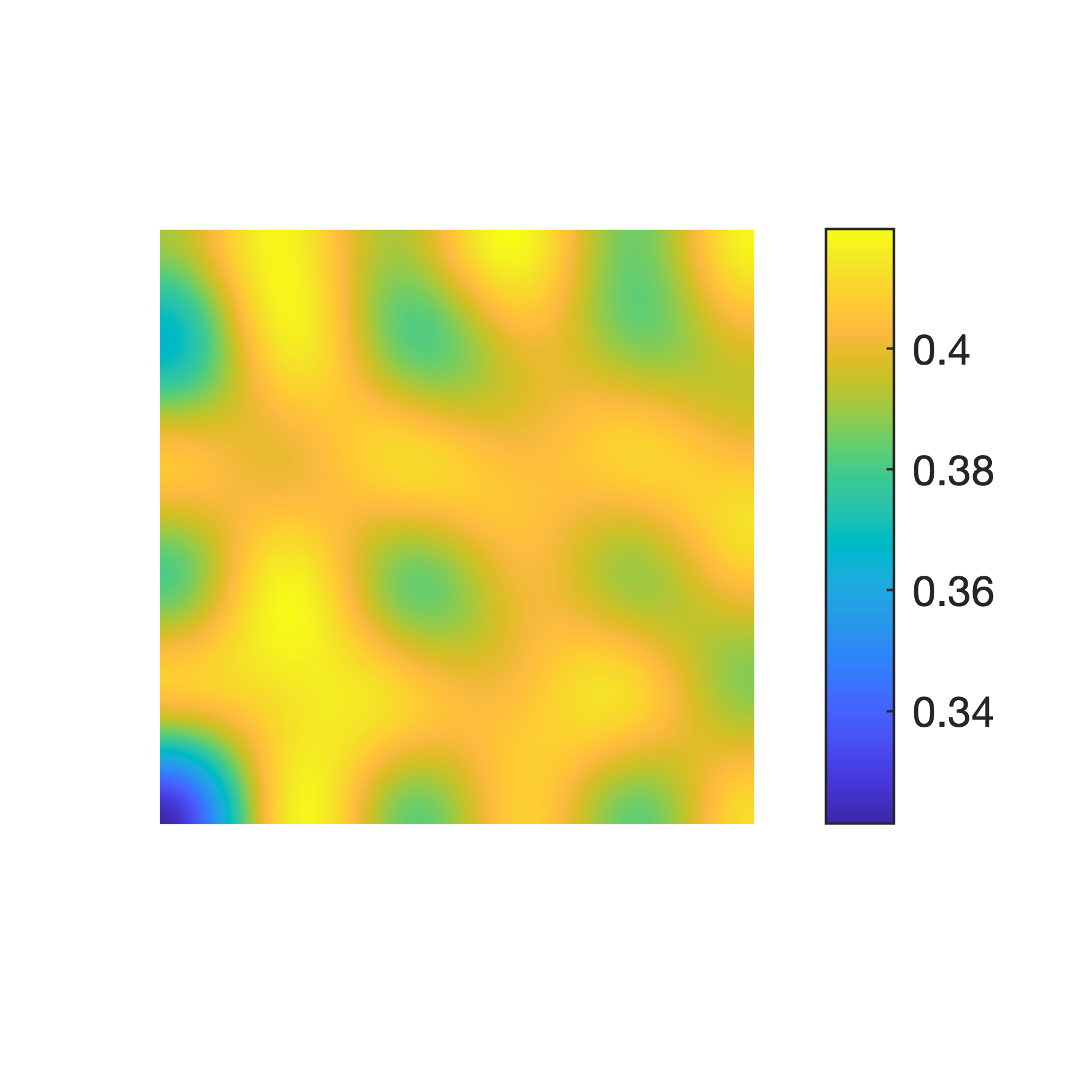}
 \includegraphics[width=0.24\textwidth,trim=1cm 2cm 0.5cm 1cm,clip]{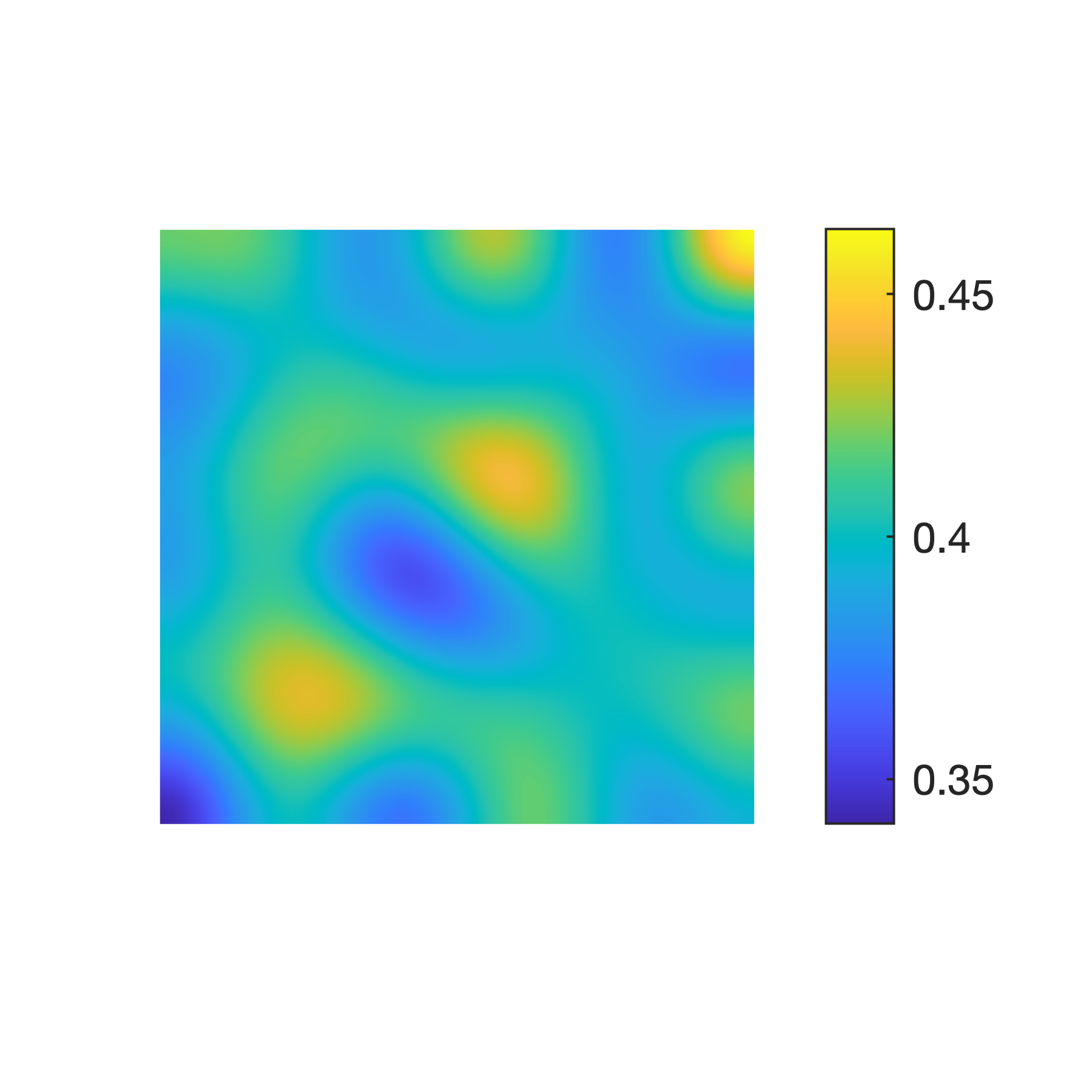}
 \includegraphics[width=0.24\textwidth,trim=1cm 2cm 0.5cm 1cm,clip]{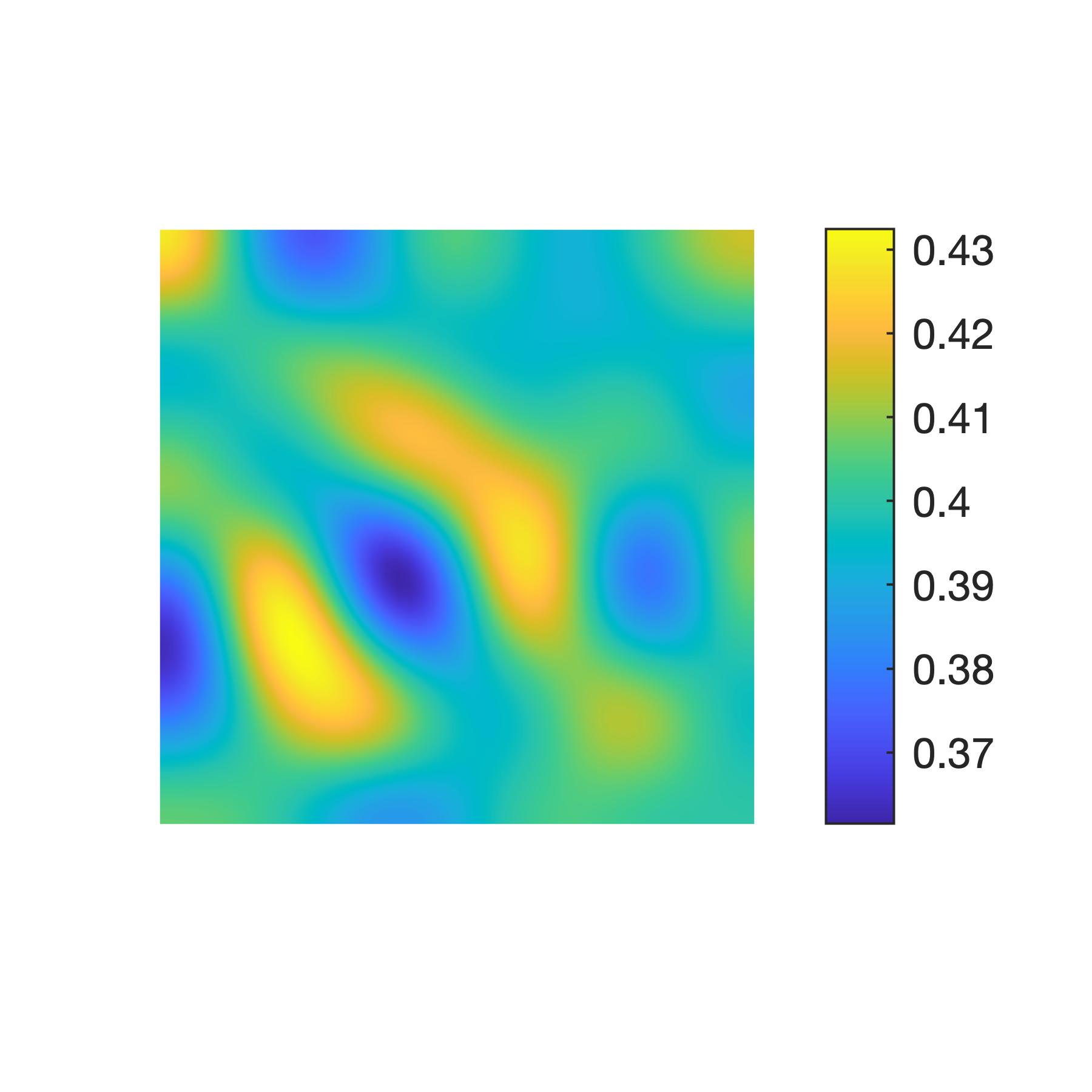}
	\caption{Four random samples of the $\gamma$ (top) and $\sigma$ (bottom) pair generated from \eqref{diffusion_fourier}-\eqref{EQ:Relation 1} with $K=5$ for the learning stage of Numerical Experiment II.}
	\label{FIG:diffusion_absorption_samples}
\end{figure}
\noindent\textbf{Learning dataset generation.} We generate a set of $N=10^4$ coupled diffusion-absorption fields $\{\gamma_j, \sigma_j\}_{j=1}^N$ based on the relation~\eqref{EQ:Relation 1}. More precisely, we randomly choose the coefficients $\{{\wh \gamma}_{{\bf k}}\}_{{\bf k}\in \mathbb{N}_0\times \mathbb{N}_0}$ from the uniform distribution $\mathcal{U}[-0.5,0.5]$, and for each ${\bf k}$ we pick $\{{a}_{{\bf k}, {\bf k}'}\}_{{\bf k}'\in \mathbb{N}_0\times \mathbb{N}_0}$ from the uniform distribution $\mathcal{U}[0,0.1]$ to construct $\wh\sigma_{\bf k}$. To impose the constraint on the coefficients' non-negativity, we rescale them linearly to make them in the right range of values. In Figure~\ref{FIG:diffusion_absorption_samples} we show some typical samples of the $(\gamma, \sigma)$ pair generated from \eqref{diffusion_fourier}-\eqref{EQ:Relation 1}. 

%%%%%%%%%%%%%%%%%%%%%%
%\subsubsection{Training and testing performance}
%%%%%%%%%%%%%%%%%%%%%%

\noindent\textbf{Learning and testing performance.}
To learn the nonlinear relation $\mathcal{N}$ between $\gamma$ and $\sigma$, we perform a standard training-validation cycle on the neural network approximating the nonlinear relation. To be precise, before the training process starts, we randomly split the synthetic dataset of $N = 10^4$ data points into a training dataset and a testing dataset. The training dataset takes $80\%$ of the original data points, and the test dataset takes the rest $20\%$ of the data points.
\begin{figure}[htb!]
	\centering
 \includegraphics[width=0.22\textwidth,trim=0.2cm 0.1cm 0.3cm 0.2cm,clip]{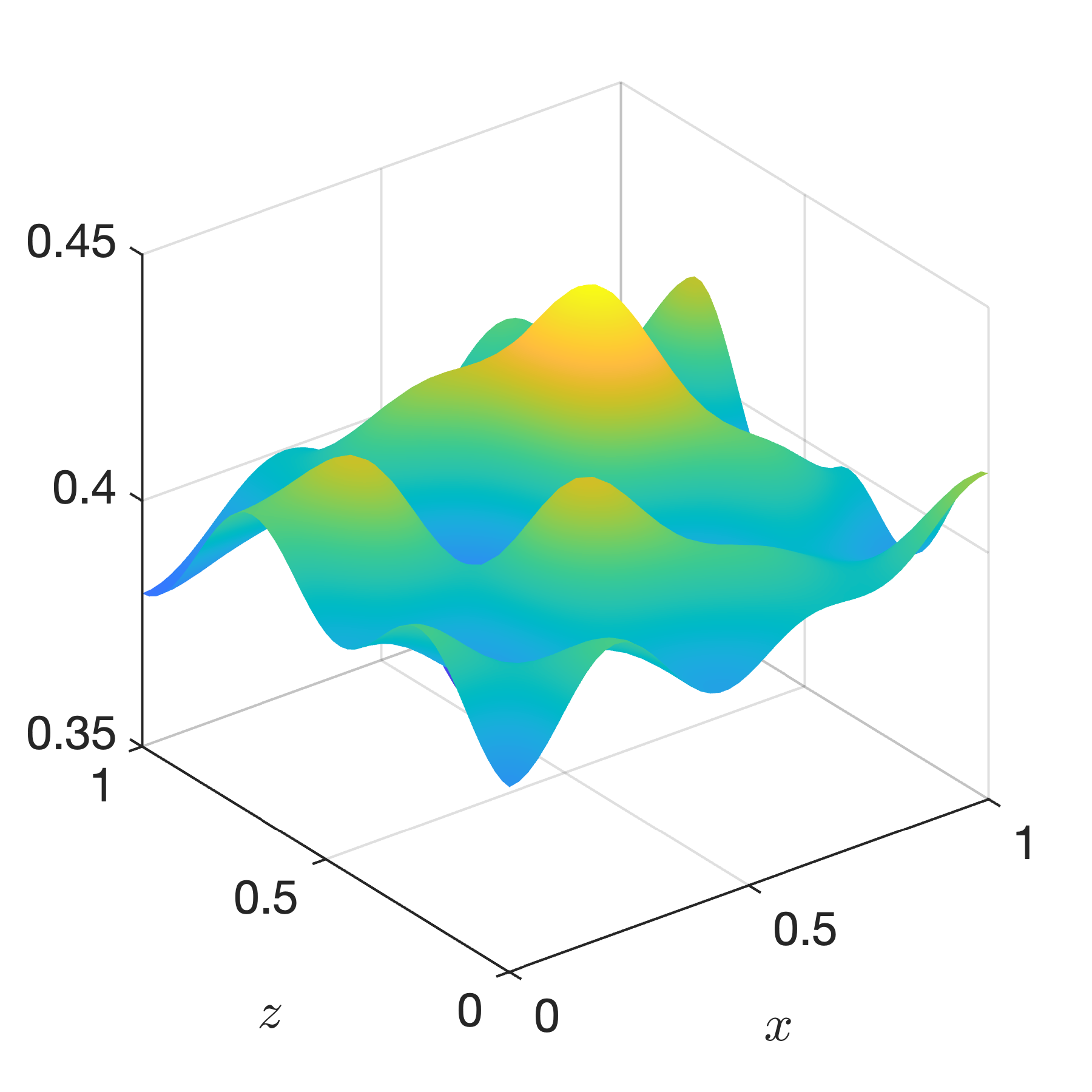}
\includegraphics[width=0.22\textwidth,trim=0.2cm 0.1cm 0.3cm 0.2cm,clip]{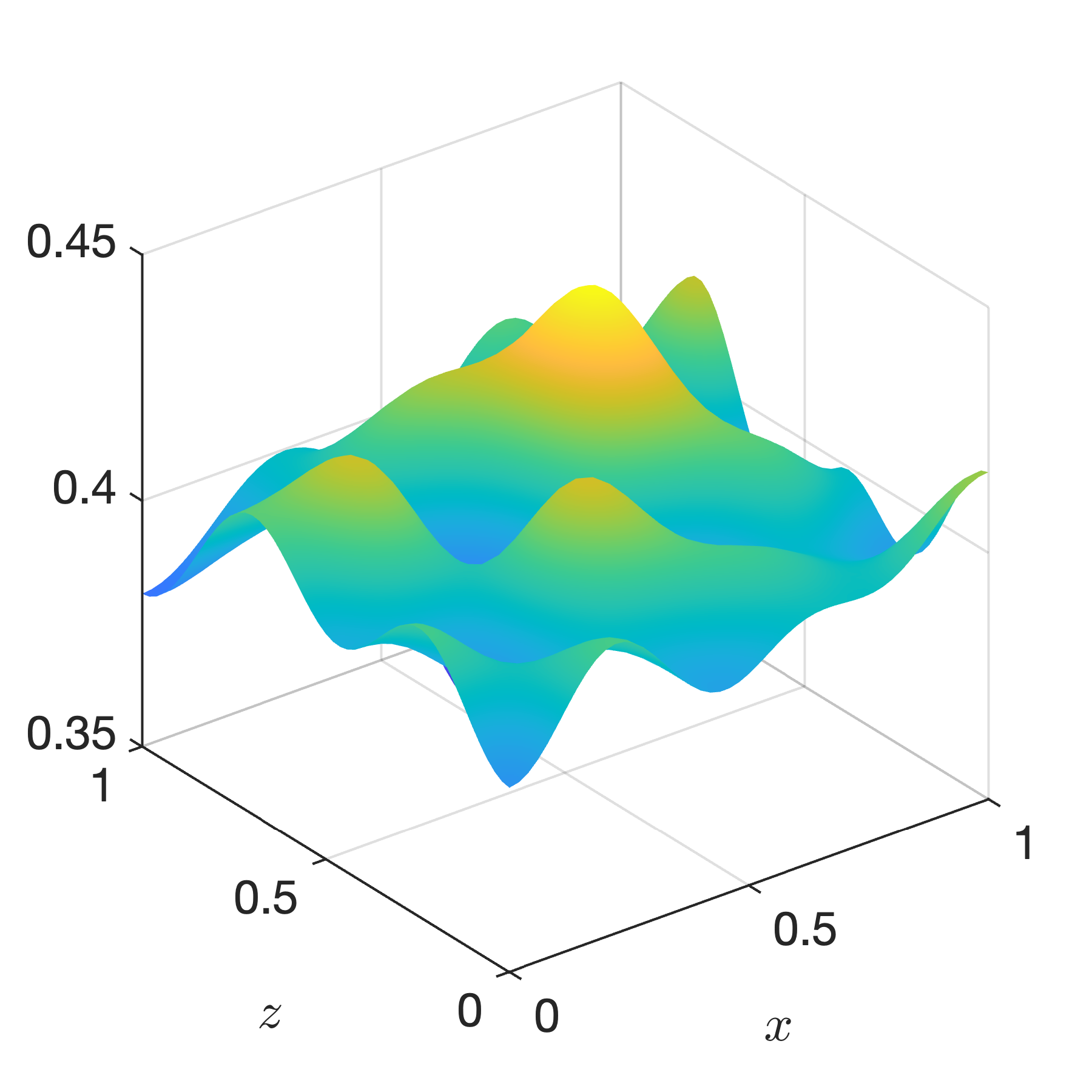}
\includegraphics[width=0.22\textwidth,trim=0.3cm 0.1cm 0.3cm 0.2cm,clip]{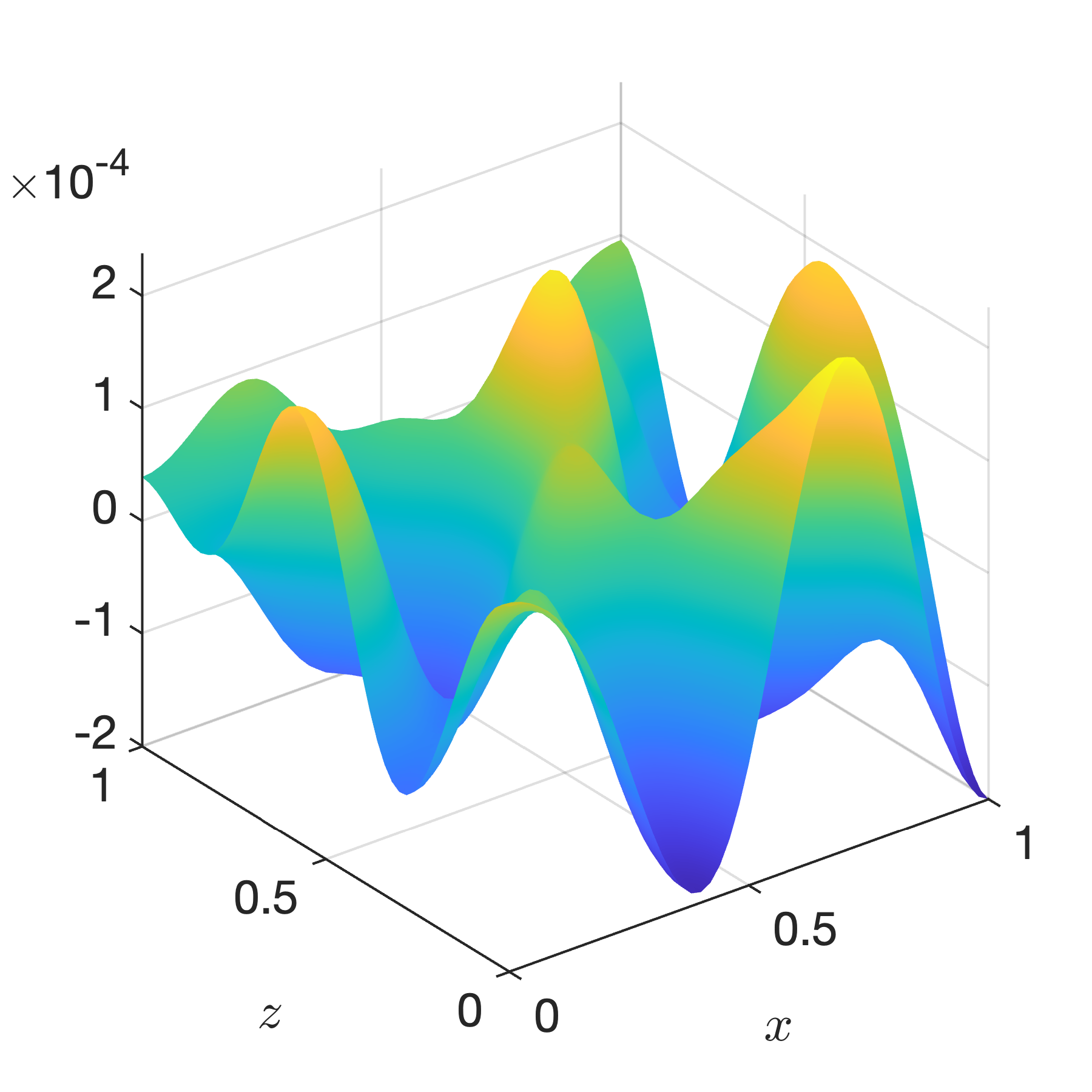}
\includegraphics[width=0.22\textwidth,trim=0.3cm 0.1cm 0.3cm 0.2cm,clip]{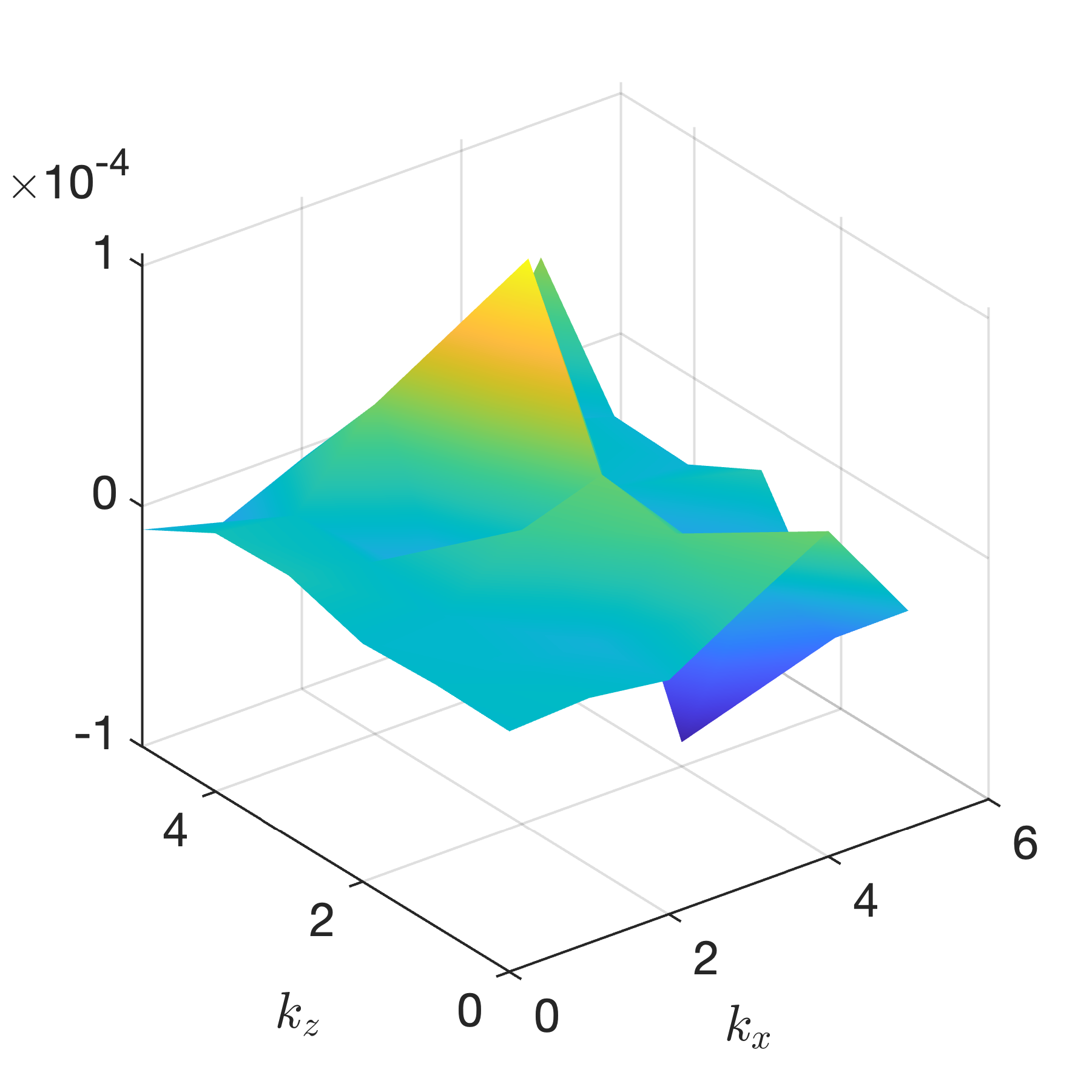}\\
 \includegraphics[width=0.22\textwidth,trim=0.2cm 0.1cm 0.3cm 0.2cm,clip]{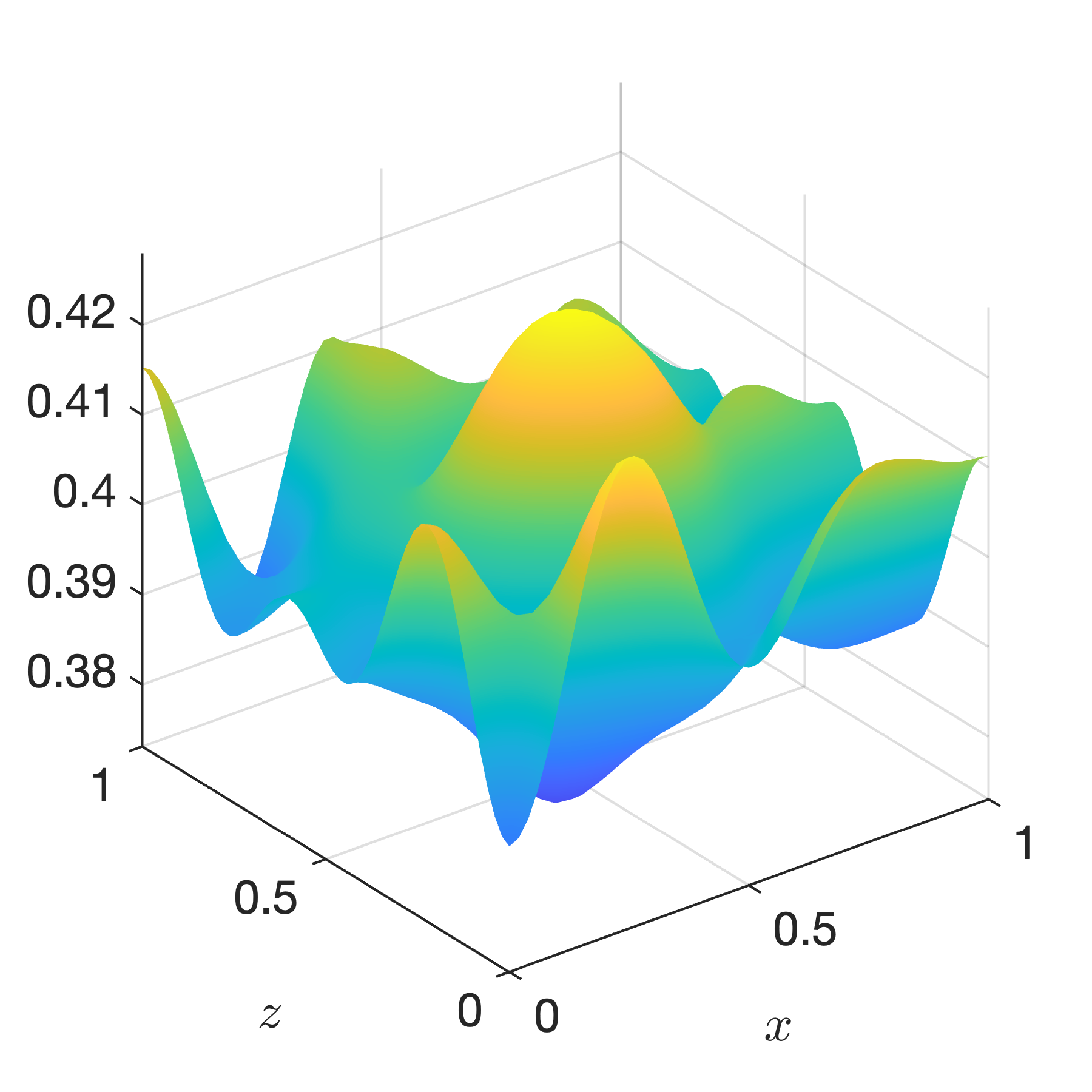}
\includegraphics[width=0.22\textwidth,trim=0.2cm 0.1cm 0.3cm 0.2cm,clip]{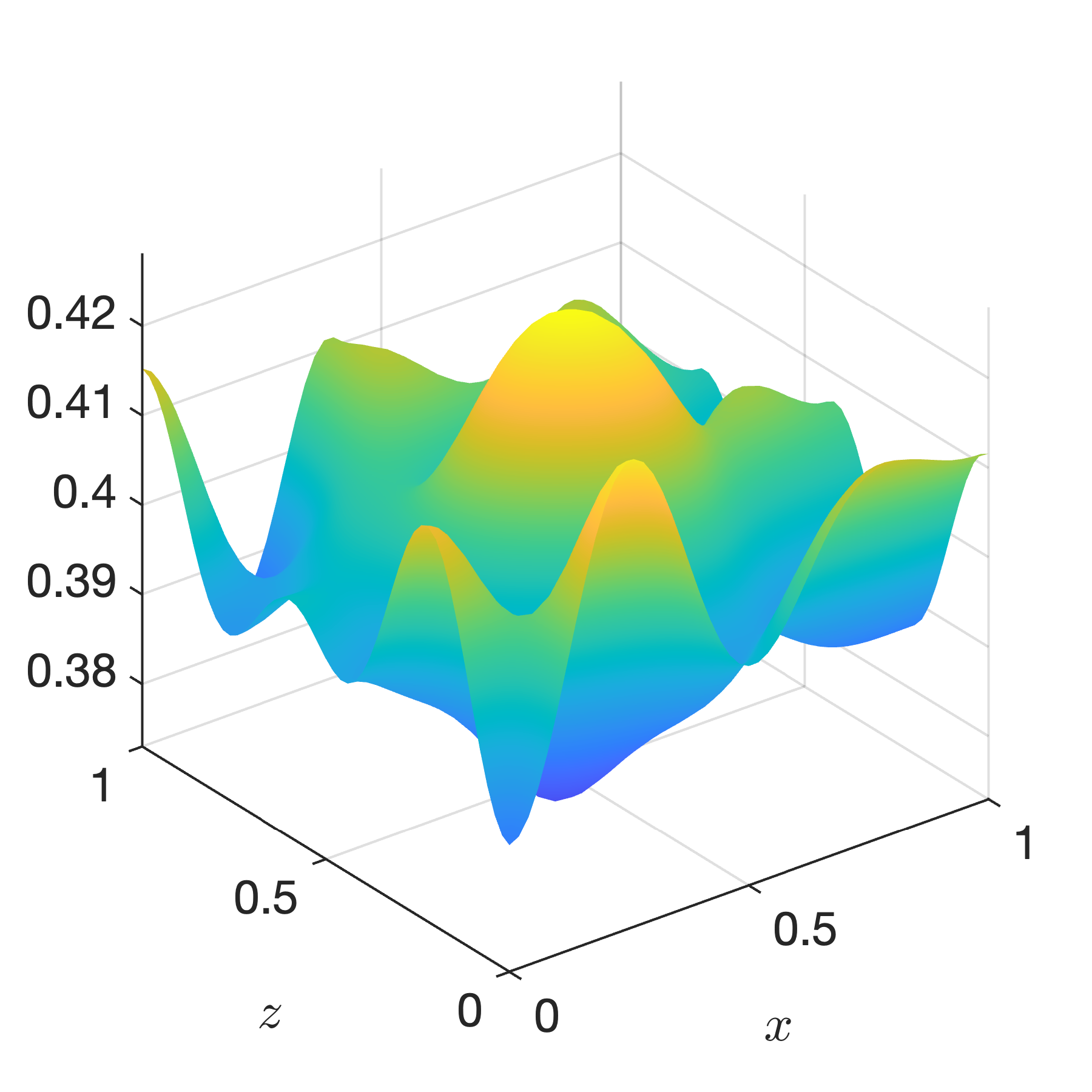}
\includegraphics[width=0.22\textwidth,trim=0.3cm 0.1cm 0.3cm 0.2cm,clip]{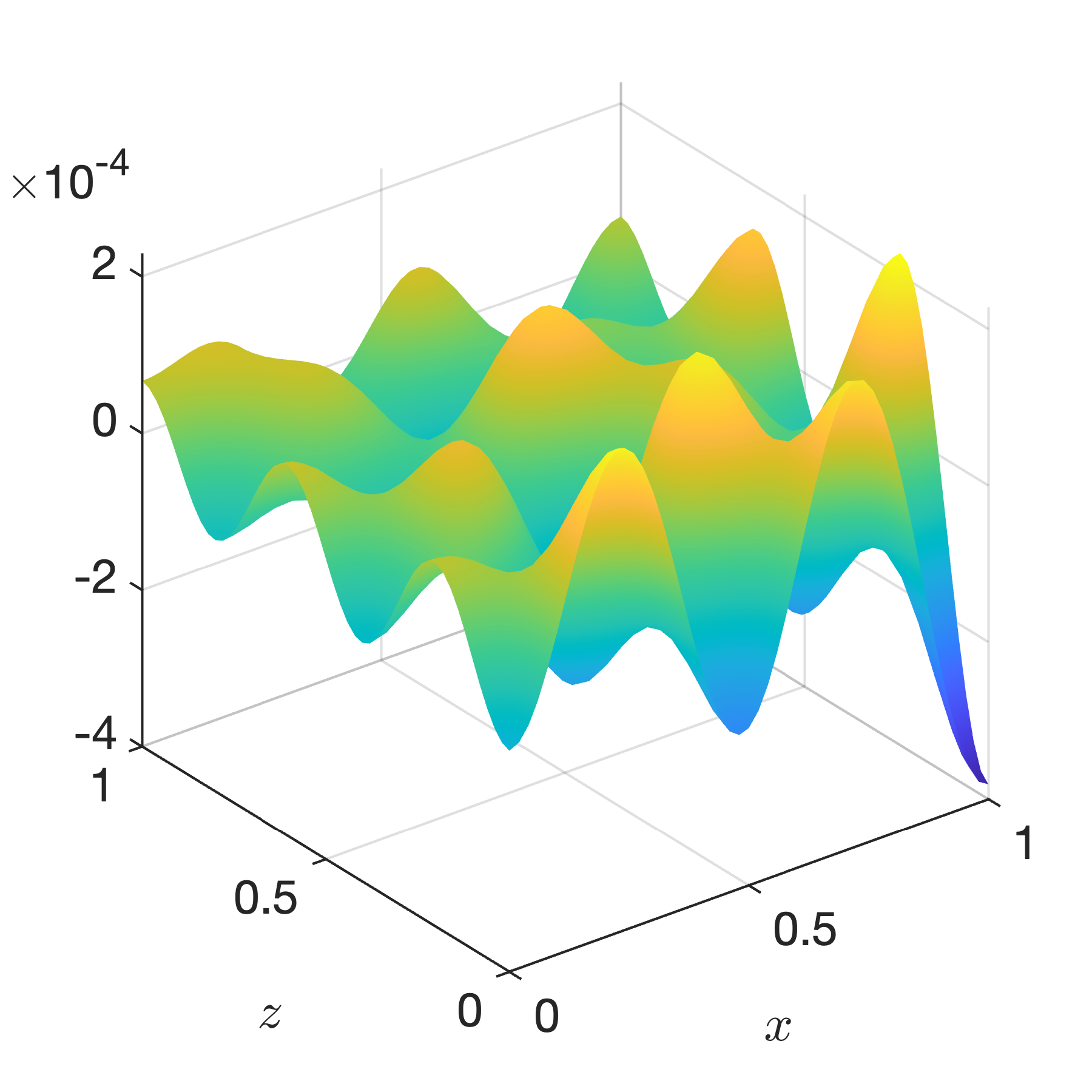}
\includegraphics[width=0.22\textwidth,trim=0.3cm 0.1cm 0.3cm 0.2cm,clip]{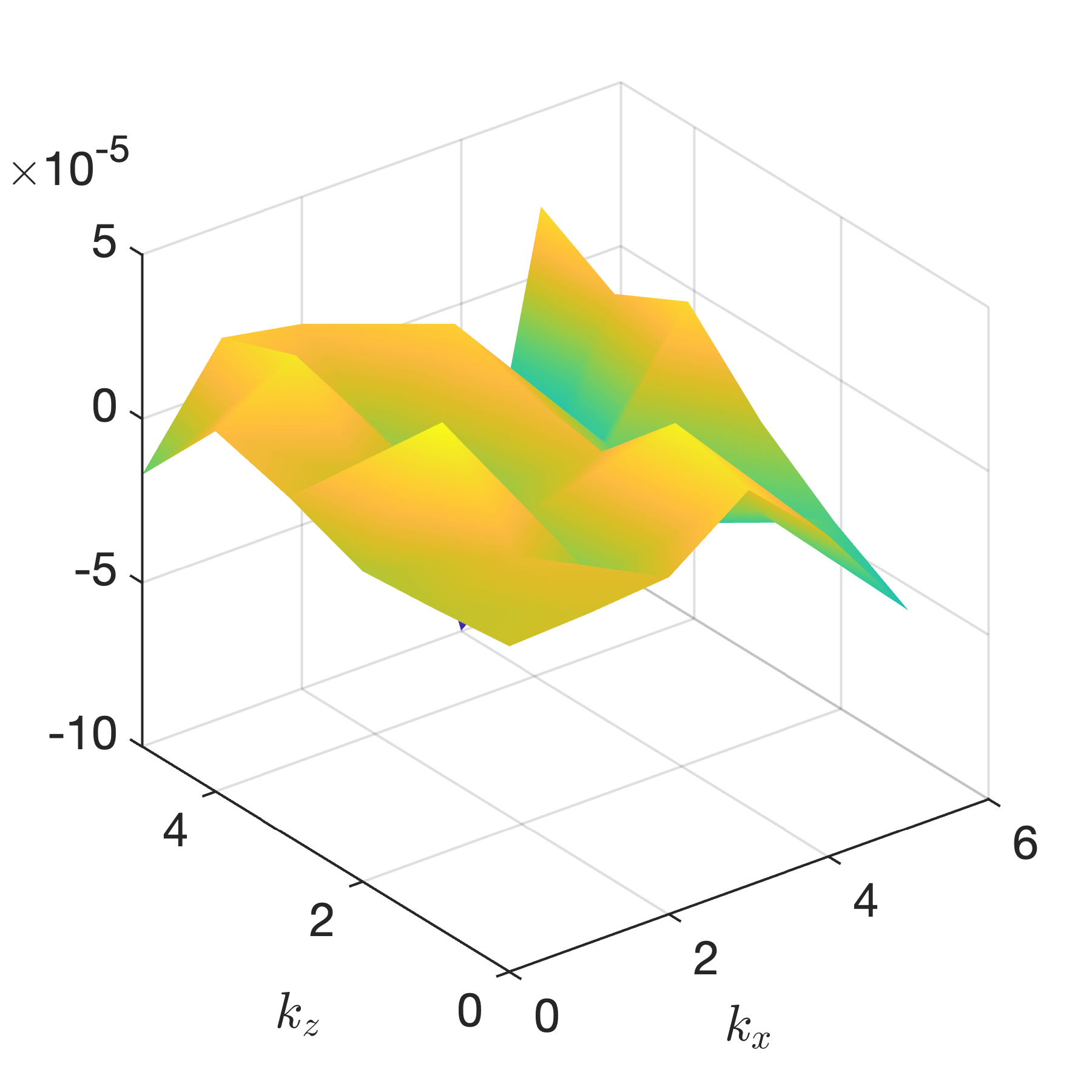}\\
\includegraphics[width=0.22\textwidth,trim=0.2cm 0.1cm 0.3cm 0.2cm,clip]{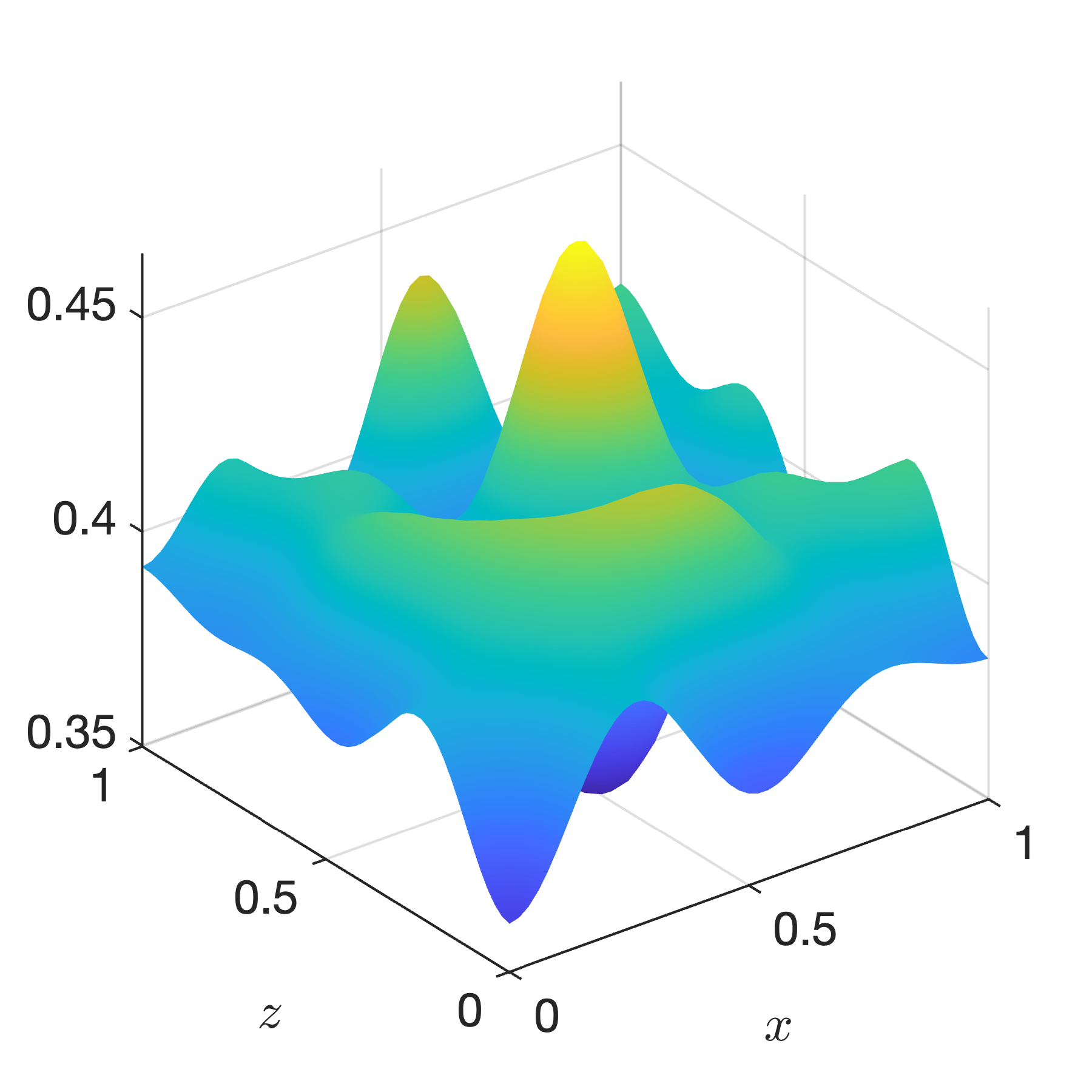}
\includegraphics[width=0.22\textwidth,trim=0.2cm 0.1cm 0.3cm 0.2cm,clip]{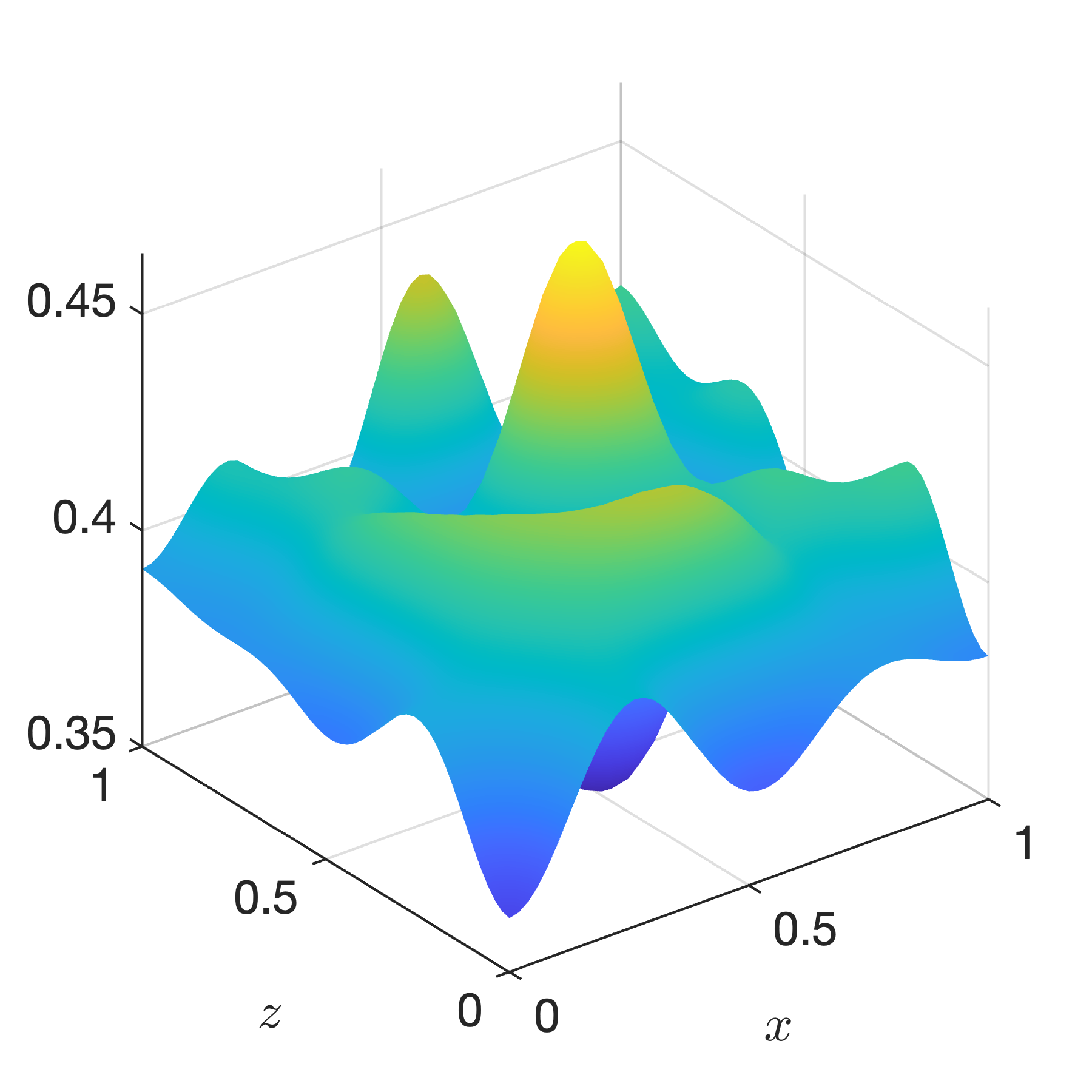}
\includegraphics[width=0.22\textwidth,trim=0.3cm 0.1cm 0.3cm 0.2cm,clip]{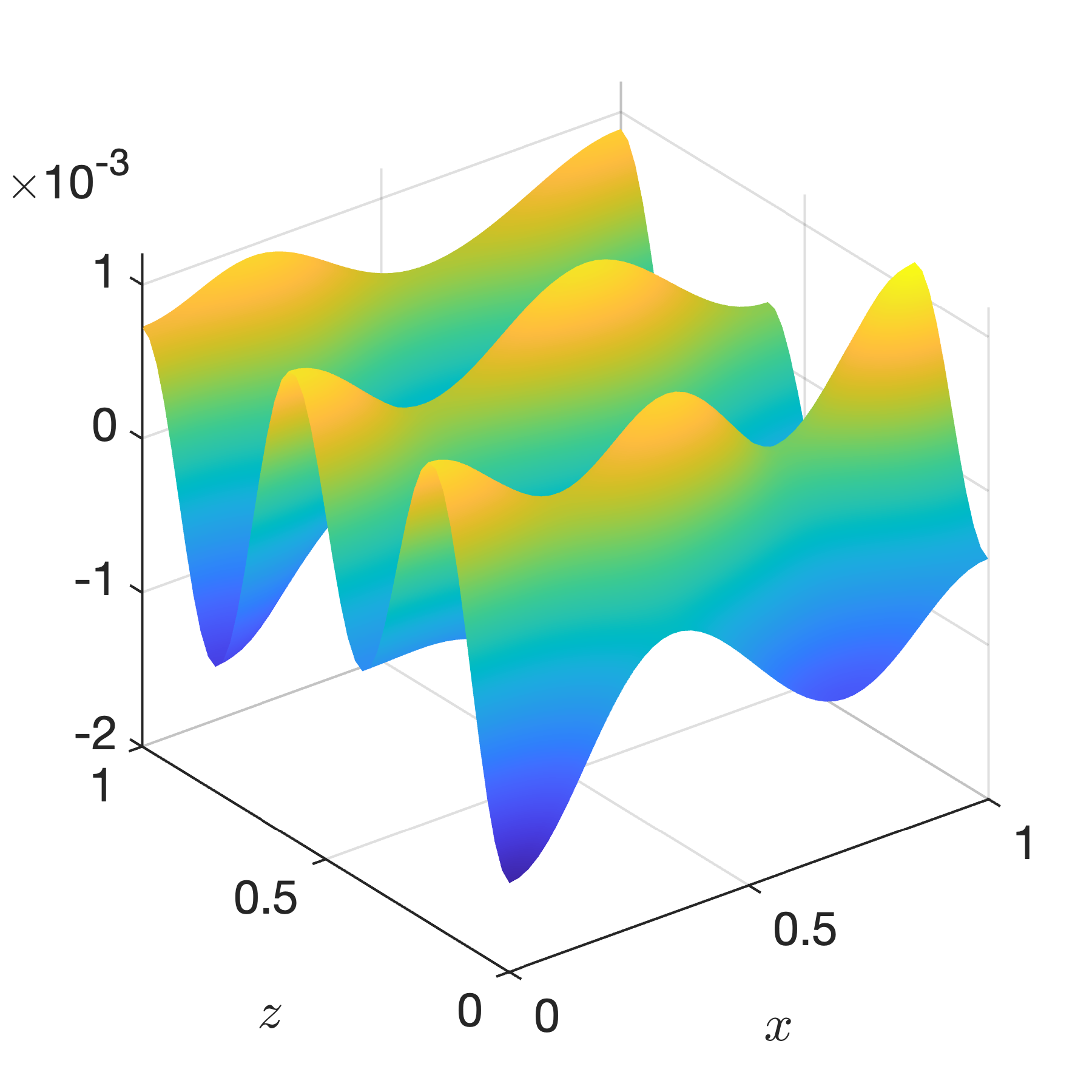}
\includegraphics[width=0.22\textwidth,trim=0.3cm 0.1cm 0.3cm 0.2cm,clip]{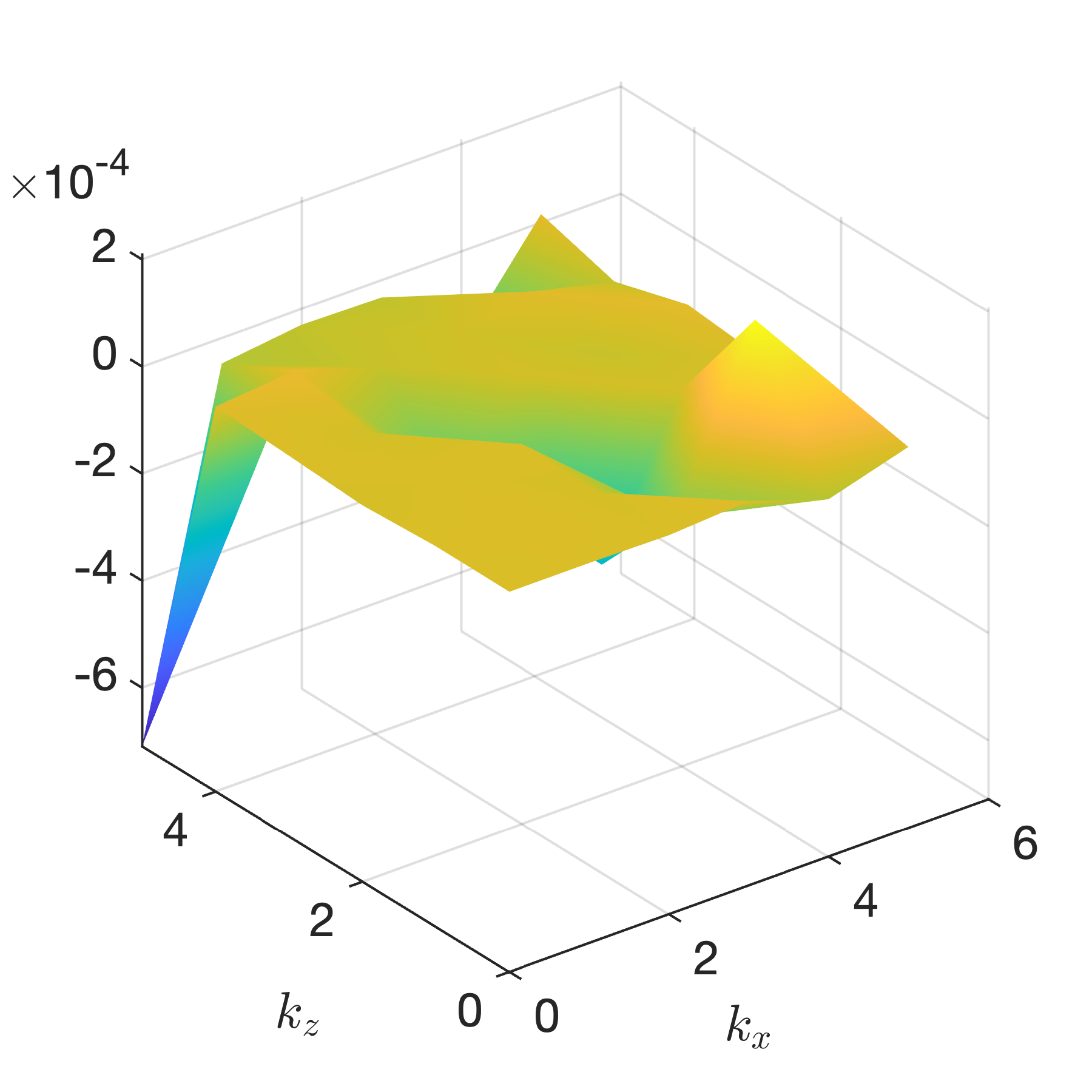}
\caption{Three randomly selected absorption fields ($\sigma$) from the testing dataset: $6 \times 6$-coefficient Fourier model in Numerical Experiment II. From left to right are the exact absorption field (column $1$), the predicted absorption field by the trained neural network (column $2$), the error of the prediction in the physical space (column $3$), and the error in the network prediction in the Fourier domain (column $4$). 
}
	\label{FIG:absorption_nn_samples}
\end{figure}
In Figure~\ref{FIG:absorption_nn_samples}, we show three randomly selected absorption fields $\sigma$ from the testing dataset, which are generated by the diffusion fields $\gamma$ based on~\eqref{EQ:Relation 1}, the corresponding neural network predictions, and the errors in the prediction. We also present the training-validation loss curves in Figure~\ref{FIG:loss_validation_curve_fourier}. To save the presentation space, we only show the training-validation loss curves for the Fourier modes $\wh\sigma_{(0,0)}$, $\wh\sigma_{(3,3)}$, $\wh\sigma_{(5,5)}$, $\wh\sigma_{(1,2)}$, $\wh\sigma_{(3,4)}$, and $\wh\sigma_{(5,3)}$ but we observe very similar curves for other Fourier modes. A quick takeaway from Figure~\ref{FIG:absorption_nn_samples} and Figure~\ref{FIG:loss_validation_curve_fourier} is that the training process is quite successful as the testing errors are pretty reasonable, especially given that our training dataset is very small, with only $0.8 \times 10^4$ data points. \RED{Our interpretation is that the close alignment of the training and validation curves suggests that the model generalizes well to the validation set and indicates that there is no significant overfitting. The large fluctuations observed in both curves are likely due to optimization dynamics. In this work, our goal is to learn an approximate relationship between the parameters rather than an exact one. While it is possible to reduce training and validation errors, as well as potentially mitigate the observed oscillations, by using more advanced optimization algorithms and fine-tuning hyperparameters, we did not pursue this here because our focus remains on capturing a reasonable approximation of the underlying relationship.}

\begin{figure}[htb!]
	\centering
 \includegraphics[width=0.325\textwidth,trim=0cm 0cm 0cm 0cm,clip]{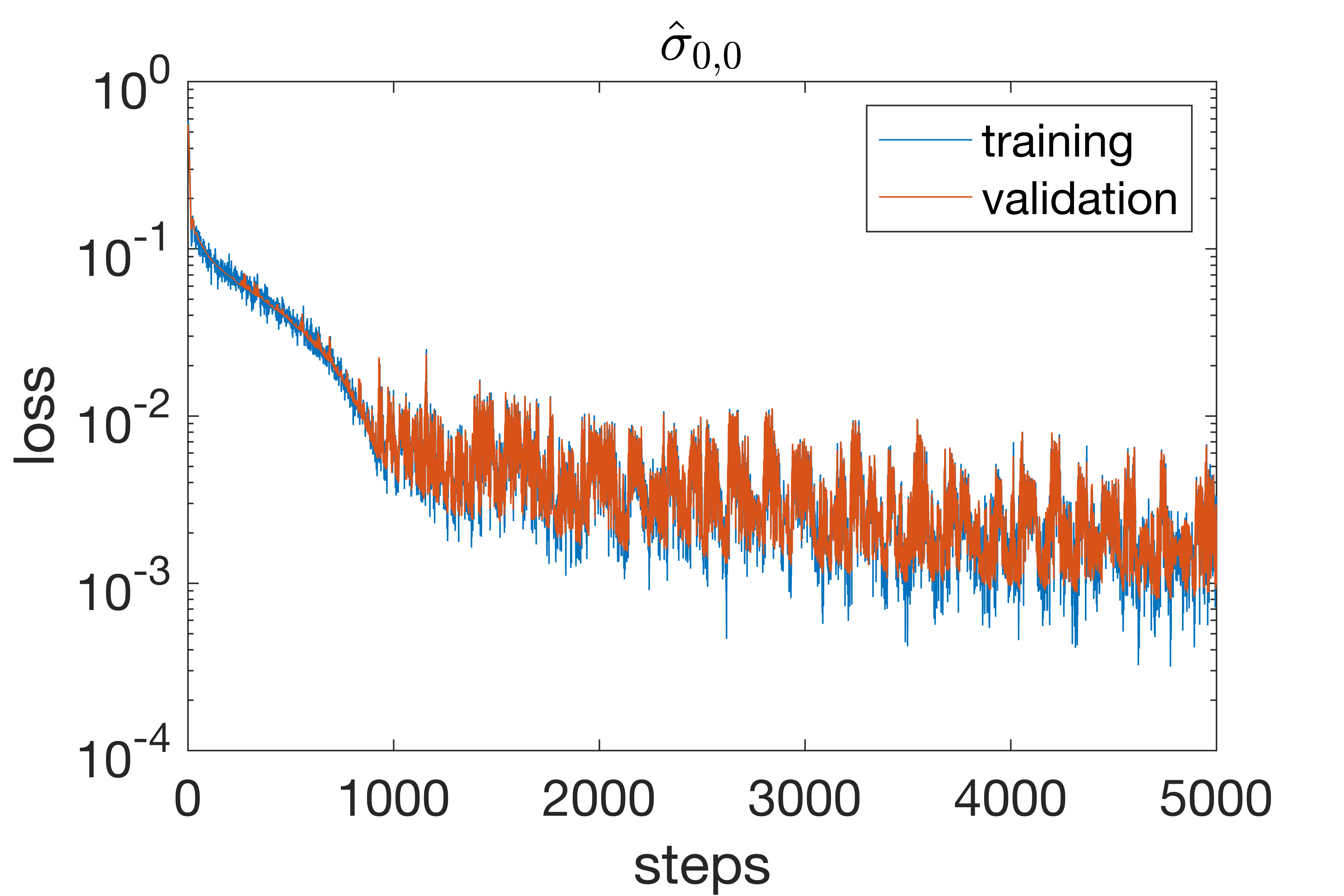}
 \includegraphics[width=0.325\textwidth,trim=0cm 0cm 0cm 0cm,clip]{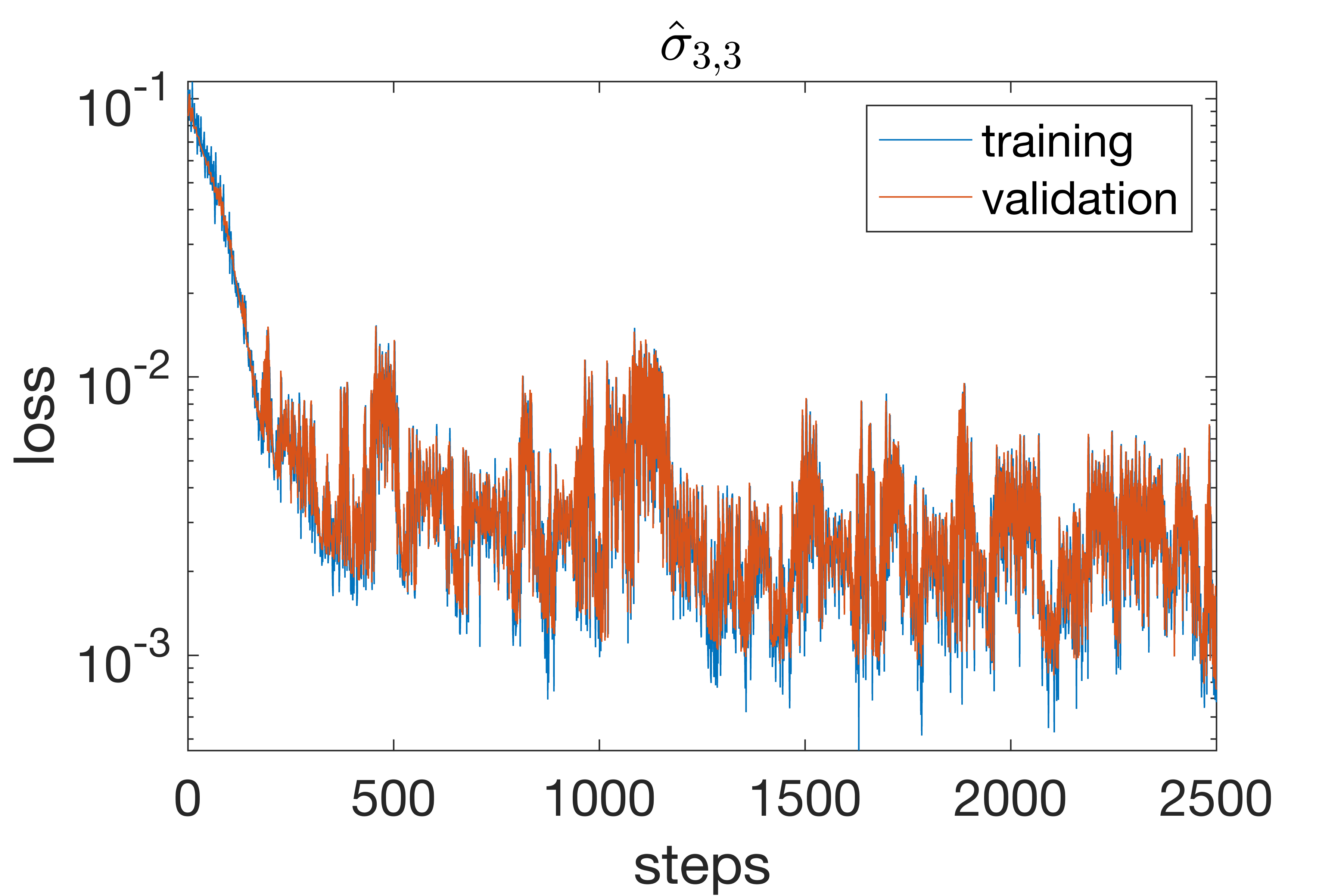}
 \includegraphics[width=0.325\textwidth,trim=0cm 0cm 0cm 0cm,clip]{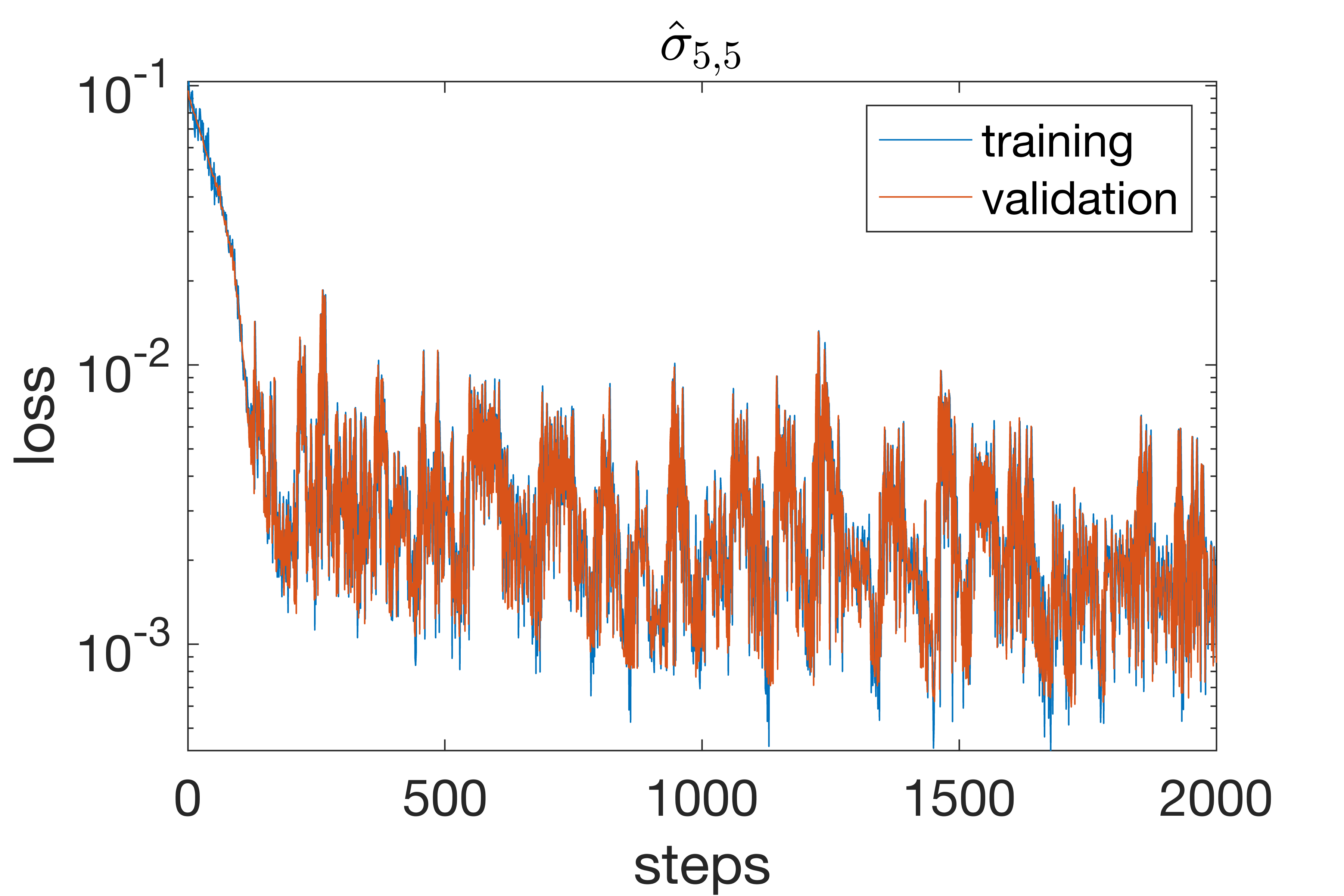}\\
 \includegraphics[width=0.325\textwidth,trim=0cm 0cm 0cm 0cm,clip]{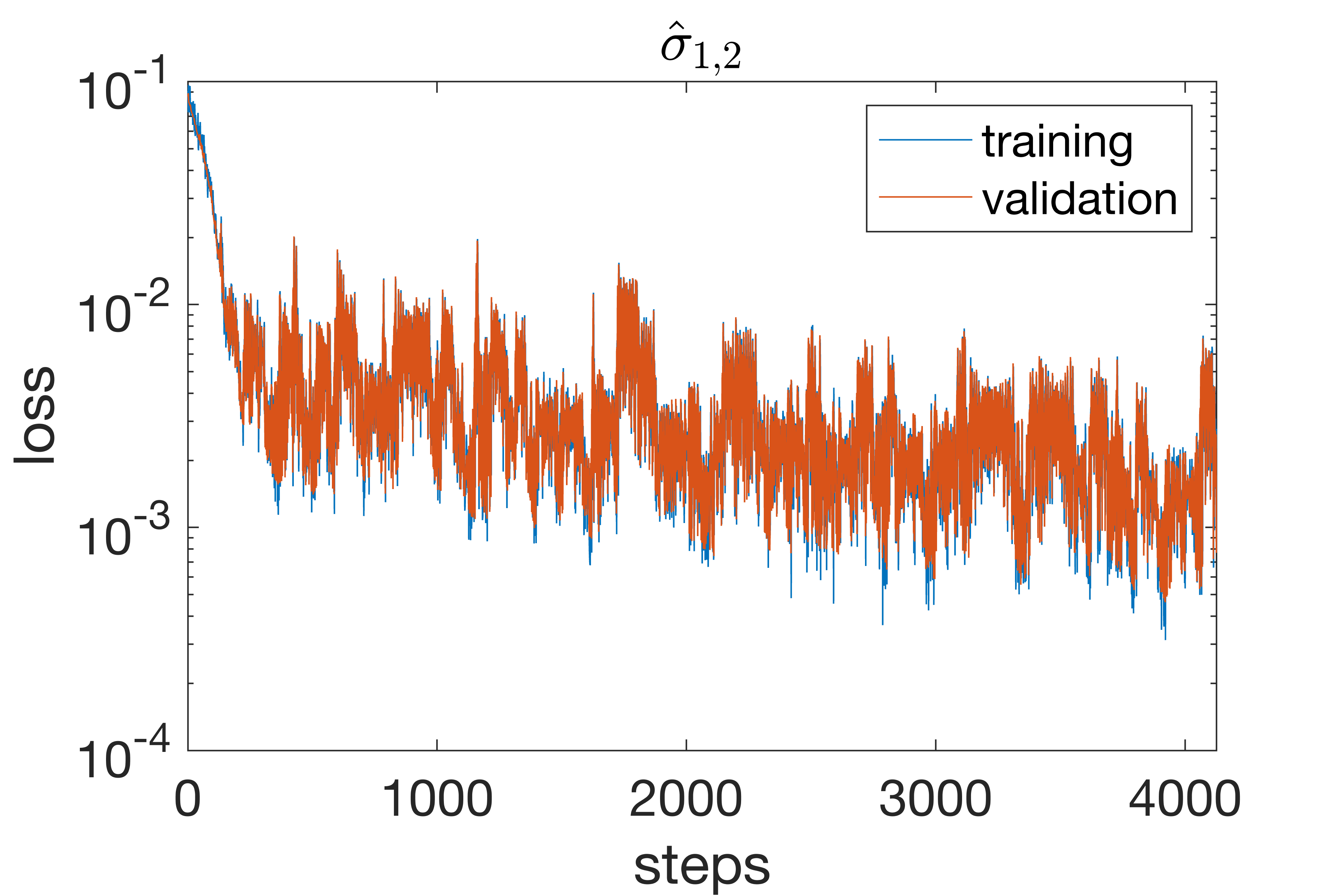}
 \includegraphics[width=0.325\textwidth,trim=0cm 0cm 0cm 0cm,clip]{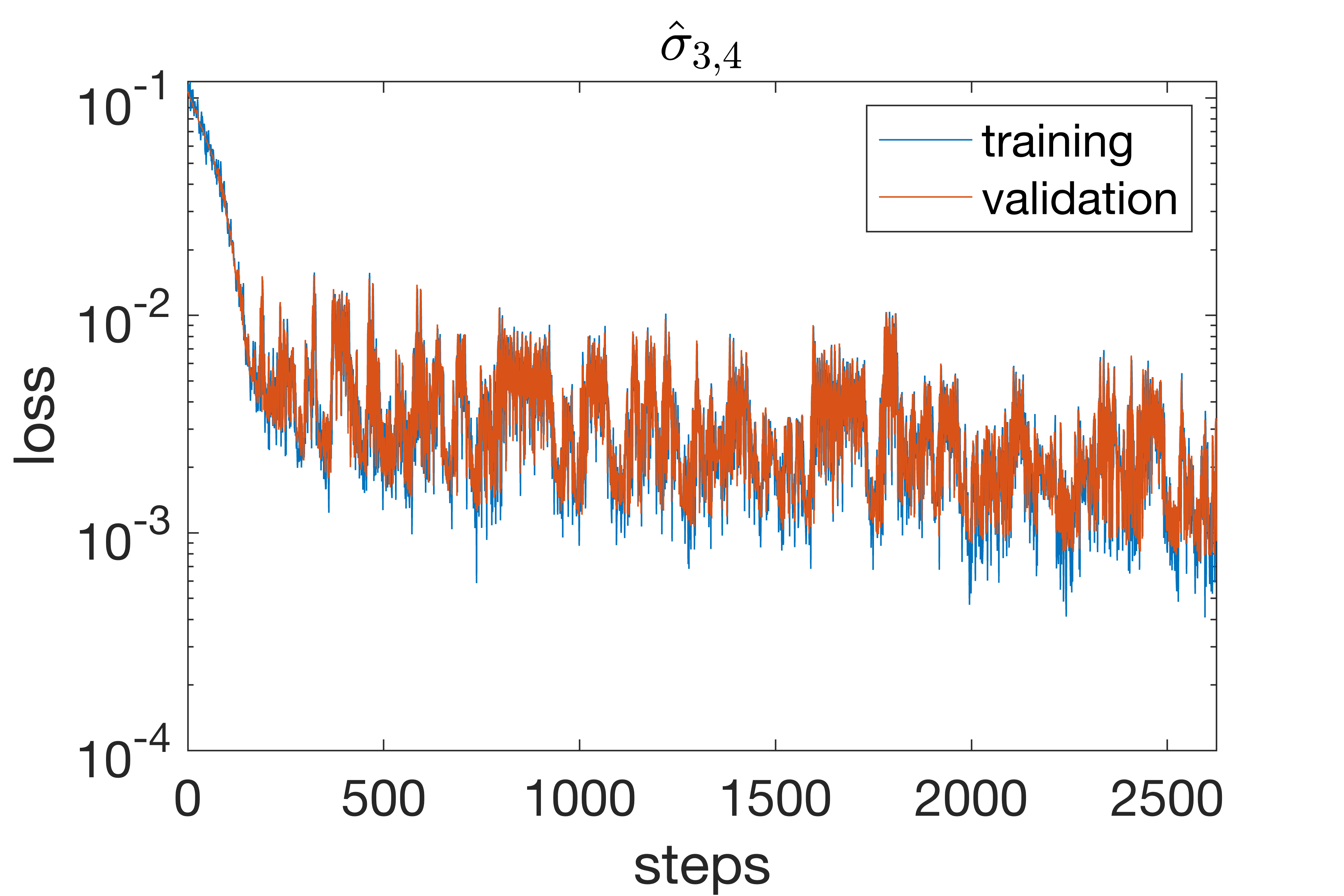}
 \includegraphics[width=0.325\textwidth,trim=0cm 0cm 0cm 0cm,clip]{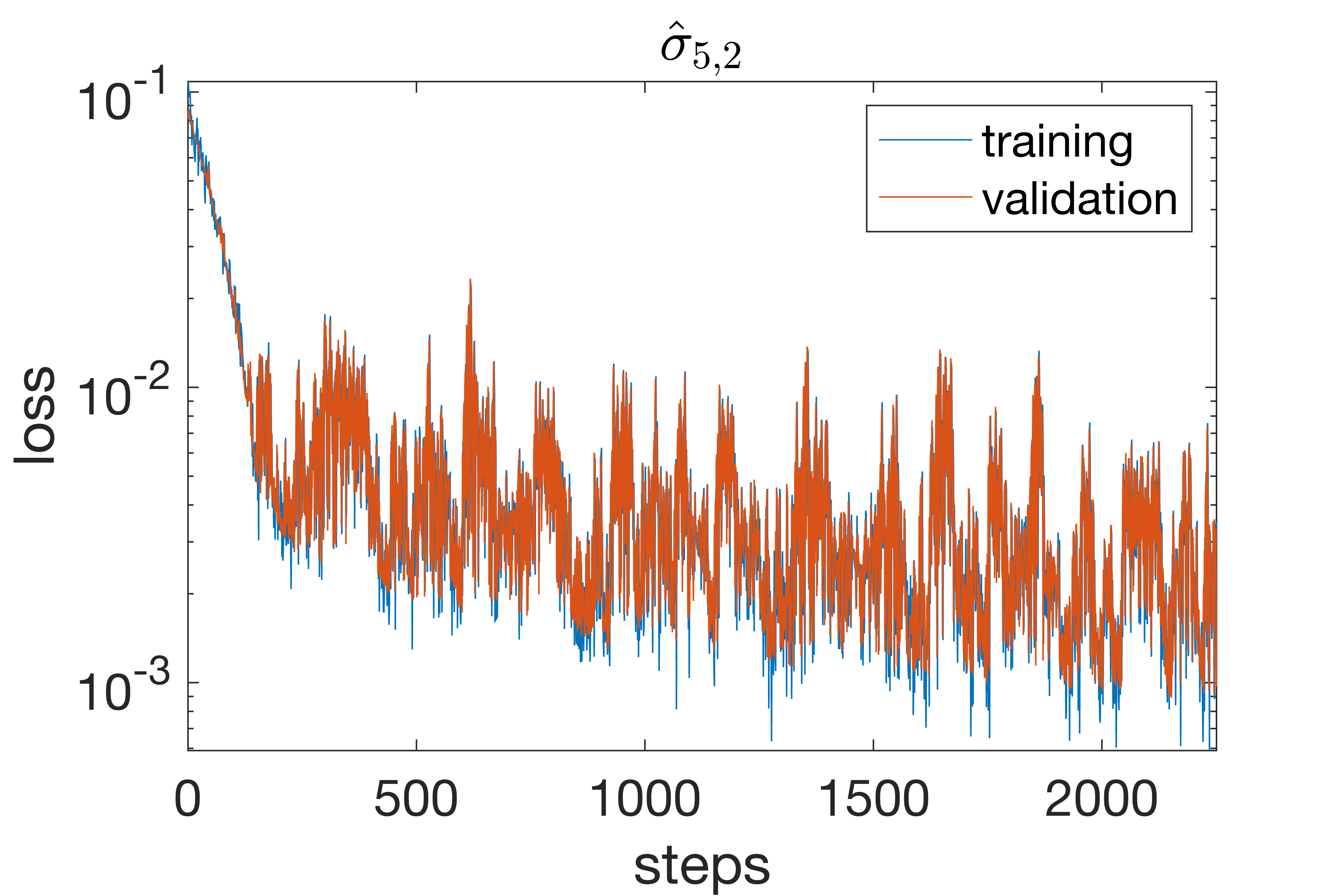}
    \caption{From top left to bottom right: the training (blue) and validation (brown) loss curves for the Fourier modes $\wh\sigma_{(0,0)}$, $\wh\sigma_{(3,3)}$, $\wh\sigma_{(5,5)}$, $\wh\sigma_{(1,2)}$, $\wh\sigma_{(3,4)}$, and $\wh\sigma_{(5,2)}$ in Numerical Experiment II.}
    \label{FIG:loss_validation_curve_fourier}
\end{figure}

%%%%%%%%%%%%%%%%%%%%%%
%\subsubsection{Data-driven joint inversion}
%%%%%%%%%%%%%%%%%%%%%%
\noindent\textbf{Data-driven joint inversion.} We now perform joint reconstructions of $(\gamma, \sigma)$ based on the learned relation from internal data~\eqref{EQ:Diff Data} (see Figure~\ref{FIG:usigma_data} for the datum $H$ generated with the boundary source~\eqref{source_top_diffusion_example} for this example). The algorithmic parameters are the same as in the previous example.
\begin{figure}[htb!]
	\centering	
 \includegraphics[width=0.30\textwidth,trim=1cm 1.5cm 0.5cm 1.5cm,clip]{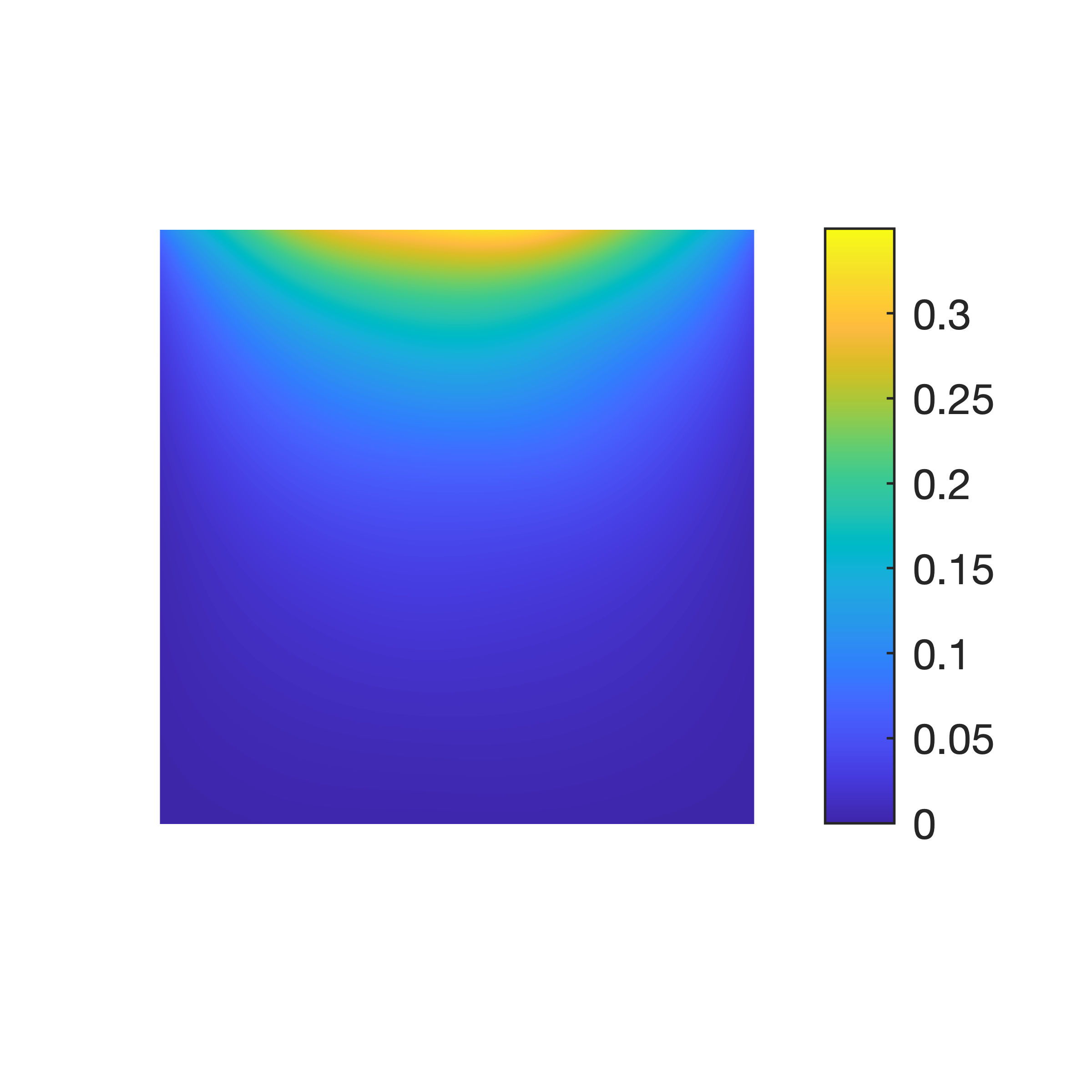}
 \includegraphics[width=0.30\textwidth,trim=1cm 1.5cm 0.5cm 1.5cm,clip]{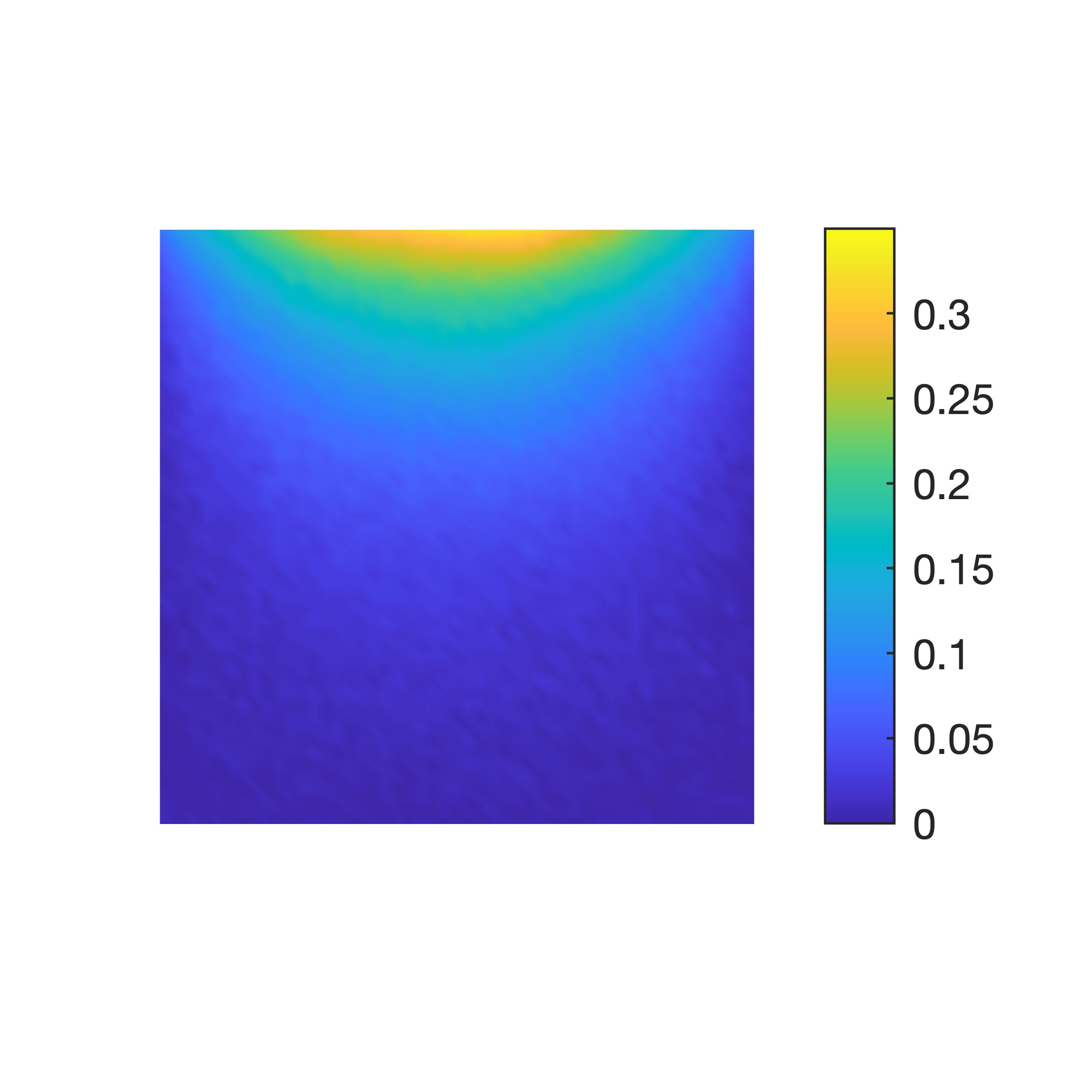}
 \includegraphics[width=0.3\textwidth,trim=1cm 1.5cm 0.5cm 1.5cm,clip]{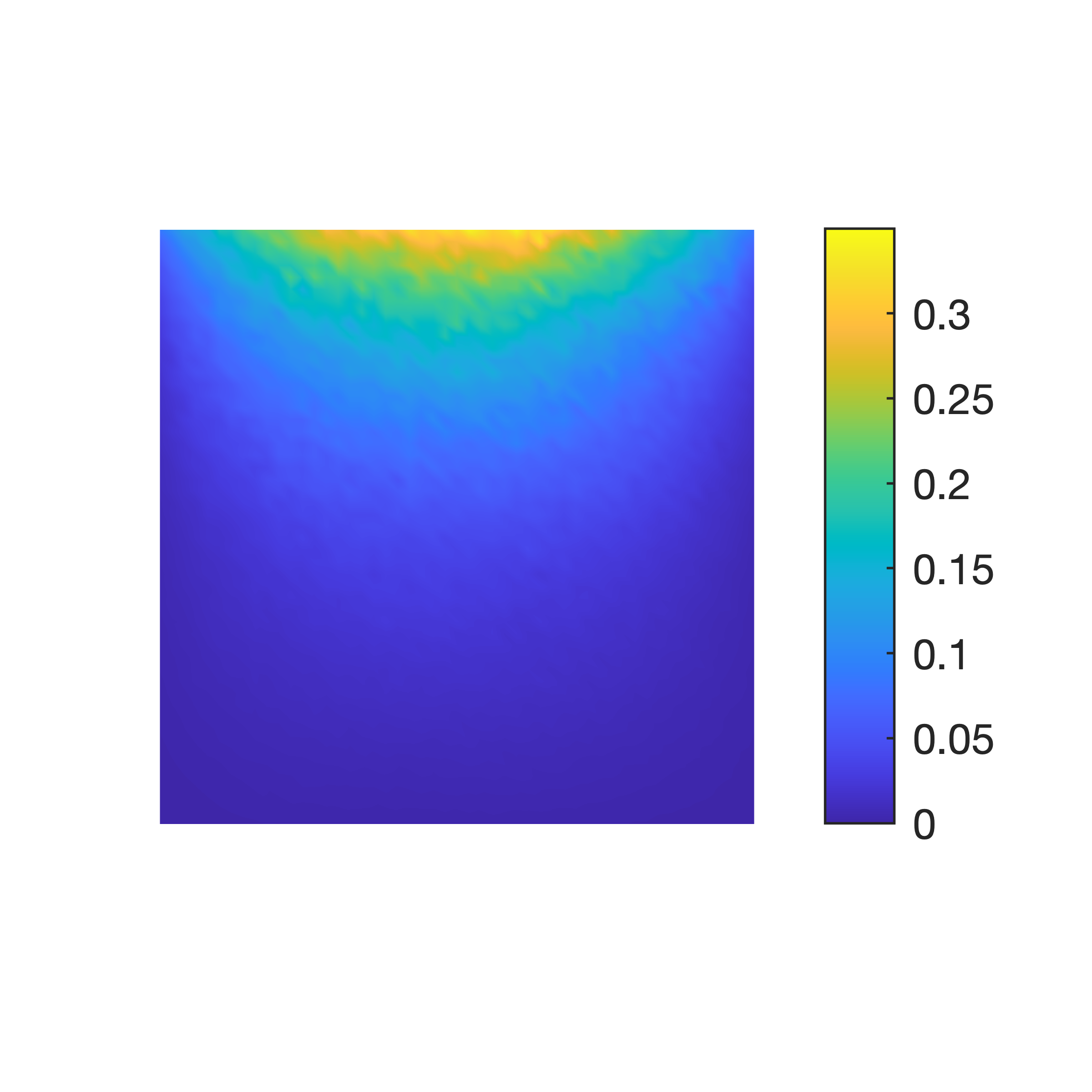}
	\caption{Internal datum $H$ generated with illumination~\eqref{source_top_diffusion_example}. From left to right: noise-free datum, datum with $5\%$ additive Gaussian noise, and datum with $5\%$ multiplication Gaussian noise.}
	\label{FIG:usigma_data}
\end{figure}

\RED{In Figure~\ref{FIG:reconstructed_diffusion_absorption}, we show the surface plots of the exact and the reconstructed $\gamma$ and $\sigma$  fields from the proposed data-driven joint inversion scheme. Specifically, we display the diffusion fields in the top panel, and the absorption fields are plotted in the bottom panel. From the left to the right are the surface plots of the true coefficients  (first column); the reconstructed diffusion/absorption fields with noise-free internal data (second column), with the internal data containing $5\%$ additive Gaussian noise (third column), and with the internal data including $5\%$ multiplication Gaussian noise (fourth column). We observe that the accuracy of the reconstruction of the proposed joint inversion scheme is pretty well. Even with the noisy internal data, we can recover the main features of the diffusion coefficient $\gamma$ and the absorption field $\sigma$.}
\begin{figure}[htb!]
	\centering
 \includegraphics[width=0.22\textwidth,trim=1cm 1.5cm 0.5cm 1.5cm,clip]{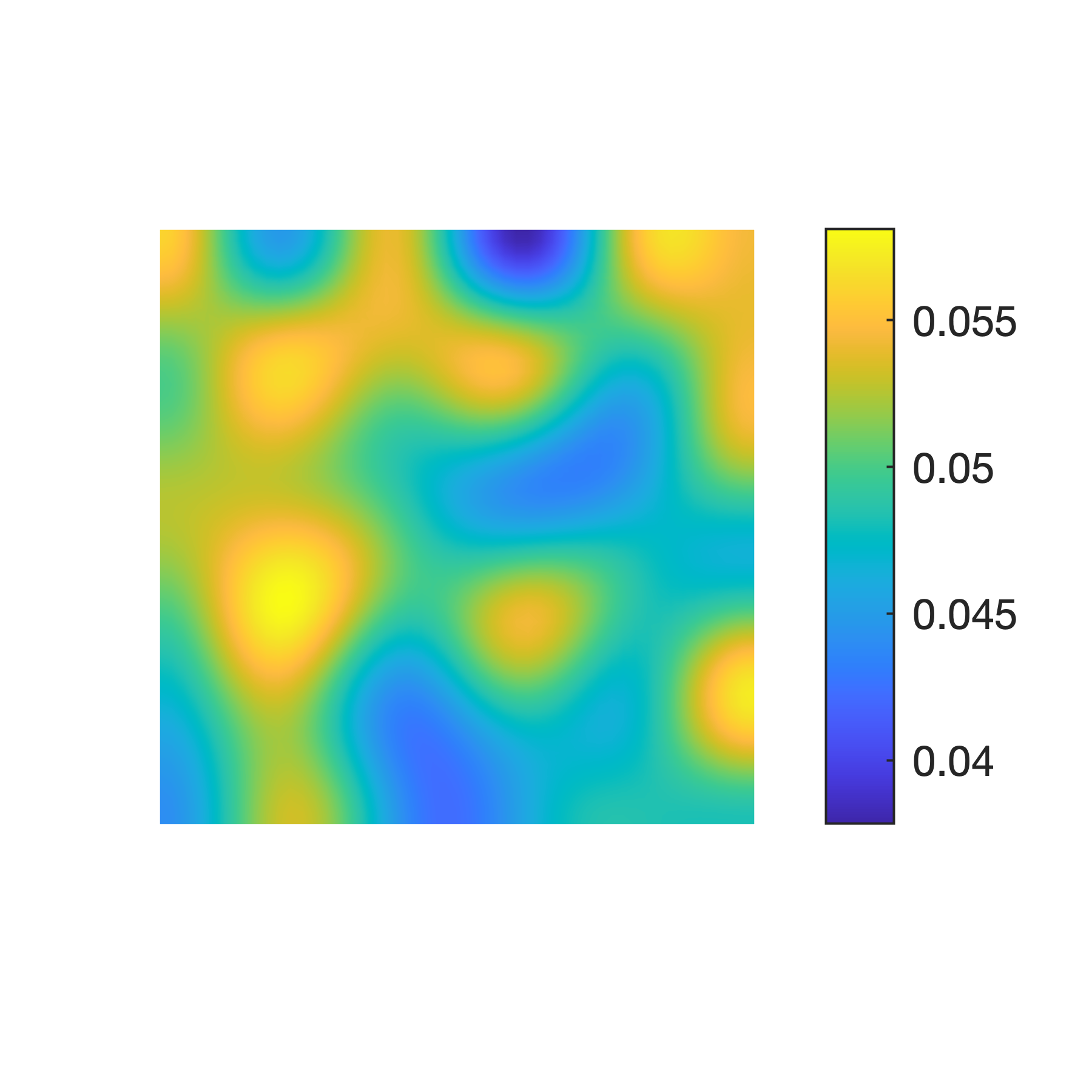}
 \includegraphics[width=0.22\textwidth,trim=1cm 1.5cm 0.5cm 1.5cm,clip]{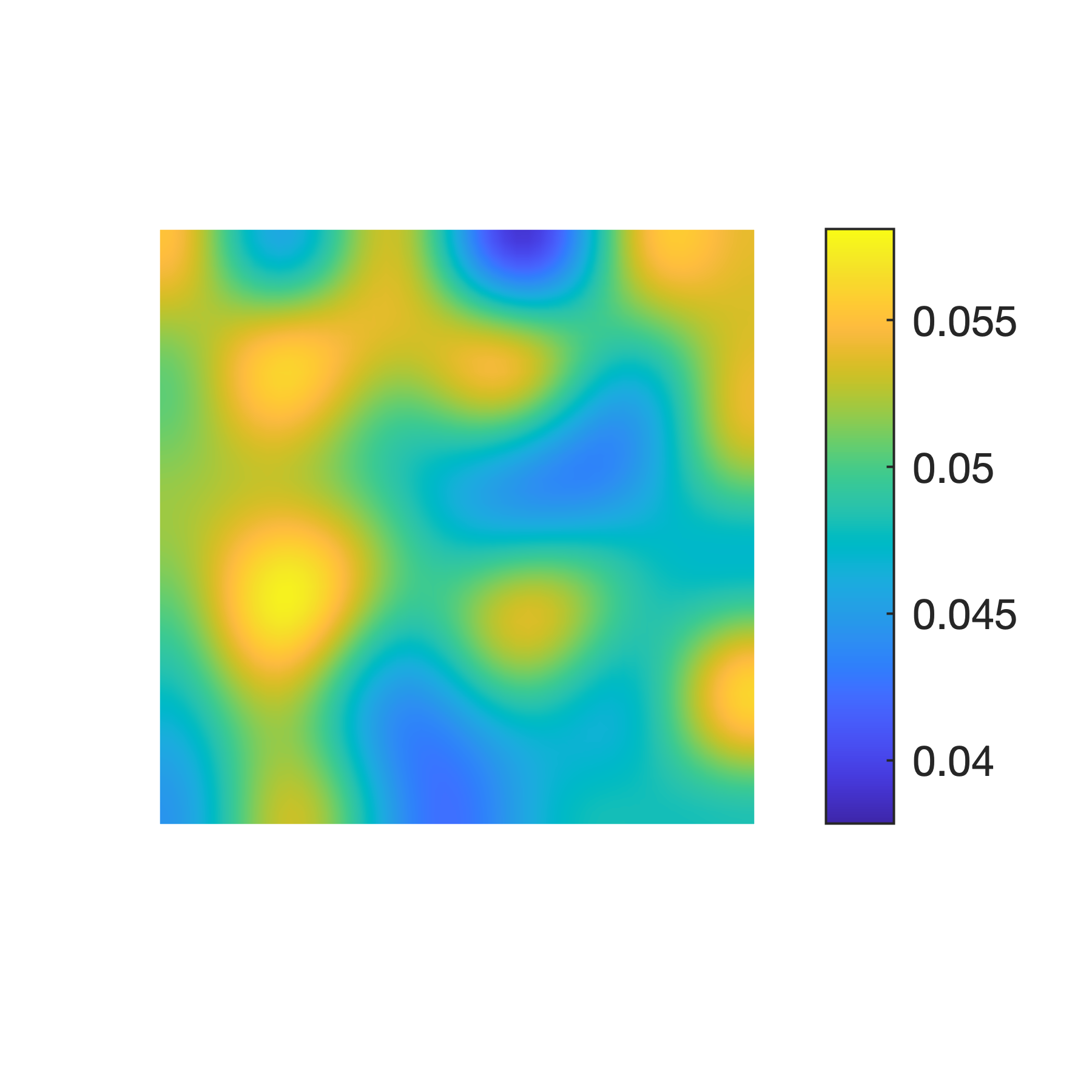}
 \includegraphics[width=0.22\textwidth,trim=1cm 1.5cm 0.5cm 1.5cm,clip]{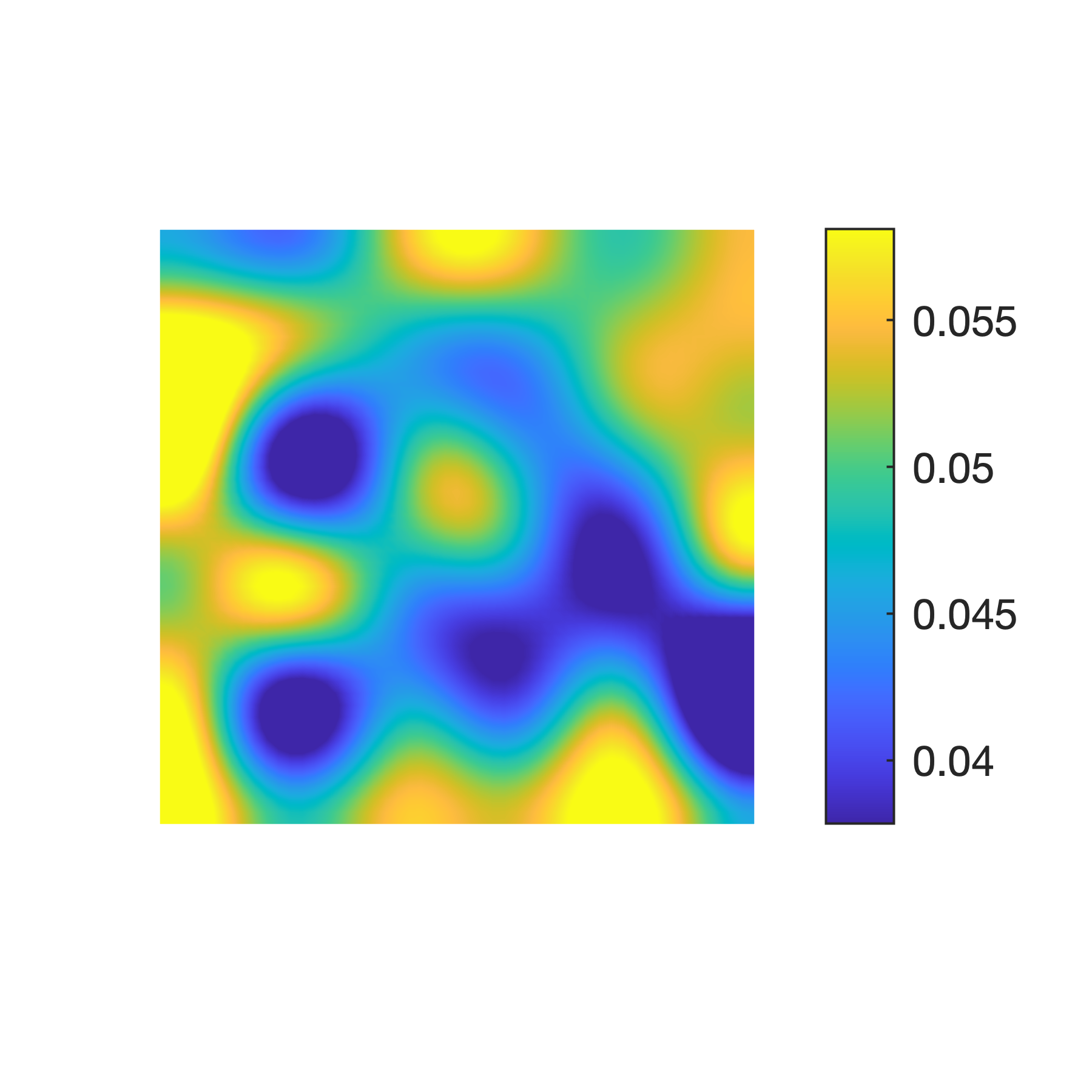}
 \includegraphics[width=0.22\textwidth,trim=1cm 1.5cm 0.5cm 1.5cm,clip]{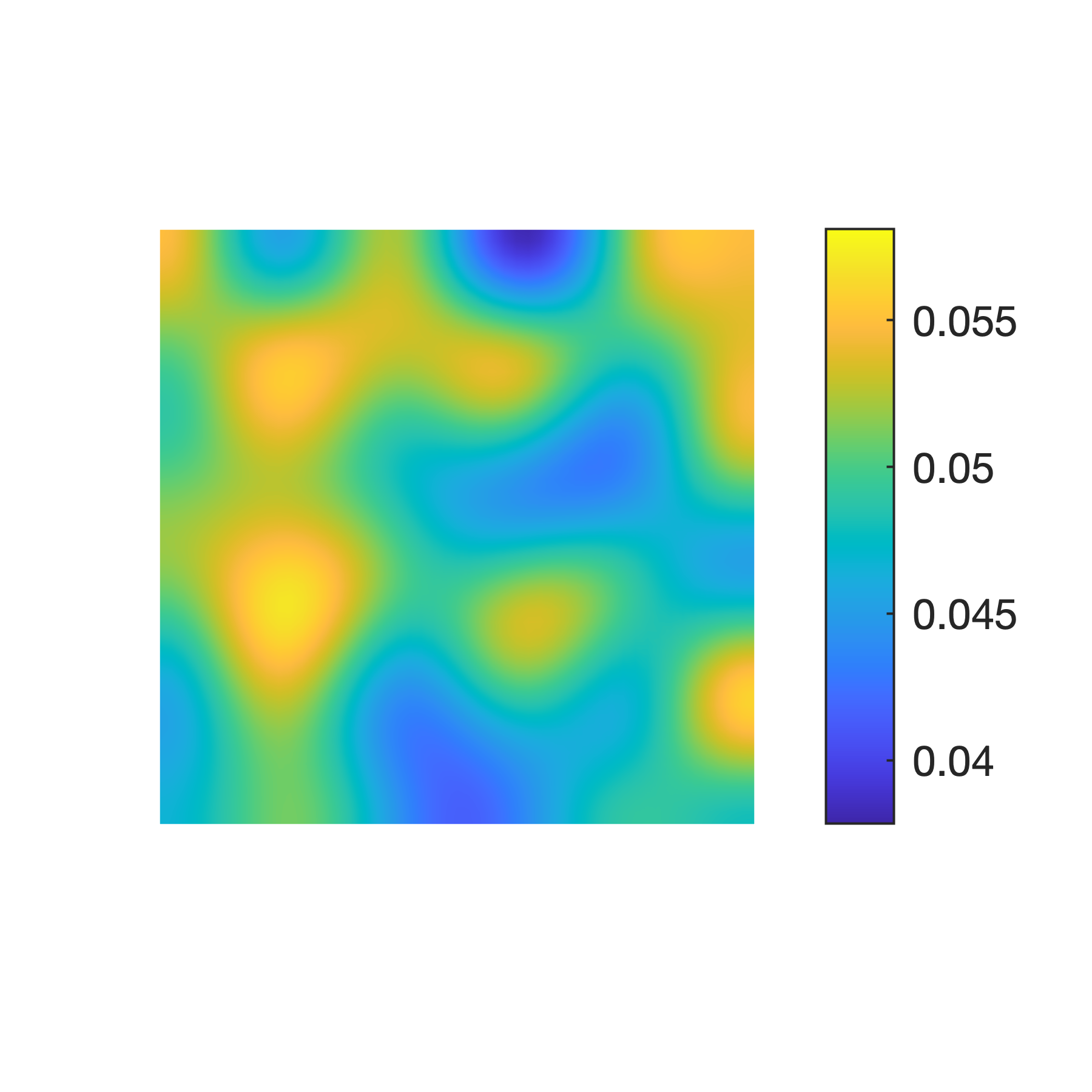}\\
 \includegraphics[width=0.22\textwidth,trim=1cm 1.5cm 0.5cm 1.5cm,clip]{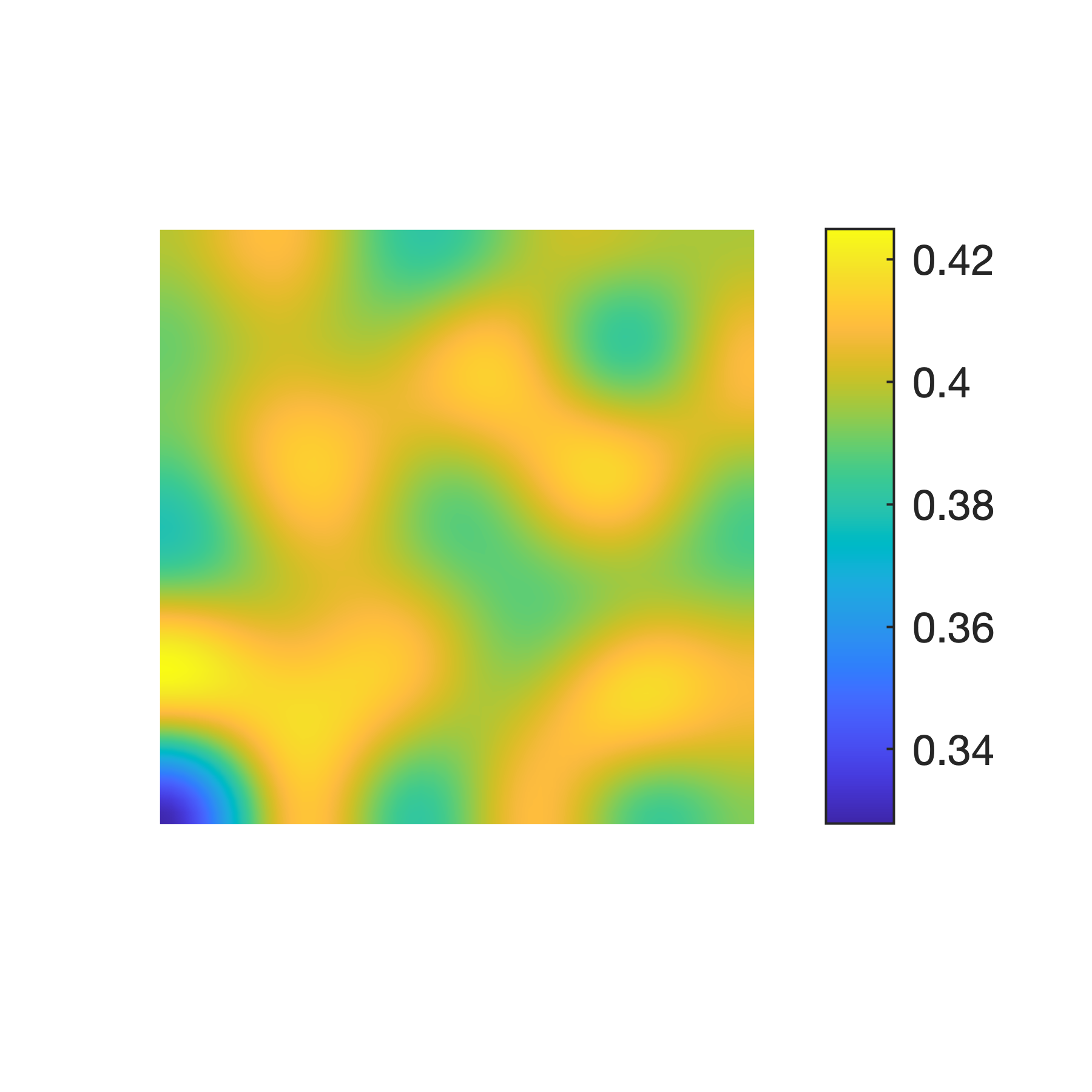}
 \includegraphics[width=0.22\textwidth,trim=1cm 1.5cm 0.5cm 1.5cm,clip]{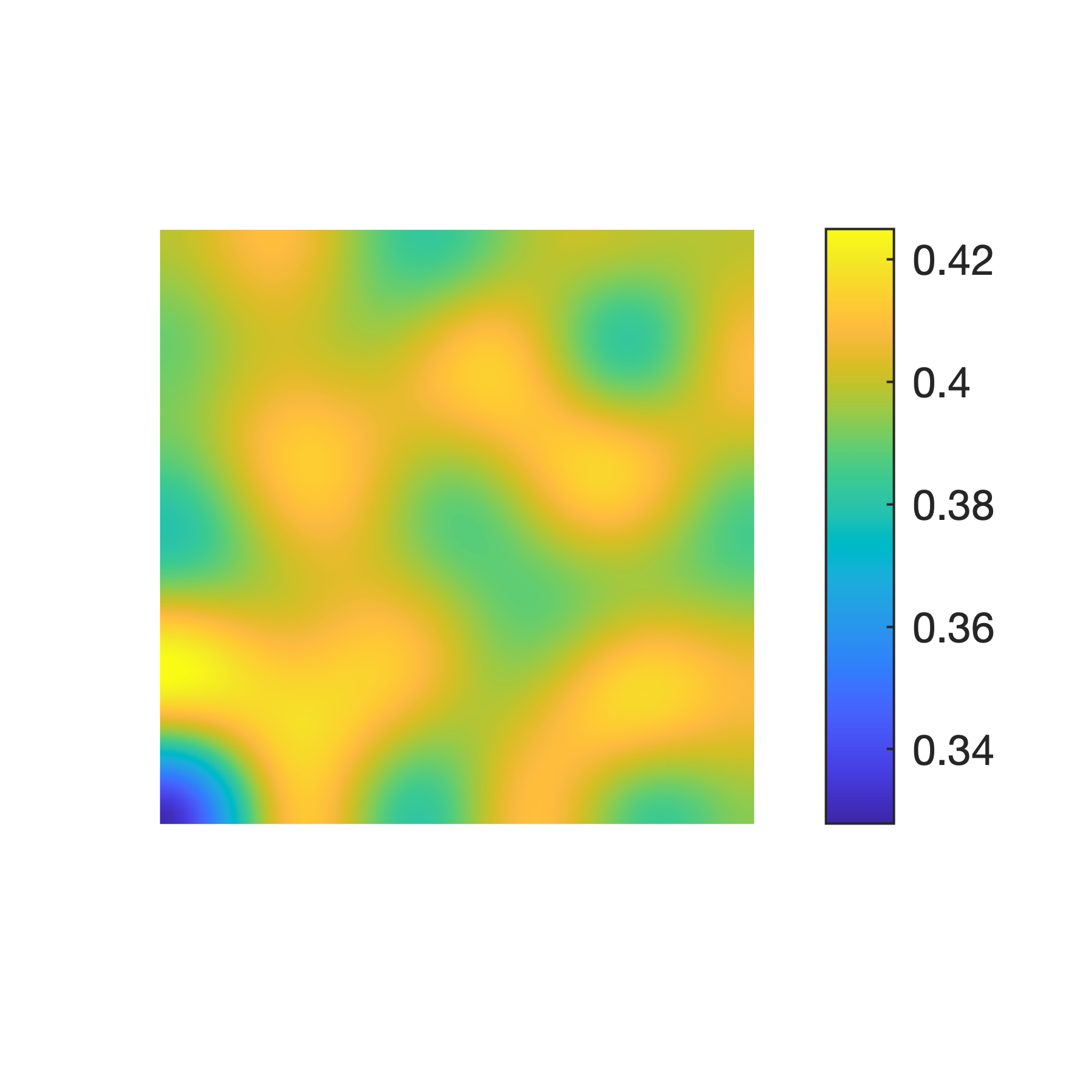}
 \includegraphics[width=0.22\textwidth,trim=1cm 1.5cm 0.5cm 1.5cm,clip]{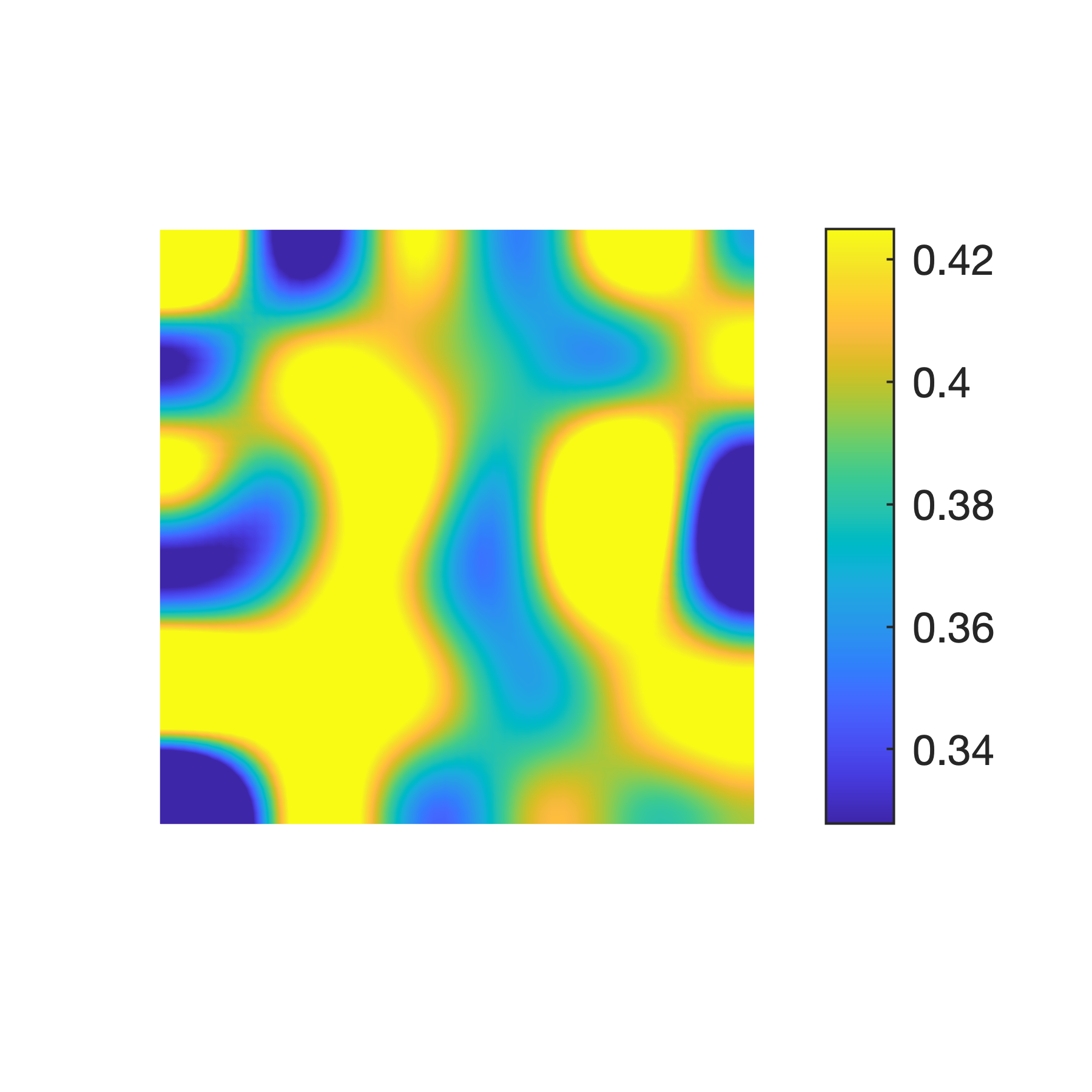}
 \includegraphics[width=0.22\textwidth,trim=1cm 1.5cm 0.5cm 1.5cm,clip]{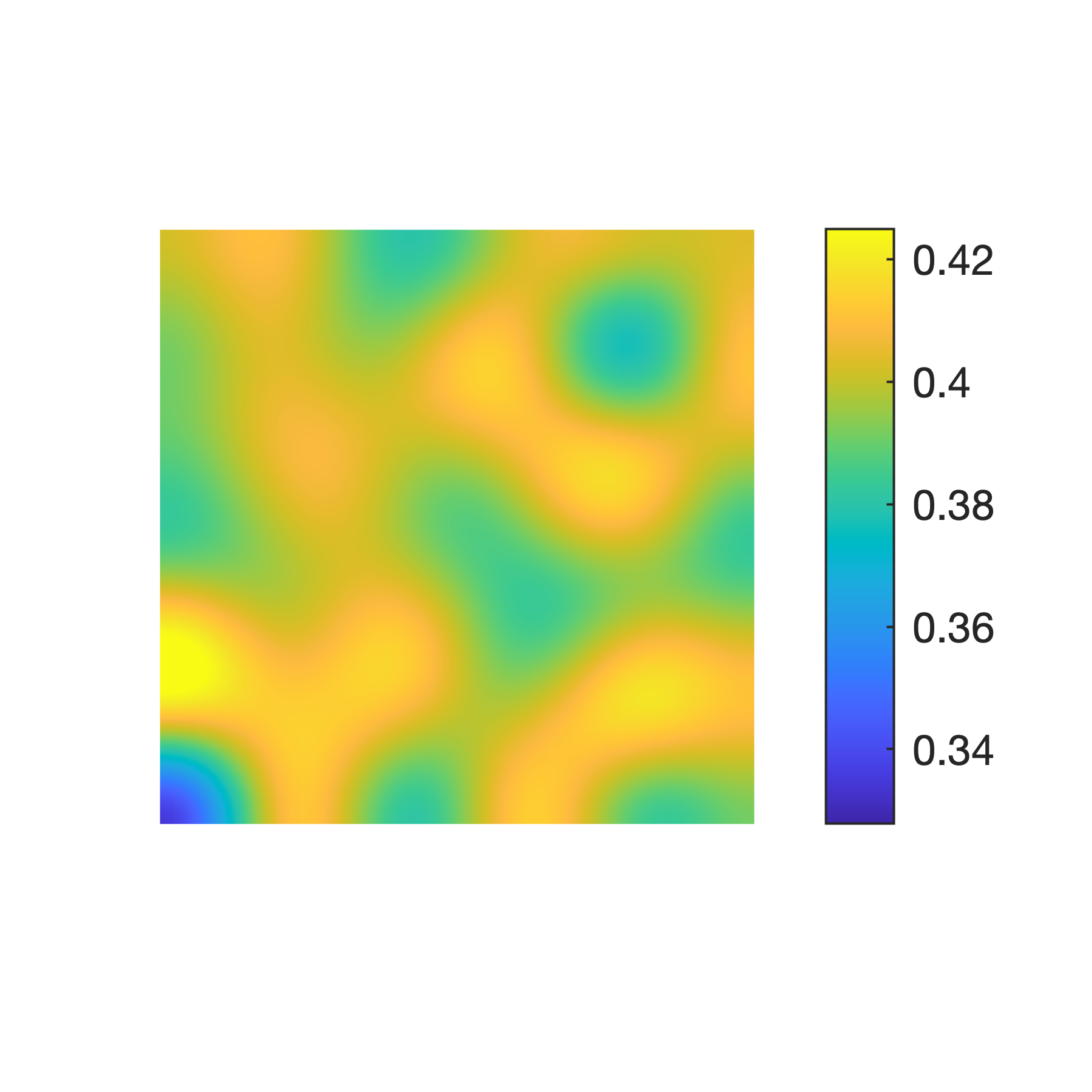}
	\caption{Joint inversion of $(\gamma, \sigma)$ (top, bottom) in Numerical Experiment II. Shown from left to right are: true coefficients, reconstruction from noise-free data, reconstruction from data with $5\%$ additive Gaussian noise, and reconstruction from data with $5\%$ multiplication Gaussian noise. }
	\label{FIG:reconstructed_diffusion_absorption}
\end{figure}

%%%%%%%%%%%%%%%%%%%%%%%%%%%%%%%%%%%%%%%%%%%%%%%%%%%%%%%%%%%%%%%%%%
\subsection{The acoustic inverse problem}
\label{SUBSEC:Num Acous}
%%%%%%%%%%%%%%%%%%%%%%%%%%%%%%%%%%%%%%%%%%%%%%%%%%%%%%%%%%%%%%%%%%

We now present some simulation results on the acoustic wave propagation model~\eqref{EQ:Wave}-\eqref{EQ:Wave Data}. We use the transmission geometry where the illuminating sources are placed on one side of the medium, while the receivers are placed on the opposite side of the medium. This is the setup in ultrasound transmission tomography~\cite{HeSaSc-PMB13}; see the setup in the right plot of Figure~\ref{FIG:Domain}. To mimic wave propagation in the free space, which is the model used in applications, we use the standard technique of perfectly matched layers (PML)~\cite{Berenger-JCP94,Johnson-MIT07}. In a nutshell, we surround the medium to be reconstructed with a layer of specially designed (absorbing) medium such that waves exiting the medium will not be bounced back into the medium again. The details of the implementation and benchmark of the forward solver are documented in~\cite{Cheng-Thesis25}.

We use the time-domain stagger-grid finite difference scheme that adopts a second-order accuracy in both the time and the spatial directions to solve the wave equation~\eqref{EQ:Wave}, then record the wave signal starting at time $t_0 = 0$ until the termination time $T$ for each time step. The spatial grid is taken to be the uniform Cartesian grid $(x_i,z_j) = (i\Delta x,-1+j\Delta z),i,j = 0,1,\cdots,M$ with $\Delta x = \Delta z = 1/M$. The receivers are equally placed on the bottom boundary, coinciding with the grid points; namely, there are $M + 1$ receivers for each pair of $(\kappa, \rho)$. We then record the measured data of the form~\eqref{EQ:Wave Data} for the reconstructions. 

\paragraph{Experiment III.} The learning stage of our data-driven two-stage joint inversion procedure is independent of the physical model involved, besides the fact that we use the physical model to constrain the learning process so that the learned relation is consistent with the underlying physical model. Therefore, here, we omit the learning stage to save space but focus on the joint reconstruction. The true relation between the density $\rho$ and the bulk modulus $\kappa$ is taken as
\begin{equation}\label{EQ:Kappa-Rho}
    \rho(\bx)=\frac{1}{\sqrt{(2\pi \sigma^2)^2}}\int_{\bbR^2} e^{-\frac{|\bx-\by|^2}{2\sigma^2}} \kappa(\by) \chi_\Omega d\by\,.
\end{equation}
with $\sigma=3\Delta x$ and $\chi_\Omega$ the characteristic function of the region of interest.

In all the numerical results, we use data collected from $N_s=20$ point sources of the form
\begin{equation}\label{wave_source1}
S(t, \bx) = h(t)\delta(\bx-\bx_s)\,,
\end{equation}
with time signature given by the Ricker wavelet~\cite{GhKr-IEEE14}:
\[
    \psi(t) =A\left(1 - 2\pi^2 f_M^2 (t-t_0) ^2 \right) e^{-\pi^2 f_M^2 (t-t_0)^2 }\,,
\]
where $A$ is the amplitude, and $f_M$ is the peak frequency of the source. For the purpose of illustration, we show in Figure~\ref{FIG:Wave Field} some typical wave fields inside the medium.
\begin{figure}[!htb]
\centering
\includegraphics[width=0.23\textwidth]{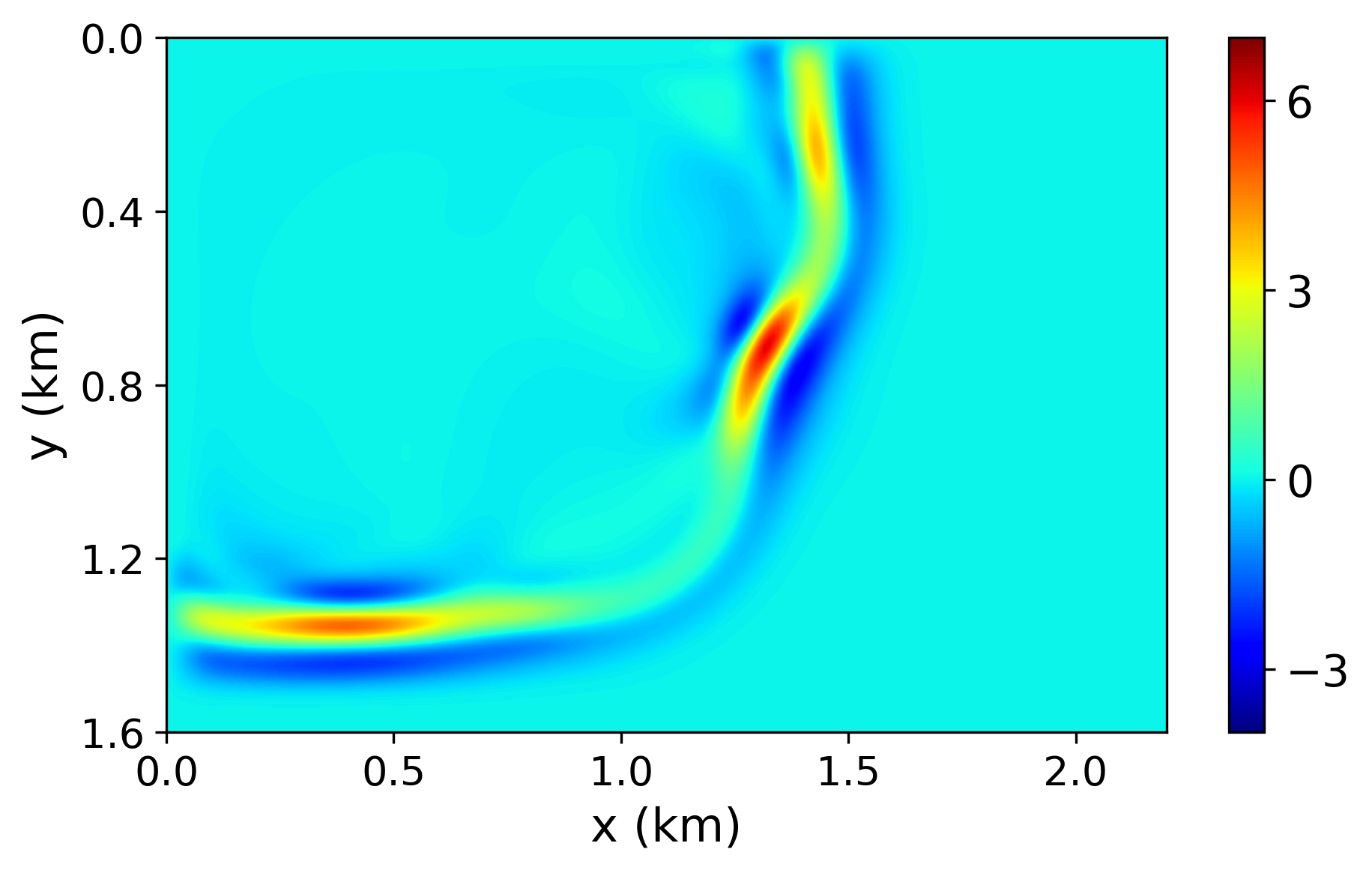}
\includegraphics[width=0.23\textwidth]{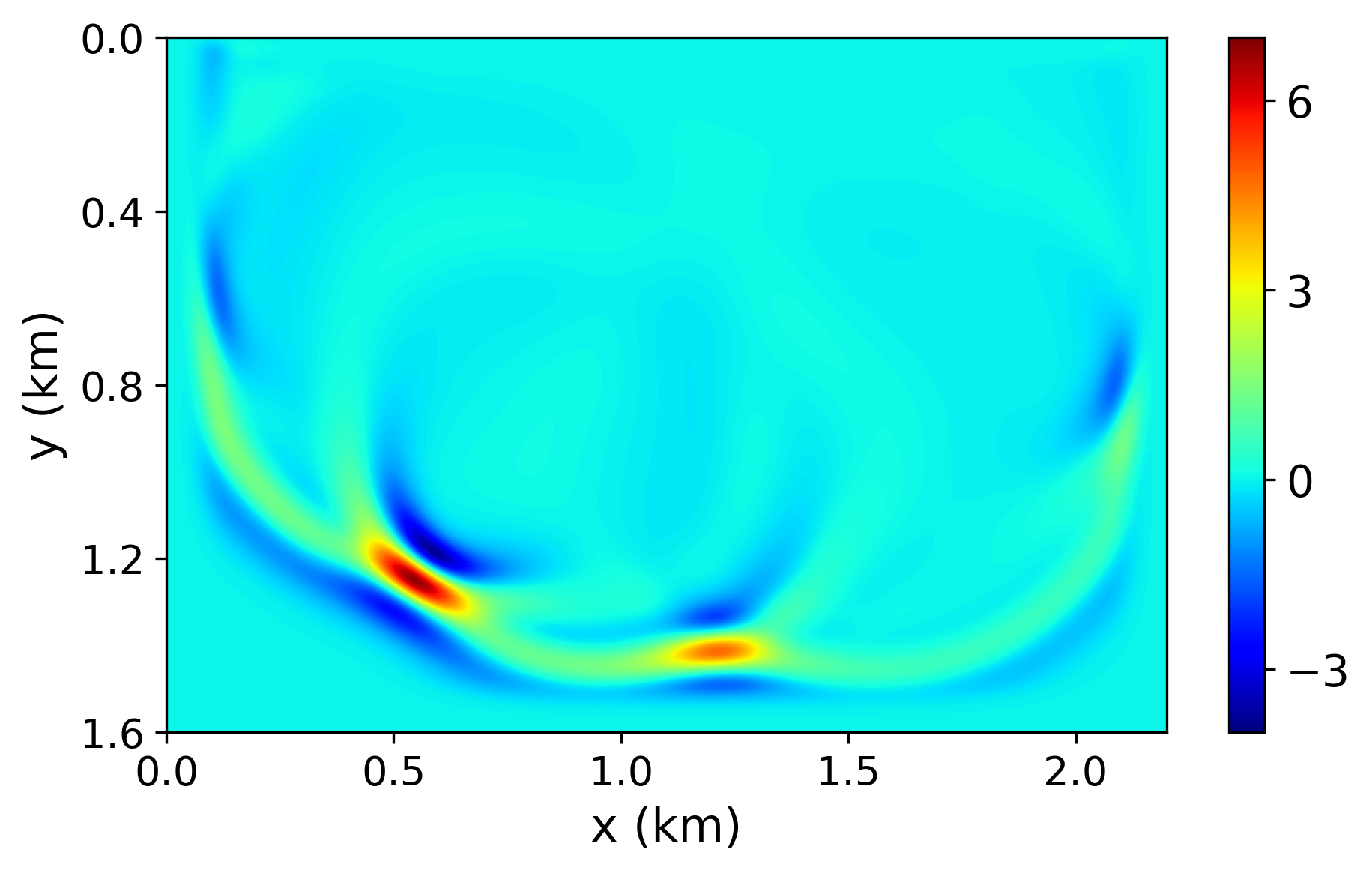}
\includegraphics[width=0.23\textwidth]{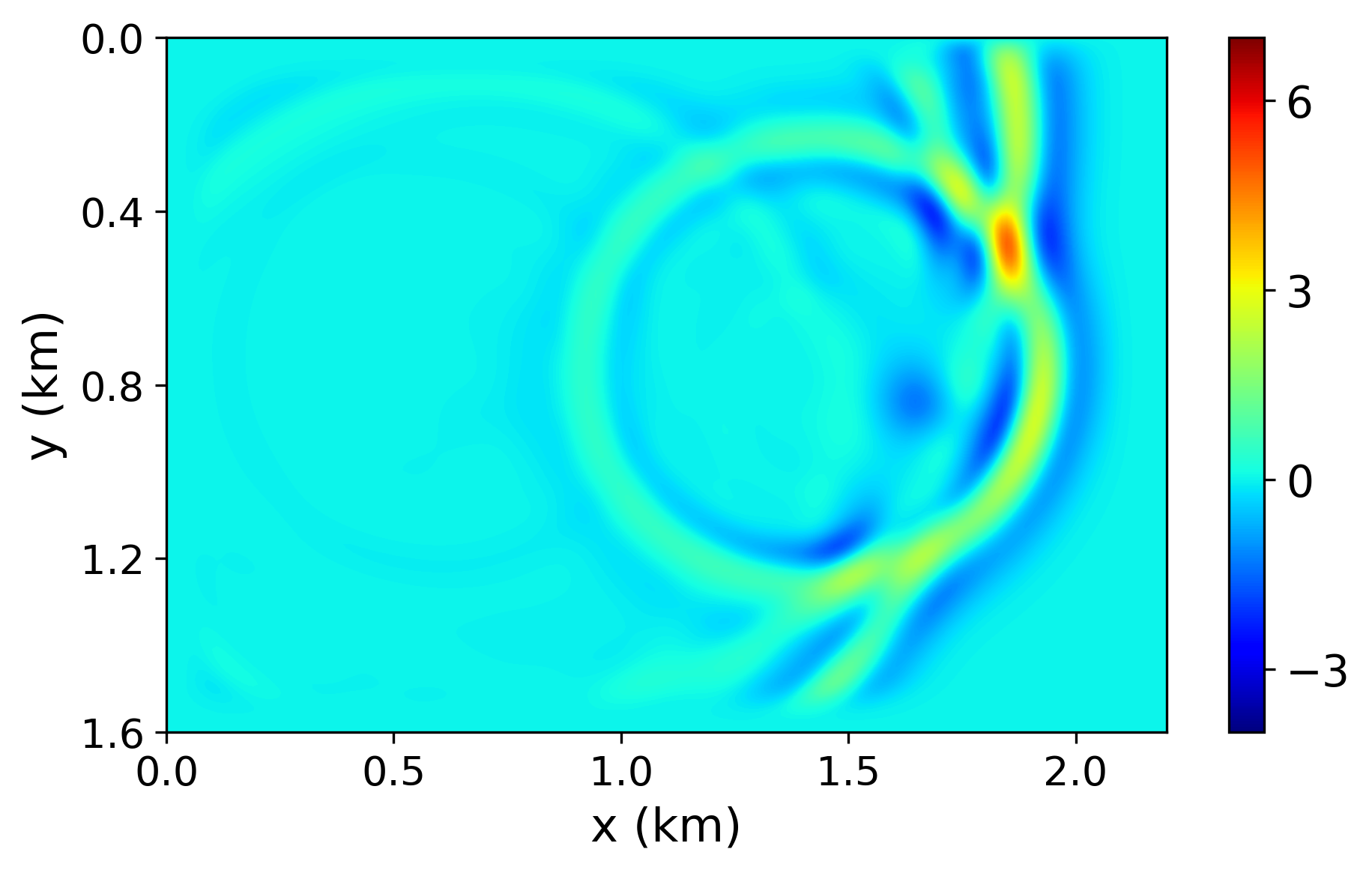}
\includegraphics[width=0.23\textwidth]{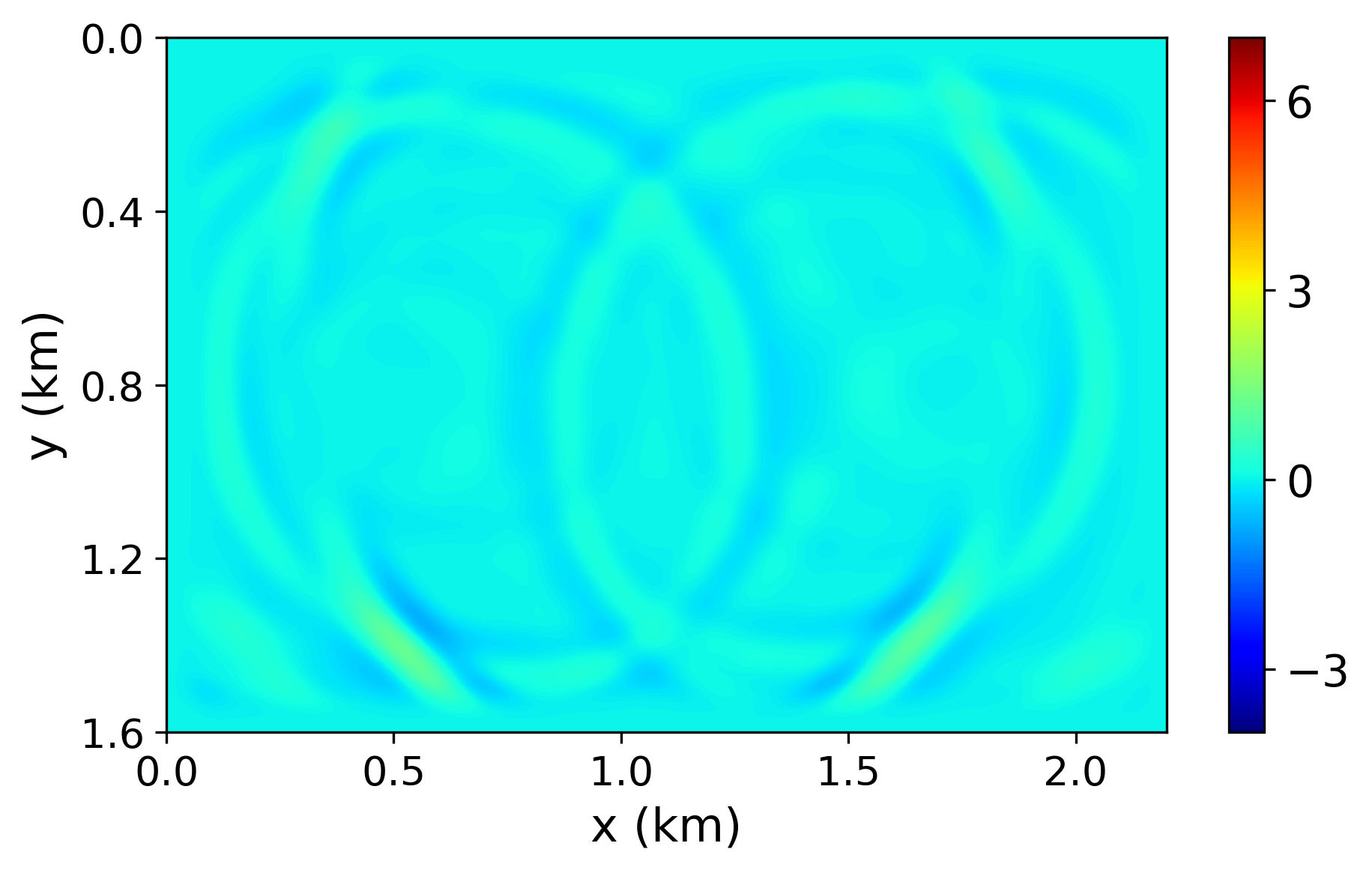}
\caption{Typical wave fields inside the medium at given times. (a-b) is for the medium in Figure~\ref{FIG:Kappa-Rho-Gaussian} and (c-d) is for the medium in Figure~\ref{FIG:Kappa-Rho-Disc}.}
\label{FIG:Wave Field}
\end{figure}

\begin{figure}[!htb]
\centering
\includegraphics[width=0.26\textwidth]{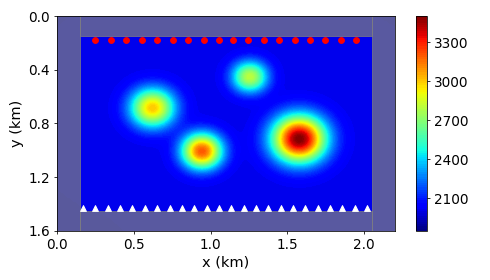}
\includegraphics[width=0.72\textwidth]{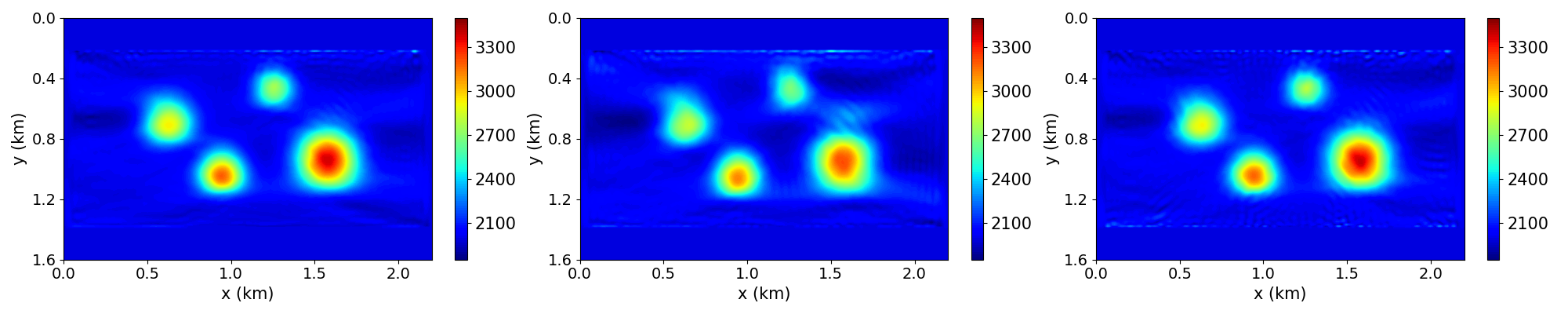}\\
\includegraphics[width=0.26\textwidth]{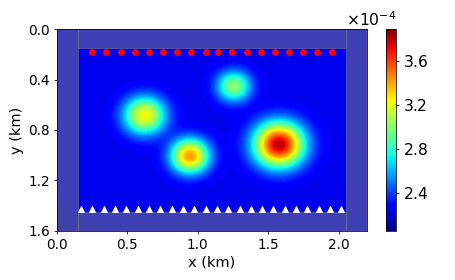}
\includegraphics[width=0.72\textwidth]{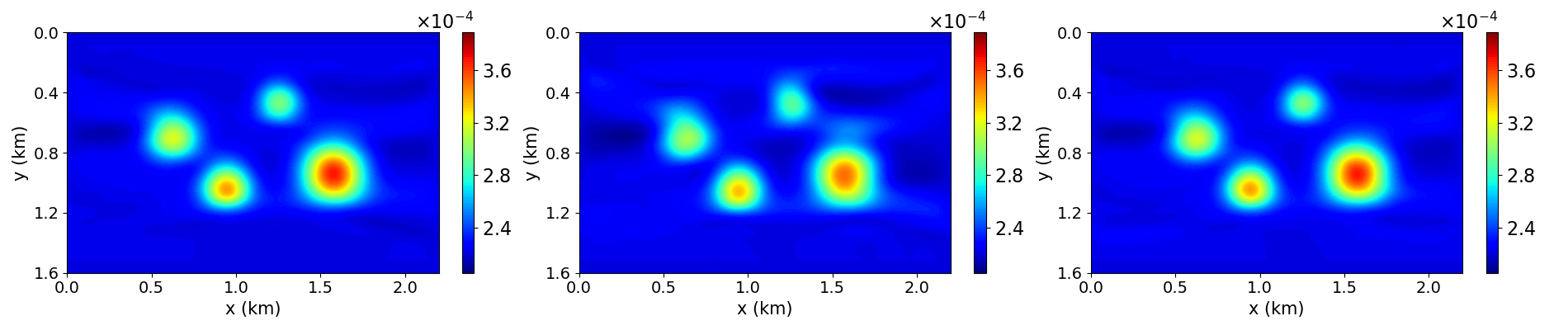}
\caption{Joint inversion of $(\kappa, \rho)$ in Numerical Experiment III. Shown from left to right are: true coefficients $(v:=\sqrt{\frac{\kappa}{\rho}}, \rho)$ (top, bottom), reconstruction from noise-free data, reconstruction from data with $5\%$ additive Gaussian noise, and reconstruction from data with $10\%$ Gaussian noise.}
\label{FIG:Kappa-Rho-Gaussian}
\end{figure}
We perform joint inversions here using the relation~\eqref{EQ:Kappa-Rho} in two different numerical simulations. In Figure~\ref{FIG:Kappa-Rho-Gaussian}, we show the true coefficients, where $\kappa$ is the superposition of a constant background and a few Gaussian functions, and the reconstructions from data with different noise levels. We plot the wave speed $v:=\sqrt{\frac{\kappa}{\rho}}$ (which is only the definition of wave speed if $\rho$ is a constant) and the density $\rho$. The optimization algorithm for the joint reconstruction is set with the termination tolerances $10^{-8}$ for both the first-order optimality condition and the updating step size. In addition, the hyper-parameters in Algorithm ~\ref{ALG:JointInv} is set to be $(\eta_j, J) = (0, 3)$. Since we used data collected from many illumination sources, $N_s=20$, the effective noise level in the data is likely much lower than the numbers provided. Even so, we observe significant errors in the reconstructions.

We repeat the simulation for a different true $(\kappa, \rho)$ profile where the bulk modulus $\kappa$ is piecewise constant. The results are shown in Figure~\ref{FIG:Kappa-Rho-Disc}. The algorithmic parameters are the same as in the previous example. The quality of the reconstructions is quite reasonable. Let us emphasize again that the true relation between $\kappa$ and $\rho$ in~\eqref{EQ:Kappa-Rho} is used in the joint reconstruction. Our main purpose here is to show that, once we know the relation between $\kappa$ and $\rho$, the joint reconstruction can be done with good quality. We are not here to show that we can learn the relation faithfully (a task that depends on the quality of the historical data of $\kappa$ and $\rho$, as well as the learning algorithm proposed).
\begin{figure}[!htb]
\centering
\includegraphics[width=0.26\textwidth]{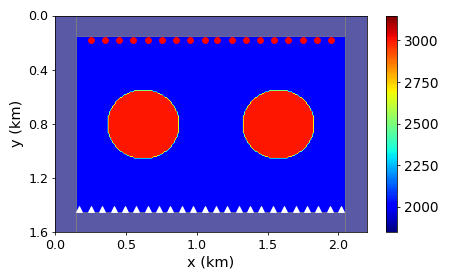}
\includegraphics[width=0.72\textwidth]{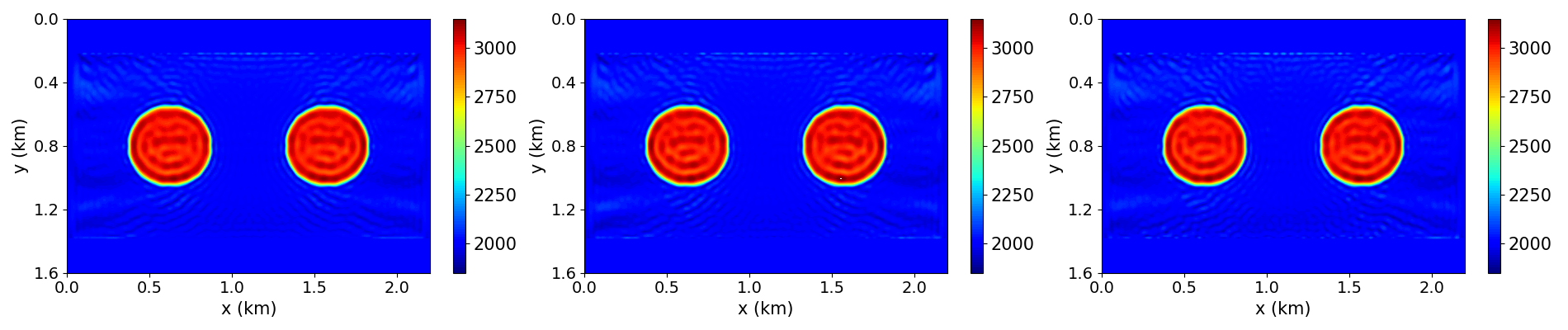}\\
\includegraphics[width=0.26\textwidth]{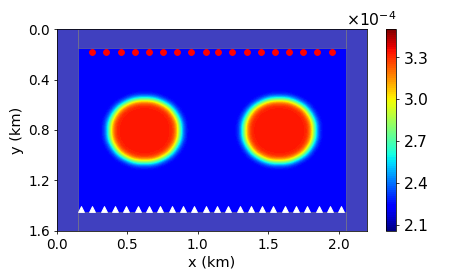}
\includegraphics[width=0.72\textwidth]{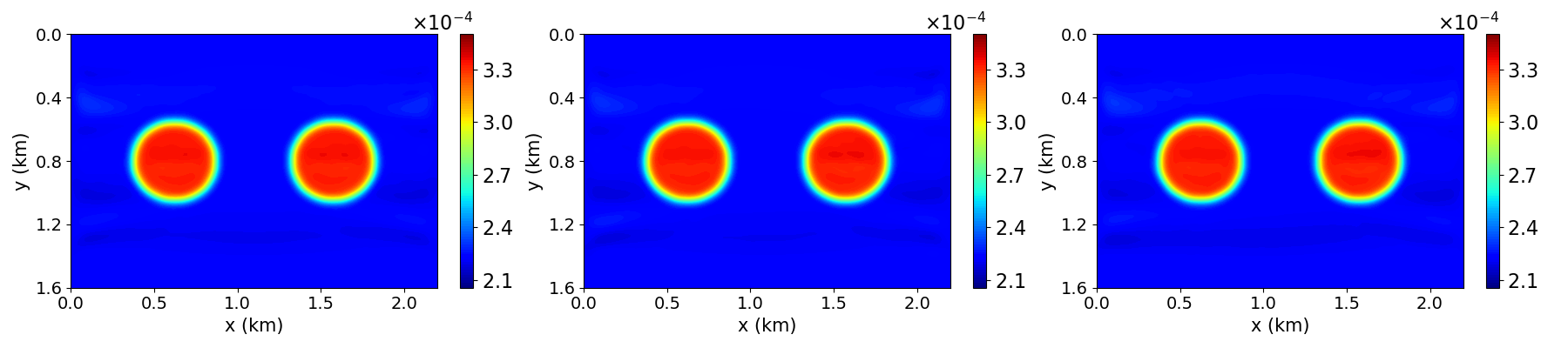}
\caption{Same as Figure~\ref{FIG:Kappa-Rho-Gaussian} except that the true coefficients are piecewise constant now.}
\label{FIG:Kappa-Rho-Disc}
\end{figure}

%%%%%%%%%%%%%%%%%%%%%%%%%%%%%%%%%%%%%%%%%%%%%%%%%%%%%%%%%%%%%%%%%%
%%%%%%%%%%%%%%%%%%%%%%%%%%%%%%%%%%%%%%%%%%%%%%%%%%%%%%%%%%%%%%%%%%
\section{Impact of learning inaccuracy on reconstruction}
\label{SEC:Uncertainty}
%%%%%%%%%%%%%%%%%%%%%%%%%%%%%%%%%%%%%%%%%%%%%%%%%%%%%%%%%%%%%%%%%%
%%%%%%%%%%%%%%%%%%%%%%%%%%%%%%%%%%%%%%%%%%%%%%%%%%%%%%%%%%%%%%%%%%

While numerical examples in rather academic environments show promises in learning relations between the coefficients from given data, it is not possible to learn the exact relation between the unknowns in practice, even if such an exact relation exists. In other words, the map $\cN_{\wh \theta}$ that we learned is only an approximation to the true relation. In this section, we show that, under reasonable assumptions on the inverse problems, when the learning error is under control, the error caused in the joint reconstruction is also under control.

We will use the following standard notations in the rest of the paper. We denote by $L^p(\Omega)$ ($1\le p\le \infty$) the usual space of Lebesgue integrable functions on $\Omega$, and $W^{k, p}(\Omega)$ the Sobolev space of functions whose $j$th derivatives ($0\le j\le k$) are in $L^p(\Omega)$. We use the short notation $\cH^k(\Omega):=W^{k, 2}(\Omega)$. The set of functions whose derivatives up to $k$ are continuous in $\Omega$ is denoted by $\cC^k(\Omega)$. We denote by $\cB(\Omega)$ the set of strictly positive functions bounded between two constants $\underline{\alpha}$ and $\overline{\alpha}$, 
\begin{equation*}
\cB(\Omega)=\{f(\bx): \Omega\mapsto\bbR: 0< \underline{\alpha} \le f(\bx) \le \overline{\alpha} <\infty,\ \forall\bx\in \Omega\}.
\end{equation*}
For the simplicity of presentation, we assume that the domain $\Omega$ is bounded with a smooth boundary $\partial\Omega$. 

%%%%%%%%%%%%%%%%%%%%%%%%%%%%%%%%%%%%%%%%%%%%%%%%%%%%%%%%%%%%%%%%%%
\subsection{Joint inversion of the diffusion model}
%%%%%%%%%%%%%%%%%%%%%%%%%%%%%%%%%%%%%%%%%%%%%%%%%%%%%%%%%%%%%%%%%%
 
For the joint inversion problem with the diffusion model we introduced in ~\Cref{SUBSEC:Diff}, we make the following assumptions:\\[1ex]
(A-i) the boundary value of the diffusion coefficient $\gamma$, $\gamma_{|\partial\Omega}>0$, is known \emph{a priori};\\[1ex]
(A-ii) we have sufficient number of data set $H$ such that the absorption coefficient $\sigma$ is uniquely determined when the relation $\gamma=\cN_\theta(\sigma)$ is known;\\[1ex]
(A-iii) the boundary condition $S(\bx)>0$ is the restriction of a smooth function on $\partial\Omega$, and is such that the solution $u$ to the diffusion equation~\eqref{EQ:Diff} satisfies $u\ge \eps$ for some $\eps>0$.\\[1ex]
The assumptions are here only to simplify the presentation, as none of the assumptions is essential. In fact, (A-i) can be removed with additional boundary data as shown in~\cite{ReGaZh-SIAM13}, (A-ii) is true when data corresponding to two well-chosen boundary conditions are available~\cite{BaRe-IP11}, and (A-iii) is true when $\gamma$ and $\sigma$ are non-negative and smooth enough~\cite{ReZh-SIAM18}. 

Following the calculations of ~\cite{ReVa-SIAM20}, we can show the result below.

\begin{theorem}
Let $\sigma\in\cC^1(\Omega) \cap \cB$ and $\wt\sigma\in \cC^1(\Omega) \cap \cB$ be the absorption coefficients reconstructed from ~\eqref{EQ:Diff}-~\eqref{EQ:Diff Data} with the relations $\gamma=\cN_{\theta}(\sigma)$ and $\wt\gamma=\cN_{\wt\theta}(\wt\sigma)$ respectively. Assuming that the operators $\cN_{\theta}$ and $\cN_{\wt\theta}: \cC^1(\Omega)\cap \cB \mapsto \cC^1(\Omega)\cap \cB$ are sufficiently smooth. Then, under the assumptions (A-i)-(A-iii), there exists a constant $\fc$ such that
\begin{equation}\label{EQ:Bound Diff}
	\|\frac{\wt \sigma-\sigma}{\wt\sigma}\|_{\cH^1(\Omega)}\le \fc\Big(\|\cN_{\wt\theta}(\wt\sigma)-\cN_{\wt\theta}(\sigma)\|_{\cH^1(\Omega)}+\|\cN_{\wt\theta}(\sigma)-\cN_{\theta}(\sigma)\|_{\cH^1(\Omega)}\Big)\|\frac{H}{\sigma}\|_{W^{2,\infty}(\Omega)}.
\end{equation}
\end{theorem}
\begin{proof}
Let $u$ and $\wt u$ be the solution to ~\eqref{EQ:Diff} corresponding to $(\gamma:=\cN_\theta(\sigma), \sigma)$ and $(\wt\gamma:=\cN_{\wt\theta}(\wt\sigma), \wt \sigma)$, respectively. Then we have
\begin{equation*}
	\begin{array}{rcll}
	-\nabla\cdot \cN_{\theta}(\sigma) \nabla u(\bx) + \sigma(\bx)  u(\bx) &=& 0, & \mbox{in}\ \ \Omega\\[1ex]
	\bn\cdot \cN_\theta(\sigma) \nabla u + \ell u(\bx) & = & S(\bx), & \mbox{on}\ \partial\Omega\,.
	\end{array} 	 
\end{equation*}
and
\begin{equation*}
	\begin{array}{rcll}
	-\nabla\cdot \cN_{\wt\theta}(\wt \sigma) \nabla \wt u(\bx) + \wt \sigma(\bx) \wt u(\bx) &=& 0, & \mbox{in}\ \ \Omega\\[1ex]
	\bn\cdot \cN_{\wt \theta}(\wt\sigma)\nabla \wt u + \ell \wt u(\bx) & = & S(\bx), & \mbox{on}\ \partial\Omega\,.
	\end{array} 	 
\end{equation*}
With the assumptions (A-i)-(A-iii) as well as the assumption on the positivity of the coefficient pairs $(\cN_\theta(\sigma), \sigma)$ and $(\cN_{\wt\theta}(\wt\sigma), \wt\sigma)$, the above diffusion equations admit unique solutions $u\in\cC^2(\bar \Omega)$ and $\wt u\in\cC^2(\bar \Omega)$ by classical elliptic theory~\cite{Evans-Book13,GiTr-Book00}. Moreover, $M\ge u, \wt u \ge \eps$ for some $\eps>0$ and $M>0$. 

Let $w:=u-\wt u$. It is straightforward to verify that $w$ solves
\begin{equation}\label{EQ:Equat for w}
	\begin{array}{rcll}
	-\nabla\cdot \cN_{\wt\theta}(\wt\sigma) \nabla w + \wt\sigma(\bx)  w &=& -\nabla\cdot (\cN_{\wt\theta}(\wt\sigma)-\cN_{\theta}(\sigma)) \nabla u + (\wt\sigma-\sigma) u, & \mbox{in}\ \ \Omega\\[1ex]
	\bn\cdot \cN_{\wt \theta}(\wt\sigma) \nabla w + \ell w(\bx) & = & 0, & \mbox{on}\ \partial\Omega
	\end{array} 	 
\end{equation}
where we used the assumption in (A-i), that is, the boundary value of the diffusion coefficient is known, so $\cN_{\theta}(\sigma)_{|\partial\Omega}=\cN_{\wt \theta}(\wt\sigma)_{|\partial\Omega}$. Meanwhile, since $\sigma$ and $\wt\sigma$ are reconstructed from the same data $H$, we have that 
\begin{equation}\label{EQ:Data Diff Eq}
	0=\sigma u-\wt \sigma \wt u=\wt \sigma w - (\wt\sigma-\sigma) u\,.
\end{equation}
This gives us, from~\eqref{EQ:Equat for w}, that $w$ solves the equation
\begin{equation*}
	\begin{array}{rcll}
	-\nabla\cdot \cN_{\wt\theta}(\wt\sigma) \nabla w(\bx) &=& -\nabla\cdot (\cN_{\wt\theta}(\wt\sigma)-\cN_{\theta}(\sigma)) \nabla u, & \mbox{in}\ \ \Omega\\[1ex]
	\bn\cdot \cN_{\wt \theta}(\wt\sigma) \nabla w + \ell w(\bx) & = & 0, & \mbox{on}\ \partial\Omega\,.
	\end{array} 	 
\end{equation*}
\RED{With the assumptions given, this equation for $w$ is uniformly elliptic. By standard elliptic estimates~\cite{Evans-Book13,GiTr-Book00}, it follows that
\begin{multline*}
    \|w\|_{\cH^1(\Omega)}\le \|\nabla\cdot (\cN_{\wt\theta}(\wt\sigma)-\cN_{\theta}(\sigma)) \nabla u\|_{L^2(\Omega)}\\
    \le\|\nabla u\|_{L^\infty(\Omega)}\|\nabla  (\cN_{\wt\theta}(\wt\sigma)-\cN_{\theta}(\sigma))\|_{L^2(\Omega)}+\|\Delta u\|_{L^\infty(\Omega)}\|\cN_{\wt\theta}(\wt\sigma)-\cN_{\theta}(\sigma)\|_{L^2(\Omega)}\\ 
    \le \|u\|_{W^{2,\infty}(\bar\Omega)}\|\cN_{\wt\theta}(\wt\sigma)-\cN_{\theta}(\sigma)\|_{\cH^1(\Omega)}\,.
\end{multline*} 
Using the fact that
\[
\cN_{\wt\theta}(\wt\sigma)-\cN_{\theta}(\sigma)=\cN_{\wt\theta}(\wt\sigma)-\cN_{\wt\theta}(\sigma)+\cN_{\wt \theta}( \sigma)-\cN_{\theta}(\sigma)\,,
\]
we can further write the above bound into
\[
    \|w\|_{\cH^1(\Omega)}\le \|u\|_{W^{2,\infty}(\bar\Omega)}\Big(\|\cN_{\wt\theta}(\wt\sigma)-\cN_{\wt \theta}(\sigma)\|_{\cH^1(\Omega)}+\|\cN_{\wt \theta}(\sigma)-\cN_{\theta}(\sigma)\|_{\cH^1(\Omega)}\Big)\,.
\]
On the other hand, we have from~\eqref{EQ:Data Diff Eq} that
\[
    \|\frac{(\wt\sigma-\sigma)}{\wt\sigma} u\|_{\cH^1(\Omega)} \le \|w\|_{\cH^1(\Omega)}\,.
\]
Therefore, we arrive at
\[
    \|\frac{(\wt\sigma-\sigma)}{\wt\sigma} u\|_{\cH^1(\Omega)} \le \|u\|_{W^{2,\infty}(\bar\Omega)}\Big(\|\cN_{\wt\theta}(\wt\sigma)-\cN_{\wt \theta}(\sigma)\|_{\cH^1(\Omega)}+\|\cN_{\wt \theta}(\sigma)-\cN_{\theta}(\sigma)\|_{\cH^1(\Omega)}\Big)\,.
\]}
The proof is complete after using the fact that $u$ is bounded from below and $u=\dfrac{H}{\sigma}$.
\end{proof}
\RED{This result says that when the operators $\cN_\theta$ and $\cN_{\wt\theta}$ are sufficiently close to each other, in which case $\|\cN_\theta(\sigma)-\cN_{\wt\theta}(\sigma)\|_{\cH^1(\Omega)}$ is small, and $\cN_{\wt\theta}$ is sufficiently smooth, in which case $\|\cN_{\wt\theta}(\wt\sigma)-\cN_{\wt\theta}(\sigma)\|_{\cH^1(\Omega)}/\|\wt\sigma-\sigma\|_{\cH^1(\Omega)}$ is small, the reconstructions based on those relations are also sufficiently close to each other. In other words, the impact of learning errors on joint reconstruction is under control.}

\RED{One should also notice that the above characterization of error propagation is relative in the sense that the quantity $\frac{H}{\sigma}$ appears on the right-hand side of the inequality. In other words, the stability estimate does not directly imply that we can reconstruct $\sigma$ accurately. In fact, it only says that if we have a reconstruction algorithm, and we use the relations $\cN_\theta$ and $\cN_{\wt\theta}$ in the reconstruction process, then $\cN_\theta$ and $\cN_{\wt \theta}$ being close implies that $\sigma$ and $\wt \sigma$ are close, even though both reconstructions can be fairly inaccurate.}

%%%%%%%%%%%%%%%%%%%%%%%%%%%%%%%%%%%%%%%%%%%%%%%%%%%%%%%%%%%%%%%%%%
\subsection{Joint inversion of the wave model}
%%%%%%%%%%%%%%%%%%%%%%%%%%%%%%%%%%%%%%%%%%%%%%%%%%%%%%%%%%%%%%%%%%

We now investigate the same issue for the acoustic wave model~\eqref{EQ:Wave}-\eqref{EQ:Wave Data}  in ~\Cref{SUBSEC:Acous}. To simplify the notation, we perform a change of variables $\kappa \to 1/\kappa$ and $\rho\to 1/\rho$, and consider relations between the new variables: $\rho=\cN_\theta(\kappa)$. We make the following assumptions:\\[1ex]
(B-i) the boundary value of the density $\rho$, $\rho_{|\partial\Omega}>0$, is known \emph{a priori};\\[1ex]
(B-ii) $S(t,\bx)>0$ is the restriction of a $\cC^\infty$ function on $\partial\Omega$.\\[1ex]
(B-iii) the problem of determining $\kappa$ (with fixed $\rho$) from data $\Lambda_{\kappa}: S \mapsto H$ can be stably solved in the sense that $\|\wt\kappa-\kappa\|_{W^{1,\infty}}\le C\|\Lambda_{\wt\kappa}-\Lambda_{\kappa}\|_{\cH^{1/2}((0, T)\times\partial\Omega) \mapsto \cH^{3/2}((0, T)\times\partial\Omega)}$ for some $C>0$.\\[1ex]
The problem in (B-iii), that is, determining $\kappa$ from data $\Lambda_{\kappa}$, has been extensively studied. The uniqueness of the reconstruction problem has been proved by the method of boundary control; see, for instance,~\cite {Belishev-IP97,Isakov-Book06} and references therein. Global stability, such as what we assumed here, is not known. Our main objective is to see the impact of uncertainty in $\rho$ on the reconstruction of $\kappa$, not how accurately we can reconstruct $\kappa$. This assumption is necessary but also reasonable.
\begin{theorem}
Under the assumptions (B-i)-(B-iii), let $\kappa\in W^{1,\infty}(\Omega) \cap \cB$ and $\wt\kappa\in W^{1,\infty}(\Omega) \cap \cB$ be reconstructed from ~\eqref{EQ:Wave}-\eqref{EQ:Wave Data} with the relations $\rho=\cN_{\theta}(\kappa)$ and $\wt\rho =\cN_{\wt\theta}(\wt\kappa)$ respectively. Assuming that the operators $\cN_{\theta}$ and $\cN_{\wt\theta}: W^{1,\infty}(\Omega)\cap \cB \mapsto W^{1,\infty}(\Omega)\cap \cB$ are sufficiently smooth. Then there exists a constant $\fc$ such that
\begin{equation}\label{EQ:Bound Wave}
	\|\wt\kappa-\kappa \|_{W^{1,\infty}(\Omega)}\le \fc\Big(\|\cN_{\wt\theta}(\wt\kappa)-\cN_{\wt\theta}(\kappa)\|_{W^{1,\infty}(\Omega)}+\|\cN_{\wt\theta}(\kappa)-\cN_{\theta}(\kappa)\|_{W^{1,\infty}(\Omega)}\Big).
\end{equation}
\end{theorem}
\begin{proof}
By the assumptions on the smoothness and boundedness of the coefficients $\kappa$ and $\wt\kappa$ as well as the operators $\cN_\theta$ and $\cN_{\wt\theta}$, the equation~\eqref{EQ:Wave} with coefficients $(\kappa, \cN_\theta(\kappa))$ and $(\wt\kappa, \cN_{\wt\theta}(\wt\kappa))$ are well-posed~\cite{Evans-Book13}. We denote by $p$ and $\wt p$ the corresponding solutions. We define $w:=\wt p-p$. Then $w$ solves
\begin{equation*}
	\begin{array}{rcll}
	\wt\kappa\dfrac{\partial^2 w}{\partial t^2} -\nabla\cdot\big(\cN_{\wt\theta}(\wt\kappa) \nabla w\big) &=&-(\wt\kappa-\kappa)\dfrac{\partial^2 p}{\partial t^2}+ \nabla\cdot(\cN_{\wt\theta}(\wt\kappa)-\cN_{\theta}(\kappa))\nabla p, & \mbox{in}\ (0, T] \times \Omega\\[1ex]
    \bn \cdot \cN_{\wt\theta} \nabla  w & = &0, & \mbox{on}\ (0, T]\times\partial\Omega\\[1ex]
 w(0, \bx)=\dfrac{\partial w}{\partial t}(0, \bx) & = & 0, & \mbox{in}\ \Omega
	\end{array}
\end{equation*}
where we used the assumption in (B-i). Let $G$ be the adjoint Green function of this problem; that is, $G$ is the solution to
\begin{equation*}
	\begin{array}{rcll}
	\wt\kappa\dfrac{\partial^2 G}{\partial t^2} -\nabla\cdot\big(\cN_{\wt\theta}(\wt\kappa) \nabla G\big) &=&0, & \mbox{in}\ (0, T] \times \Omega\\[1ex]
    \bn \cdot \cN_{\wt\theta} \nabla  G & = &\delta(t-s)\delta(\bx-\by), & \mbox{on}\ (0, T]\times\partial\Omega\\[1ex]
 G(T, \bx; s, \by)=\dfrac{\partial G}{\partial t}(T, \bx; s, \by) & = & 0, & \mbox{in}\ \Omega\,.
	\end{array}
\end{equation*}
Then for any $(t, \bx) \in (0, T)\times \partial\Omega$, we have
\[
    w(t, \bx) =\int_0^T\int_{\Omega}G(s, \by; t,\bx)\Big[ (\wt\kappa-\kappa)\dfrac{\partial^2 p}{\partial t^2}- \nabla\cdot(\cN_{\wt\theta}(\wt\kappa)-\cN_{\theta}(\kappa))\nabla p\Big] d\by ds\,.
\]
Using the fact that the reconstructions are from the same data, we have that 
\[
    w(t,\bx)  = 0,\ \ \  \mbox{on}\ (0, T]\times\partial\Omega\,.
\]
Therefore, we have
\[
    \int_{\Omega} (\wt\kappa-\kappa) \int_0^T G(s, \by; t,\bx) \dfrac{\partial^2 p}{\partial t^2} d\bs d\bx =
    \int_{\Omega} \int_0^T G\nabla\cdot(\cN_{\wt\theta}(\wt\kappa)-\cN_{\theta}(\kappa))\nabla p d\by ds\,.
\]
This can be further written into, using the assumption that $\cN_{\wt\theta}(\wt\kappa)=\cN_\theta(\kappa)$ on $\partial\Omega$,
\[
    \int_{\Omega} (\wt\kappa-\kappa) \Big(\int_0^T G(s, \by; t,\bx) \dfrac{\partial^2 p}{\partial t^2} d\bs\Big) d\bx =
    \int_{\Omega} \Big(\cN_{\wt\theta}(\wt\kappa)-\cN_{\theta}(\kappa)\Big)\Big(\int_0^T \nabla G\cdot\nabla p ds\Big) d\by\,.
\]
Under the assumptions we made, by classical theory~\cite{Isakov-Book06}, the integral operator on the left, denoted by $\cG$, with kernel $ \Big(\int_0^T G(s, \by; t,\bx) \dfrac{\partial^2 p}{\partial t^2} d\bs\Big)$ is invertible. We, therefore, write the above equation into the form
\[
    \wt\kappa-\kappa=\cG^{-1} \Big[
    \int_{\Omega} \Big(\cN_{\wt\theta}(\wt\kappa)-\cN_{\theta}(\kappa)\Big)\Big(\int_0^T \nabla G\cdot\nabla p ds\Big) d\by\Big]:=\cG^{-1} \cG_{\Delta} \Big(\cN_{\wt\theta}(\wt\kappa)-\cN_{\theta}(\kappa)\Big)\,.
\]
With the assumption in (B-iii), we have a version of $\cG^{-1}: \cH^{1/2}((0,T]\times\partial\Omega) \mapsto W^{1,\infty}(\Omega)$ that is bounded. Moreover, $\cG_{\Delta}: \cH_1^\infty(\bar \Omega)\mapsto \cH^{1/2}((0, T)\times\partial\Omega)$ is a bounded operator. The result of ~\eqref{EQ:Bound Wave} then follows directly.
\end{proof}
We emphasize again here that the assumption in (B-iii) is necessary to get the result in this theorem. While this assumption is not verified mathematically, it is necessary for the uncertainty characterization here. On the other hand, the assumption is not essential in the sense that we are only interested in the relative change to the solution caused by the difference between the relationships $\cN_\theta$ and $\cN_{\wt\theta}$. This relative change is small as long as $\cN_{\wt\theta}-\cN_\theta$ is small (even if $\kappa$ might not have been reconstructed perfectly).

%%%%%%%%%%%%%%%%%%%%%%%%%%%%%%%%%%%%%%%%%%%%%%%%%%%%%%%%%%%%%%%%%%
%%%%%%%%%%%%%%%%%%%%%%%%%%%%%%%%%%%%%%%%%%%%%%%%%%%%%%%%%%%%%%%%%%
\section{Concluding remarks}
\label{SEC:Concl}
%%%%%%%%%%%%%%%%%%%%%%%%%%%%%%%%%%%%%%%%%%%%%%%%%%%%%%%%%%%%%%%%%%
%%%%%%%%%%%%%%%%%%%%%%%%%%%%%%%%%%%%%%%%%%%%%%%%%%%%%%%%%%%%%%%%%%

In this work, we provided a proof-of-concept study on a computational framework for data-driven joint reconstruction problems, or multiple coefficient inverse problems, for partial differential equations. We developed a method that fuses data with the mathematical model involved to learn relations between different unknowns to be reconstructed. Our method provides a learned model that is consistent with the underlying physical model for reconstruction purposes. Moreover, we use the learned model to guide the joint reconstruction instead of using it as a hard constraint. This gives the flexibility for the joint reconstruction algorithm to find solutions outside of the training dataset.

Our main objective for this study is not to replace classical inversion methods with learning methods but rather to use additional historical data to help solve the joint reconstruction problem that can not be stably solved without the necessary relations between the unknowns. Moreover, due to the fact that the data-driven modeling part of the computation can be conducted offline, we are not concerned with the computational cost introduced by the supplementary dataset. The online model-based joint inversion, even though now done with an additional loop, has a computational cost comparable to a standard joint inversion algorithm.

Our study is based on the assumption that there is a relation~\eqref{EQ:Relation Gen} between the $f$ and $g$ data that is available to us. This is indeed the case in many physics-based inverse problems where $f$ and $g$ are physical coefficients that are related. Our study would not make any sense, and the learning process for the data-driven modeling process will not converge to a stable solution if this assumption does not hold. When the assumption indeed holds, and the joint inversion problem is stable with respect to the data, our preliminary sensitivity analysis shows that the reconstruction is stable, in an appropriate sense, with respect to the inaccuracy in the learning result.

There is a critical issue that we left without discussion in this work, that is, should we learn the map from $f$ to $g$, that is, $\cN_\theta$ in our presentation, or the map from $g$ to $f$ (which would be roughly the inverse of $\cN_\theta$). The answer to this depends on how much \emph{a priori} information we have on those maps. It is well-known in deep learning research that a smoothing map, corresponding to the information compression process, is, in general, easier to learn (since the parameterization of the map requires a smaller number of parameters) than its map. Therefore, we should choose to learn the smoother one, whether it is $\cN_\theta$ or its inverse.

%%%%%%%%%%%%%%%%%%%%%%%%%%%%%%%%%%%%%%%%%%%%%%%%%%%%%%%%%%%%%%%%%%
%%%%%%%%%%%%%%%%%%%%%%%%%%%%%%%%%%%%%%%%%%%%%%%%%%%%%%%%%%%%%%%%%%
\section*{Acknowledgments}
%%%%%%%%%%%%%%%%%%%%%%%%%%%%%%%%%%%%%%%%%%%%%%%%%%%%%%%%%%%%%%%%%%
%%%%%%%%%%%%%%%%%%%%%%%%%%%%%%%%%%%%%%%%%%%%%%%%%%%%%%%%%%%%%%%%%%

\RED{We would like to thank the anonymous referees for their useful comments that helped us improve the quality of this work. We would like to thank Yan Cheng for his help on the numerical experiment in Section~\ref{SUBSEC:Num Acous}. This work is partially supported by the National Science Foundation through grants DMS-1913309, DMS-1937254, and EAR-2000850.} 

\appendix
%%%%%%%%%%%%%%%%%%%%%%%%%%%%%%%%%%%%%%%%%%%%%%%%%%%%%%%%%%%%%%%%%%
%%%%%%%%%%%%%%%%%%%%%%%%%%%%%%%%%%%%%%%%%%%%%%%%%%%%%%%%%%%%%%%%%%
\section{Details on computational implementation}
\label{SEC:Details}
%%%%%%%%%%%%%%%%%%%%%%%%%%%%%%%%%%%%%%%%%%%%%%%%%%%%%%%%%%%%%%%%%%
%%%%%%%%%%%%%%%%%%%%%%%%%%%%%%%%%%%%%%%%%%%%%%%%%%%%%%%%%%%%%%%%%%

We discuss in this appendix some technical details on the computational implementation of the algorithms.

%%%%%%%%%%%%%%%%%%%%%%%%%%%%%%%%%%%%%%%%%%%%%%%%%%%%%%%%%%%%%%%%%%
\subsection{Learning with polynomial models}
\label{SUBSEC:Polynomial}
%%%%%%%%%%%%%%%%%%%%%%%%%%%%%%%%%%%%%%%%%%%%%%%%%%%%%%%%%%%%%%%%%%

As we have emphasized, our main objective is not to learn the exact relation between the coefficients but only a reasonable approximation between them to improve the joint inversion process. Therefore, we are interested in learning the relation with low-order polynomials that have a small number of degree of freedom (which in turn require a small amount of training data points). 

The polynomial model~\eqref{EQ:Poly} is constructed in a way that the different components of the output are completely independent of each other. In other words, we did not impose any constraint on the coefficient $\theta_{1,\balpha}, \cdots, \theta_{K,\balpha}$ even though some constraints should be enforced. We do, however, need to force the predicted coefficient $g$ to be physically relevant. For instance, for the two joint inversion problems we considered, we impose the constraint that $g$ is positive and is bounded from above and below by some known constants $\underline{\alpha}$ and $\overline{\alpha}$, that is
\begin{equation}
    \underline{\alpha}\le \sum_{k=1}^K\Big( \sum_{|\balpha|\le n}\theta_{k,{\balpha}} P_{\balpha}(\wh\bff)\Big) \varphi_k(\bx) \le \overline{\alpha}\,.
\end{equation}
This is a linear constraint on the components of the polynomial coefficient $\theta$.

\paragraph{Generalization to piecewise smooth case.} While the generalized Fourier parameterization~\eqref{EQ:Fourier Model} is a global method, we can extend it to the case of piecewise smooth coefficients by constructing the basis on subdomains of $\Omega$. The polynomial relation~\eqref{EQ:Poly} is now dependent on the subdomains, that is, 
\begin{equation}\label{EQ:Poly Piecewise}
    \wh \bg_{\zeta}:= \Big(\sum_{|\balpha|\le n}\theta_{1,{\balpha},\zeta} P_{\balpha}(\wh\bff_{\zeta}), \cdots, \sum_{|\balpha|\le n}\theta_{j,{\balpha},\zeta} P_{\balpha}(\wh\bff_{\zeta}), \cdots, \sum_{|\balpha|\le n}\theta_{K, {\balpha},\zeta} P_{\balpha}(\wh\bff_{\zeta})\Big)\,,
\end{equation}
where the subscript $\zeta$ is used to highlight that the corresponding quantity is supported on the $\zeta$-th subdomain. When the polynomial relation is learned directly from the coefficient pair~\eqref{EQ:Data Hist}, the learning process on the subdomains has the same amount of data and can be done in parallel. In our model-consistent learning framework, however, this poses additional challenges as it increases the number of parameters in the representation of the operator unless we impose additional constraints on the relations on different subdomains (for instance, by asking them to be sufficiently close to each other). 

%%%%%%%%%%%%%%%%%%%%%%%%%%%%%%%%%%%%%%%%%%%%%%%%%%%%%%%%%%%%%%%%%%
\subsection{Network representation and training}
\label{SUBSEC:Network}
%%%%%%%%%%%%%%%%%%%%%%%%%%%%%%%%%%%%%%%%%%%%%%%%%%%%%%%%%%%%%%%%%%

To illustrate the feasibility of our approach for problems where the relation between the coefficients is more complicated than what we have discussed, we implemented a neural network representation to benchmark the polynomial representation. 

In our implementation of the autoencoder architecture, each substructure (that is, the encoder, the decoder, or the predictor) is a standard fully connected network of $L+1$ layers. To be precise, let $n^\ell$ be the width of layer $\ell$ ($0\le \ell\le L$), and $\Theta^\ell({\bx}) : \mathbb{R}^{n^{\ell-1}} \rightarrow \mathbb{R}^{n^\ell}$ the standard linear transform
\[
\Theta^\ell({\bx}) = {\bf W}^{\ell-1} {\bx} + {\bf b}^{\ell-1}\,,
\]
where ${\bf W}^{\ell-1} \in \mathbb{R}^{n^\ell \times n^{\ell-1} }$ and ${\bf b}^{\ell-1}\in\mathbb{R}^{n^\ell}$ are respectively the weight and bias for layer $\ell$. Then, the network output $\by$ corresponding to input $\bx$ is described by the following iterative propagation scheme
\begin{equation}\label{nn_traing}
\begin{array}{ll}
P^{0} &= {\bx}, \\
P^{\ell} &= \phi(\Theta^{\ell}({P^{\ell-1}})), \ \ \ell = 1,2,\cdots, L,\\
\by &=\Theta^{L+1}(P^{L})\,,
\end{array}
\end{equation}
where $\phi$ is the nonlinear activation function and $\phi(\bx)$ for any $d$-vector is understood as the $d$-vector $(\phi(x_1),\cdots, \phi(x_d))$. We denote by $\theta$ the set of parameters $\{\bW^\ell, \bb^\ell\}_{\ell=1}^{L+1}$ for the network.  

The training of the autoencoder network is done by minimizing a loss function $\cL(\theta)$ that combines the loss for the encoder-decoder substructure and another loss term for the encoder-predictor substructure. More precisely, we minimize
\begin{multline}\label{EQ:Loss ED}
\cL(\theta):=\dfrac{1}{2N|H|}\dsum_{h\in H}\dsum_{k=1}^N\|\cA_h(f_k, \wt {\cN_\theta^{ep}}(f_k))-u_{h,k}\|_Y^2\\ +\dfrac{1}{2N}\sum_{k=1}^N \|\cF(g_k)-\cN_\theta^{ep}(\cF(f_k))\|_{\wt X}^2 +\dfrac{1}{2N}\sum_{k=1}^N \|\cF(f_k)-\cN_\theta^{ed}(\cF(f_k))\|_{\wt X}^2\,,
\end{multline}
where the operators $\cN_\theta^{ep}=P_\theta\circ E_\theta$ and $\cN_\theta^{ed}=D_\theta\circ E_\theta$ are respectively the encoder-predictor and encoder-decoder networks, and $\wt {\cN_\theta^{ep}}:=\cF^{-1}\circ \cN_\theta^{ep} \circ \cF$ is the relation between the coefficients $f$ and $g$ in the physical space $X$. The first two terms in the loss functions are for the encoder-predictor substructure, and the third term is for the encoder-decoder substructure. The first term in the loss function is to enforce the model-consistency requirement.

%%%%%%%%%%%%%%%%%%%%%%%%%%%%%%%%%%%%%%%%%%%%%%%%%%%%%%%%%%%%%%%%%%
%%%%%%%%%%%%%%%%%%%%%%%%%%%%%%%%%%%%%%%%%%%%%%%%%%%%%%%%%%%%%%%%%%
{\small
%\bibliography{C:/RenGDrive/Publications/Bibliography/RH-BIB}
%\bibliography{BIB-REN,BIB-JointInversion}
%\bibliographystyle{siam}

}
%%%%%%%%%%%%%%%%%%%%%%%%%%%%%%%%%%%%%%%%%%%%%%%%%%%%%%%%%%%%%%%%%%
%%%%%%%%%%%%%%%%%%%%%%%%%%%%%%%%%%%%%%%%%%%%%%%%%%%%%%%%%%%%%%%%%%

\end{document}